\documentclass[10pt]{article}
\usepackage{}
\usepackage{amsmath}
\allowdisplaybreaks[4]
\usepackage{amsthm}
\usepackage{amsfonts}
\usepackage{fancybox}
\usepackage{epsfig}
\usepackage{epsf}
\usepackage{amssymb}
\usepackage{epic,eepic}
\usepackage{latexsym,bm}
\usepackage{graphicx}
\usepackage{multirow}
\usepackage{float}
\usepackage{color}
\usepackage[top=0.8in, bottom=0.8in, left=0.8in, right=0.8in, dvips, letterpaper]{geometry}
\usepackage{algorithm}
\usepackage{algorithmic}
\usepackage{cite}
\usepackage{boxedminipage}

\usepackage{url}
\usepackage{lineno} 
\usepackage{makecell}

\usepackage{subcaption}
\usepackage{booktabs} 
\usepackage{caption}  

\newcommand\argmin{\mathop{\textrm{argmin}}}
\newcommand\crule[1][5cm]{%
  \par
  \nointerlineskip
  \centerline{\hbox to #1{\hrulefill}}%
  \nointerlineskip}

\numberwithin{equation}{section}
\numberwithin{algorithm}{section}
\newtheorem{theorem}{{\sc Theorem}}[section]
\newtheorem{lemma}{{\sc Lemma}}[section]
\newtheorem{corollary}[theorem]{Corollary}
\newtheorem{remark}{Remark}[section]

\newtheorem{definition}{{\sc Definition}}[section]

\newcommand{\R}{\mathbb{R}}

\newcommand{\be}{\begin{equation}}
\newcommand{\ee}{\end{equation}}
\newcommand{\bee}{\begin{equation*}}
\newcommand{\eee}{\end{equation*}}
\newcommand{\bea}{\begin{eqnarray}}
\newcommand{\eea}{\end{eqnarray}}
\newcommand{\beaa}{\begin{eqnarray*}}
\newcommand{\eeaa}{\end{eqnarray*}}

\title{
	\textbf{A sequential regularized piecewise affine algorithm for nonconvex nonsmooth multicomposite optimization in RNN training}}

\author{Lingzi Jin\thanks{{ling-zi.jin@connect.polyu.hk}, Department of Applied Mathematics, Hong Kong Polytechnic University, Kowloon, Hong Kong.  }
        \and
        Xiao Wang\thanks{{wangx936@mail.sysu.edu.cn}, School of Computer Science and Engineering, Sun Yat-sen University, Guangzhou, China. 
		}
       	\and
       	Xiaojun Chen\thanks{{maxjchen@polyu.edu.hk}, Department of Applied Mathematics, Hong Kong Polytechnic University, Kowloon, Hong Kong. 
		}
}

\begin{document}

\maketitle
\begin{abstract}
  This paper focuses on a class of nonconvex nonsmooth multicomposite optimization problems for training RNNs (Recurrent Neural Networks). 
  We first establish easily verifiable conditions under which an approximate first-order d-stationary point of the problem is guaranteed to be an approximate second-order d-stationary point.  
  Subsequently, we propose a sequential regularized piecewise affine algorithm, SRPA, that minimizes a regularized piecewise affine function at each iteration, where the idea of active-set strategy is incorporated to reduce the computational cost per iteration. 
  Leveraging the aforementioned conditions for second-order d-stationarity, we establish the global convergence of SRPA to second-order d-stationary points and the complexity bound of $ \mathcal{O}(\epsilon^{-2}) $ for obtaining an $ \epsilon $-approximate second-order d-stationary point. 
 Finally, numerical experiments for training RNNs on synthetic and real-world datasets demonstrate the promising performance of SRPA compared with state-of-the-art algorithms.
\vspace{0.8cm}


\vspace{0.5cm}


\end{abstract}

\section{Introduction} \label{sec1}
Consider 
\begin{align}
	\label{l1pen}
	\tag{P1}
	\min_{z \in \R^{\bar{N}} } ~ \Theta(z):= F(z) + \sum_{\ell=1}^{L} \beta_{\ell} \| u_{\ell}  - \psi_{\ell-1} ( \theta, u_{1}, \dots, u_{\ell-1} ) \|_{1},
\end{align}
 where $ z = ( \theta^{\top}, u_{1}^{\top}, \dots, u_{L}^{\top} )^{\top} \in \R^{\bar{N}} $ with $ \theta \in \R^{n}, u_{\ell} \in \R^{N_{\ell}} $ for all $ \ell \in [L] := \{ 1,\dots,L\} $, $ ( u_{1}, \dots, u_{\ell-1} ) $ is an empty placeholder when $ \ell=1 $,
 \begin{align*}
	& F(z):= g(u_{1},\dots,u_{L}) + \lambda \|\theta \|^{2} \text{ with } \lambda >0, \, g : \R^{ \bar{N}_{L} } \rightarrow \R_{+}, \\
	& \beta_{\ell} >0 \text{ and } \psi_{\ell - 1} : \R^{n+ \bar{N}_{\ell-1} } \rightarrow \R^{N_{\ell}} \text{ for all } \ell =1,\dots,L,
 \end{align*}
 and $ \bar{N}_{0}:=0 , \, \bar{N}_{\ell}:= \sum_{j=1}^{\ell} N_{j} $ for all $ \ell \in [L] $.
 Here, $g$ and $ \{ \psi_{\ell-1}, \ell=1,\dots,L \} $ are locally Lipschitz continuous, twice directionally differentiable (Definition \ref{def:dd}) but possibly nonsmooth. 
Problem \eqref{l1pen} covers a wide range of applications in machine learning, including the training process of RNNs, where $ \theta $ refers to the network parameter and $ \{ u_{\ell}, \, \ell \in [L] \} $ are the auxiliary variables for the $ \ell $th layer of the network \cite{JWC2025,LRP-LLC}.

We aim to propose an algorithm for finding a second-order d-stationary point of \eqref{l1pen} with theoretical complexity guarantees. 
In addition to local Lipschitz continuity and twice directional differentiability of $g$ and $ \{ \psi_{\ell-1} \} $, we further assume that $g$ is smooth (i.e., differentiable with locally Lipschitz continuous gradient) and $ \{ \psi_{\ell-1,i} , \ell \in [L], i \in [N_{\ell}] \} $ are either smooth or piecewise affine (Definition \ref{def:PA}). Then, \eqref{l1pen} can be reorganized as 
\begin{equation}
	\tag{P1$_{r}$}
	\label{l1pen_r}
	\min_{z}~ \Theta(z)= F(z) + H(z) + \sum_{j \in \mathcal{J}} f_{j}(z),
\end{equation}
where $ H: \R^{\bar{N}} \rightarrow \R_{+} $ is continuous and piecewise affine, $ f_{j}(z):=| c_{j}(z) | $ with a smooth $ c_{j}:\R^{\bar{N}} \rightarrow \R $ for any $ j \in \mathcal{J}:= \{ 1,\dots,J \} $, $J$ is a nonnegative integer with $ J \leq \bar{N}_{L} $.

There exist several algorithms with asymptotic convergence to d-stationary points of reformulated \eqref{l1pen_r}. 
According to \cite[Chapter 2]{Scholtes2012}, any piecewise affine function can be represented as the difference of the maxima of two sets of affine functions. Hence, reformulating $ \Theta $ as the difference of a convex function and the maximum of a finite number of smooth convex functions, and applying Algorithm 1 of \cite{PRA2017B-sta} on \eqref{l1pen_r}, we can obtain a sequence converging to a d-stationary point. 
Asymptotic convergence to d-stationary points is also established in \cite{ALT2025}, which requires reformulating $ \Theta $ as a sum of smooth functions, proximable functions and weakly concave functions.
Furthermore, since the maximum of affine functions can be represented as the composition of an absolute value function and affine functions, $\Theta$ can be reformulated as a function composed by absolute value functions and smooth functions.
For minimizing such abs-normal functions, Griewank \cite{Griewank2013} proposes a successive piecewise linearization algorithm that generates a sequence converging to a d-stationary point when the subproblems can be solved exactly.
If the subproblem is solved inexactly in the sense that increments are merely
Clarke stationary points of the subproblem, then the cluster points are only guaranteed to be Clarke stationary \cite{FWG2019}.  
Local linear and quadratic convergence rates of Griewank's algorithm are established under the assumptions of certain CQ conditions and second-order sufficient conditions, and linear growth conditions at the cluster points, respectively \cite{GW2019}. 
 
Without complicated reformulations, algorithms for finding a first-order stationary point of \eqref{l1pen_r} are mostly restricted to the case $ H \equiv 0 $. In this case, \eqref{l1pen_r} reduces to $ \min_{z} h(c(z)) $ with convex $ h $ and smooth $ c $, where d-stationarity is equivalent to both limiting stationarity and Clarke stationarity. 
Seminal works on asymptotic convergence to and evaluation complexities for first-order stationary points include \cite{Burke1985,Fletcher1981PMO,Fletcher1982,powell1983general,Y141} and \cite{CGT2011,DP2019,GST2021,Proxlinear2021,NN2025}, respectively. Notably, the residual measure in \cite{CGT2011,DP2019,GST2021,Proxlinear2021,NN2025} are different. 
\cite{CGT2011,GST2021,NN2025} measure the optimality gap by the maximum reduction attained by the linear model of the objective function over a trust region, 
while \cite{DP2019,Proxlinear2021} uses the length of trial steps.
Moreover, in \cite{NN2025}, Nabou and Necoara consider the complexity for finding an approximate first-order stationary point defined by the minimum norm of subgradients of neighboring points. 
The results on asymptotic convergence are extended to prox-regular $ h $ by Lewis and Wright \cite{LW2016} under certain conditions at accumulation points.
In an adaptive Lagrangian-based framework, Hallak and Teboulle \cite{HT2023} further demonstrate that for any $h$ with an easily computable proximal operator, the asymptotic convergence and complexity bounds remain unchanged, in the sense of first-order stationarity and residual measure defined via limiting subdifferential.

Extensive research on second-order conditions for nonsmooth optimization problems, including \eqref{l1pen} and \eqref{l1pen_r}, has been conducted since the 1970s \cite[Chapter 13]{rockafellar2009variational}. 
Among them, \cite[Proposition 9.4.2]{CP-book} offers second-order conditions for the multicomposite form of \eqref{l1pen} with $ L=1 $ and twice semidifferentiable component functions through \eqref{l1pen}.
In \cite{JWC2025}, the second-order necessary conditions for \eqref{l1pen} and its composite form are generalized to the case where $ L \geq 1 $ and the component functions are merely twice directionally differentiable.  
Furthermore, there is substantial research on second-order optimality conditions of  \eqref{l1pen_r} with $ H \equiv 0 $ \cite{Burke1987,Ioffe1979b,Penot1994a,PR1995}.
Despite the existence of abundant forms of second-order optimality conditions for \eqref{l1pen_r}, algorithms for computing its second-order stationary points are limited. 
Cartis, Gould and Toint \cite{CGT2022} study the evaluation complexities of minimizing $ ( f(x)+
h(c(x)) ) $ with smooth $ (f,c) $ and nonsmooth $ h $ to reach the high-order approximate minimizers defined by the maximum reduction of the high-order Taylor expansions of the objective function within a trust region. 
However, their approach requires computing global minimizers of high-order Taylor series within Euclidean balls. 
Additionally, \cite{CGT2022} requires $h$ to be Lipschitz continuous, subadditive, and zero at zero. Consequently, the framework excludes practical instances, such as $ \min_{x_{1},x_{2} } \, ( f(x_{1},x_{2})  + |x_{1} - [x_{2}]_{+} | ) $ with smooth $ f $.

\subsection{Contribution} 

First, we identify easy verifiable conditions under which an approximate first-order d-stationary point of \eqref{l1pen} is an approximate second-order d-stationary point of \eqref{l1pen}. 
	This result extends Remark 3.12 of \cite{JWC2025} and enables a first-order algorithm to attain an $\epsilon$-approximate second-order stationary point for \eqref{l1pen}.

Second, we propose a sequential regularized piecewise affine algorithm, SRPA, and consider the complexity bounds for finding $\epsilon$-approximate first-order and second-order d-stationary points of \eqref{l1pen_r}.
	In each iteration, SRPA constructs a piecewise affine approximate function and 
	minimizes the quadratically regularized piecewise affine function within a region defined by the previous iterate.  
    The global convergence and convergence rate of SRPA are established under certain conditions, which provides an upper bound $ \mathcal{O}(\epsilon^{-2}) $ on the iteration complexity for finding an $ \epsilon $-approximate second-order d-stationary point of \eqref{l1pen_r}.
	Specifically, our contributions regarding the algorithmic design and theoretical analysis are detailed as follows. 
	\begin{itemize}
		\item Algorithmic design. Inspired by \cite{CGT2022}, we handle the piecewise smooth $ \Theta $ by constructing regularized piecewise affine approximations. Furthermore, by integrating the active-set strategy from DCA \cite{PRA2017B-sta} and restricting the minimization to a neighborhood defined by the previous iterate, we ensure that each subproblem decomposes into a small number of strongly convex problems, thus reducing the computational cost per iteration.

		\item Theoretical analysis. Leveraging the structure of \eqref{l1pen_r}, we establish the complexity bounds required to achieve $ \epsilon $-approximate first-order and second-order d-stationarity.  
		Notably, these bounds are derived under conditions significantly weaker than second-order sufficient conditions or linear growth conditions used in \cite{GW2019}. 
		Additionally, the complexity analysis justifies the use of a relaxation parameter $ \delta $ in active-set construction.
	\end{itemize}

Finally, numerical experiments compared with state-of-the-art algorithms validate the effectiveness of SRPA for RNN training on synthetic, financial, and audio datasets.

\subsection{Organization} 
In Subsection \ref{subsec:1.3}, we introduce definitions of $ \epsilon $-approximate first-order and second-order d-stationary points of \eqref{l1pen}, and some concepts for the subsequent analysis. 
	Section \ref{sec:1d_2d} derives the relationship between approximate first-order and second-order d-stationary points of \eqref{l1pen}.   
	SRPA (Algorithm \ref{alg2}) is proposed in Section \ref{sec:r-SRPA}, where its global convergence to second-order d-stationary points and complexity bound for finding an $ \epsilon $-approximate second-order d-stationary point are established in Subsections \ref{subsec:1-d.1-r} and \ref{subsec:1-d.2-r}, respectively. 
    Numerical experiments are conducted in Section \ref{sec:4.4}.

\subsection{Notation and preliminaries}
\label{subsec:1.3}
For any differentiable vector function $f$, denote its Jacobian matrix as $J_{f}$ and define $ \nabla f (x) = J_{f}(x)^{\top} $. 
For brevity, we denote $ [m]:= \{ 1, \dots,m \} $ for any positive integer $m$, define $ c(z) :=( c_{1}(z), \dots, c_{J}(z) )^{\top} \in \R^{J} $ for all $z \in \R^{\bar{N}}  $, and use the following notations. 
 \begin{itemize}
	\item $ \mathbf{u}_{\ell} := (u_{1}^{\top}, \dots, u_{\ell}^{\top} )^{\top} $ for all $ \ell \in [L] $, $ u:= \mathbf{u}_{L} $ and $ \mathbf{u}_{0} $ is an empty placeholder. 
	\item For any direction $ d \in \R^{\bar{N}} $, break $d$ according to the blocks of $z$ as $ d= ( d_{\theta}^{\top} , d_{u_{1}}^{\top}, \dots, d_{u_{L}}^{\top} )^{\top} $ with $ d_{\theta} \in \R^{n}, d_{ u_{\ell} } \in \R^{N_{\ell}} $ for all $ \ell \in [L] $. And similarly define $ d_{\mathbf{u}_{\ell}} := (d_{u_{1}}^{\top}, \dots, d_{u_{\ell}}^{\top} )^{\top} $ for all $ \ell \in [L] $, $ d_{u}:= d_{\mathbf{u}_{L}} $ and $ d_{\mathbf{u}_{0}} $ as an empty placeholder. 
 \end{itemize}

The accumulative summation and multiplication are denoted by $ \sum $ and $\prod$, respectively.
For any sequence $ \{ a_{j} \geq 0, j\in \mathbb{Z}_{+}  \} $ and any $ j_{1},j_{2} \in \mathbb{Z}_{+}  $ with $ j_{1}>j_{2} $, denote $ \sum_{j=j_{1}}^{j_{2}} a_{j} :=0  $ and $ \prod_{j=j_{1}}^{j_{2}} a_{j} :=1  $.
For any vector sequence $ \{ u_{j}  , j\in \mathbb{Z}_{+}  \} $ and any $ j_{1},j_{2} \in \mathbb{Z}_{+}  $ with $ j_{1}>j_{2} $, denote $ (u_{j_{1}}, \dots, u_{j_{2}}  ) $ as an empty placeholder.
For any vector $a$ and positive integer $i$, $ [a]_{i} $ refers to the $i$th component of $a$.
We use $\|\cdot\|$ and $\| \cdot\|_{1}$ to denote the $\ell_2$-norm and $\ell_1$-norm of a vector or a matrix, respectively.
For any matrix $A$, we denote $\| A \|_{\rm F}$ as the Frobenius norm of $A$.
For any two sets $A,B \subseteq \R^{m} $, denote 
$ A \backslash B= \{ a \in A \mid a \notin  B \} $.
For any $ \bar{z} \in \R^{  \bar{N}  } $ and any $ r \in \R_{+} $, denote $\mathbb{B}(\bar{z};r):= \{ z\in \R^{  \bar{N}  } \mid \|  z - \bar{z} \| \leq r \} $.
For any $ m \in \mathbb{Z}_{++}, \gamma \in \R $ and any function $ f:\R^{m} \rightarrow \R\cup \{ + \infty \} $, the level set is defined as $ lev_{\leq \gamma} f:=\{ x \in \R^{m} \mid f(x) \leq \gamma \} $.
For any set $ A \subseteq \R^{m} $, denote the interior of $A$ as $ int(A) $. 
For any point $ x\in \R^{m} $ and any set $ A \subseteq \R^{m} $, denote the distance between $ x $ and $ A $ as $ \mathbf{d}(x,A):= \inf_{y\in A} \|x-y\| $.

Denote the global solution set of \eqref{l1pen} as $ \mathcal{S}_{1}:= \argmin_{ z \in \R^{\bar{N}} } \Theta(z) $. 
According to \cite[Lemma 2.2]{JWC2025}, $ \mathcal{S}_{1} $ is nonempty and compact. Hence, $  \Theta^{*}:= \min_{ z \in \R^{\bar{N}} } \Theta(z) $ is finite.
Before introducing the definitions of first-order and second-order d-stationarity, we present first-order and second-order directional derivatives as follows.
\begin{definition}[twice directional differentiability]  \label{def:dd}
	Given an open subset $ \mathcal{O} $ of $ \R^{n} $ and a scalar-valued function $ f : \mathcal{O}  \rightarrow \R $. The directional derivative of $ f $ at a point $ x\in    \mathcal{O}   $ along a direction $ d\in \R^{n} $ is defined as
	\begin{align}
		f^{\prime}(x;d):=\lim_{ \tau \downarrow 0}  \frac{f(x+\tau d ) - f(x) }{\tau}, \label{eq:1-dd}
	\end{align} if the limit exists.
	The function $ f $ is directionally differentiable at $ x $, if the limit \eqref{eq:1-dd} exists for all $ d \in \R^{n}$.
    The second-order directional derivative of $ f $ at a point $ x\in    \mathcal{O}   $ along a direction $ d\in \R^{n} $ is defined as
    \begin{align}
        f^{(2)} ( x; d ) := \lim_{ \tau \downarrow 0}  \frac{f(x+\tau d ) - f(x) - \tau f^{\prime}(x;d)  }{\tau^{2} /2 }, \label{eq:2-dd}
    \end{align}
    if the limit and the one for \eqref{eq:1-dd} exist. The function $ f $ is twice directionally differentiable at $ x $, if the limits \eqref{eq:1-dd} and \eqref{eq:2-dd} exist for all $ d \in \R^{n}$.
\\ \indent
    For a vector-valued function $ f: \mathcal{O} \rightarrow \R^{m} $ with component functions $ \{ f_{i}: \mathcal{O} \rightarrow \R, i\in [m] \} $, the directional derivative $ f^{\prime}( x;d ) $ is defined as
    $$ f^{\prime}( x;d ):=  ( f_{1}^{\prime}( x;d ) , \dots, f_{m}^{\prime}( x;d ) )^{\top}, $$
    if $ f_{i}^{\prime}( x;d ) , i\in [m]$ exist.
    Furthermore, if $ f_{i}^{(2)}( x;d ) , i\in [m]$ exist, then the second-order directional derivative  $ f^{(2)}( x;d ) $ is defined as
    $$ f^{(2)}( x;d ):=  ( f_{1}^{(2)}( x;d ) , \dots, f_{m}^{(2)}( x;d ) )^{\top} . $$
    Function $f$ is (twice) directionally differentiable at $x$, if all of its component functions are (twice) directionally differentiable at $x$.
\end{definition} 

Since $ g $ and $ \{ \psi_{\ell-1}, \ell=1,\dots,L \} $ are locally Lipschitz continuous and twice directionally differentiable, it follows from  Propositions 3.4 and 3.5 of \cite{JWC2025} that $ \Theta $ is twice directionally differentiable on $ \R^{\bar{N}} $. 
Thus, we can define first-order and second-order d-stationary points of \eqref{l1pen}, which encompasses \eqref{l1pen_r}, as follows. 
\begin{definition}[(second-order) d-stationary point]
	\label{def:D-SD}
	We call $ z \in \R^{\bar{N}} $ as a  d-stationary point of \eqref{l1pen}, if $ \Theta^{\prime}(z;d)\geq 0$ for all $ d $. 
	We call $ z \in \R^{\bar{N}} $ as a second-order d-stationary point of \eqref{l1pen}, if $ \Theta^{\prime}(z;d)\geq 0 $ for all $ d $ and $\Theta^{(2)}(z;d)\geq 0 $ for all $ d $ satisfying $\Theta^{\prime}(z;d) =  0 $.
	Denote the sets of d-stationary and second-order d-stationary points of \eqref{l1pen} as
	\begin{align*}
		&\mathcal{D}_{1}:= \{ z\in \R^{\bar{N}} \mid  \Theta^{\prime}(z;d)\geq 0 , \, \forall d\}, \\
	 	&  \mathcal{SD}_{1}:= \{ z\in \mathcal{D}_{1} \mid  \Theta^{(2)}(z;d)\geq 0 , \, \text{for all } d \text{ satisfying } \Theta^{\prime}(z;d) =  0 \} ,
	\end{align*}
	respectively.
\end{definition} 

According to \cite[Lemma 3.1]{JWC2025}, the second-order d-stationarity is necessary for the local optimality of \eqref{l1pen}, which implies that $ \mathcal{D}_{1}  \supseteq  \mathcal{SD}_{1}  \supseteq   \mathcal{S}_{1} \neq  \emptyset $. 
Correspondingly, the $ \epsilon $-approximate d-stationarity and $ \epsilon $-approximate second-order d-stationarity of \eqref{l1pen}, which encompasses \eqref{l1pen_r}, are defined as follows.
\begin{definition}[$ \epsilon $-approximate d-stationary point]
	\label{def:eps-1d}
	For any $ \epsilon>0 $, we call $ z \in \R^{\bar{N}} $ as an $ \epsilon $-approximate d-stationary point for \eqref{l1pen} if there exists $ y \in   \mathbb{B}(z;\epsilon)  $ such that
	\begin{align*}
		\Theta^{\prime}(y;d) \geq -  \epsilon, \, \forall d \text{ with } \|d\|=1.
	\end{align*} 
\end{definition}

When $  \epsilon =0 $, it implies that $ \Theta^{\prime}(z;d) \geq 0 $ for all $d$ with $ \|d\|=1 $. Together with the positive homogeneity of directional derivative in direction, it further implies that $ \Theta^{\prime}(z;d) \geq 0 $ for all $d$, i.e., $z$ is a d-stationary point of \eqref{l1pen}.

\begin{definition}[$ \epsilon $-approximate second-order d-stationary point]
	\label{def:eps-2d}
	For any $ \epsilon>0 $, we call $ z \in \R^{\bar{N}} $ as an $ \epsilon $-approximate second-order  d-stationary point for \eqref{l1pen} if there exists $ y \in   \mathbb{B}(z;\epsilon)  $ such that
	\begin{align*}
		&  \Theta^{\prime}(y;d) \geq -  \epsilon, \, \forall d \text{ with } \|d\|=1, \\
		&  \Theta^{(2)}(y;d) \geq -  \epsilon, \, \forall d \text{ with } \|d\|=1 \text{ and } \Theta^{\prime}(y;d) = -  \epsilon .
	\end{align*} 
\end{definition}

When $  \epsilon =0 $, it implies that $ \Theta^{\prime}(z;d) \geq 0 $ for all $d$ as discussed earlier. 
Similarly, it indicates that $ \Theta^{(2)}(y;d) \geq 0 $ for all $d$ with $ \|d\|=1$ and $ \Theta^{\prime}(y;d) = 0 $. 
Together with the positive homogeneity of $ \Theta^{\prime}(y;d) $ in $d$ and the positive homogeneity of degree $2$ of $ \Theta^{(2)}(y;d) $ in $d$, it further implies that $ \Theta^{(2)}(y;d) \geq 0 $ for all $d$ with $  \Theta^{\prime}(y;d) = 0 $, i.e., $z$ is a second-order d-stationary point of \eqref{l1pen}.

The following is the definition of piecewise affine functions. 
\begin{definition}[piecewise affine function, Definition 2.47 of \cite{rockafellar2009variational}]
    \label{def:PA}
    A function $ f:\mathcal{D} \subseteq \R^{n} \rightarrow \R $ is piecewise affine (PA),
	if $ \mathcal{D} $ can be represented as the union of finitely many polyhedral sets, relative to each of which $ f(x) $ is given by an expression of the form $ a^{\top} x + b $ for some $ a \in \R^{n},b\in \R $.
\end{definition} 
For continuous piecewise affine $ H $, there exists a finite number of polyhedra $  \{ P_{i} \}_{i=1}^{I}  $ and the corresponding $  \{ a_{i} \in \R^{\bar{N}},b_{i} \in \R \}_{i=1}^{I} $ such that $  \{ P_{i} \}_{i=1}^{I}  $ are nonempty and mutually distinct, $ \R^{\bar{N}} = \cup_{i \in [I]} P_{i} $ and relative to each $ P_{i} $, $ H(z)=a_{i}^{\top} z + b_{i} $.

\section{Relationship between first-order and second-order d-stationarity}
\label{sec:1d_2d}
For any $ \epsilon \in \R_{+} $, denote
\begin{equation}
	\label{def:eD_eSD}
	\begin{aligned}
	&\mathcal{D}_{1}^{\epsilon}:= \left\{  z \in \R^{\bar{N}} \mid \Theta^{\prime}(z;d) \geq -\epsilon , \forall d \text{ with } \|d\|=1 \right\}, \\
	&\mathcal{SD}_{1}^{\epsilon}:= \left\{  z \in \mathcal{D}_{1}^{\epsilon} \mid \Theta^{(2)}(z;d) \geq -\epsilon , \forall d \text{ with } \|d\|=1 \text{ and } \Theta^{\prime}(z;d) = -\epsilon \right\}.
\end{aligned}
\end{equation} 
Then we have $ \mathcal{D}_{1} = \mathcal{D}_{1}^{0} $ and $ \mathcal{SD}_{1} = \mathcal{SD}_{1}^{0}  $ from the positive homogeneity of $ \Theta^{\prime}(z;\cdot) $ and $ \Theta^{(2)}(z;\cdot) $. 
This section will analyze the relationship between $ \mathcal{D}_{1}^{\epsilon} $ and $ \mathcal{SD}_{1}^{\epsilon} $ for any $ \epsilon \in [0,1) $.

Firstly, it follows from \cite{JWC2025} that under certain conditions, any d-stationary point of \eqref{l1pen} is a second-order d-stationary point of \eqref{l1pen}.

\begin{corollary}
	\label{cor:D-SD}
    Set $ \{ \beta_{\ell}, \, \ell \in [L] \} $ satisfying 
	\begin{align} \label{eq:br-threshold-2}
				\beta_{\ell} >  K_{g}  \prod_{j=\ell+1}^{L} ( 1+ K_{j-1} ) , \forall \ell \in [L],
    \end{align}
	where $ K_g  $ and $ \{ K_{\ell} , \ell \in [L-1] \} $ are the positive constants satisfying  
	\begin{equation} 
		\label{eq:Lip}
		\begin{aligned}
			&| g^{\prime}(u;d_{u}) - g^{\prime}(u;\bar{d}_{u}) | \leq K_{g} \|d_{u}-\bar{d}_{u}  \|, \\
			& \| \psi_{\ell-1}^{\prime}( \theta, \mathbf{u}_{\ell-1} ; d_{\theta},  d_{ \mathbf{u}_{\ell-1} }  ) - \psi_{\ell-1}^{\prime}( \theta, \mathbf{u}_{\ell-1} ; \bar{d}_{\theta},  \bar{d}_{ \mathbf{u}_{\ell-1} }  ) \| 
			\\
				& ~
				\leq K_{\ell-1} \|  d_{ \mathbf{u}_{\ell-1} } - \bar{d}_{ \mathbf{u}_{\ell-1} }     \|,  \, \ell=2,\dots,L,
		\end{aligned}
	\end{equation}
	for any $ d,\bar{d} \in \R^{\bar{N}} $ with $ d_{\theta} = \bar{d}_{\theta} $, and any $z  \in lev_{\leq \Theta(z^{0}) } \Theta $ with 
	\begin{align}
		\label{def:z0}
		z^{0}:= ( \mathbf{0}, (u_{1}^{0})^{\top}, \dots, (u_{L}^{0})^{\top} )^{\top},~ 
		u_{\ell}^{0}:= \psi_{\ell-1} (\mathbf{0}, (u_{1}^{0})^{\top}, \dots,  u_{\ell-1}^{0} ), \, \forall \ell \in [L].
	\end{align} 
    Assuming that $g$ is convex, then we have $ \mathcal{D}_{1} \cap lev_{\leq \Theta(z^{0}) } \Theta \subseteq \mathcal{SD}_{1} $.
\end{corollary}
\begin{proof}
	For any $ z \in \mathcal{D}_{1} \cap lev_{\leq \Theta(z^{0}) } \Theta $, Remark 3.12 of \cite{JWC2025} implies that under \eqref{eq:br-threshold-2}, $ \Theta^{(2)}(z;d ) $ can be simplified as
	\begin{align}\label{eq:sim_T'2}
		\Theta^{(2)}(z;d )= F^{(2)}(z;d) + \sum_{\ell \in [L]} \beta_{\ell} \| \psi_{\ell-1}^{(2)} ( \theta, \mathbf{u}_{\ell-1} ; d_{\theta} , d_{\mathbf{u}_{\ell-1}} ) \|_{1}
	\end{align}
	for all $d$ with $ \Theta^{\prime}(z;d ) =0 $.
	Since $ g $ is convex and $ F(z)= g(u) + \lambda \|\theta \|^{2} $ for all $ z = ( \theta^{\top}, u^{\top} )^{\top} $, it follows from \eqref{eq:sim_T'2} that for all $d$ with $ \Theta^{\prime}(z;d ) =0 $,
	\begin{align*}
		\Theta^{(2)}(z;d ) \geq F^{(2)}(z;d) \geq 0,
	\end{align*}
	which completes the proof.
   
\end{proof}
  
However, under the same inequalities as \eqref{eq:br-threshold-2}, the convexity of $F$ and the condition that $ \Theta(z ) \leq \Theta(z^{0}) $ cannot guarantee that 
$ \mathcal{D}_{1}^{\epsilon} \cap lev_{\leq \Theta(z^{0}) } \Theta \subseteq \mathcal{SD}_{1}^{\epsilon} $ for $ \epsilon>0 $.
The key obstacle is to obtain  
\begin{align}
	\label{def:Feas0}
	 y \in \mathcal{F}_{0}:=\lbrace z \in \R^{\bar{N}} \mid u_{\ell}=\psi_{\ell-1}( \theta,  \mathbf{u}_{\ell-1} ), \ell \in [L] \rbrace
\end{align}
from $ y \in \mathcal{D}_{1}^{\epsilon} $.
Next we show that it can be overcome under a stronger version of \eqref{eq:br-threshold-2}.
\begin{lemma}\label{lem:eD1-F0}
    Let $ \{ \beta_{\ell}>0, \ell\in [L] \} $ satisfy
    \begin{align} \label{eq:br-threshold-3}
				\beta_{\ell} >  ( K_{g} +1 )   \prod_{j=\ell+1}^{L} ( 1+ K_{j-1} ) , \text{ for all } \ell \in [L],
    \end{align}
	where $ K_g  $ and $ \{ K_{\ell} , \ell \in [L-1] \} $ are the positive constants satisfying \eqref{eq:Lip}. For any $ \epsilon \in [0,1) $, if $z$ satisfies 
	\begin{align}
		\label{eq:con}
		\Theta^{\prime}(z;d) \geq -  \epsilon, \, \forall d \text{ with } \|d\|=1 , 
		\text{ and } \Theta(z) \leq   \Theta(z^{0}), 
	\end{align}
	then we have $ z \in \mathcal{F}_{0} $, where $ z^{0} $ and $ \mathcal{F}_{0} $ are defined in \eqref{def:z0} and \eqref{def:Feas0}, respectively.
\end{lemma}
\begin{proof}
    To show $ z \in \mathcal{F}_{0} $, we will prove $ u_{\ell} = \psi_{\ell-1} ( \theta, u_{1}, \dots, u_{\ell-1} ) $ in the order of $ \ell=L,\dots,1 $ similar to Lemma 3.8 of \cite{JWC2025}. For all $ \ell \in [L] $, denote 
	\begin{equation}
		\label{eq:I_+-0}
		\begin{aligned}
			&
			I_{+}^{\ell}(z):=\{ i \in [ N_{\ell}] \mid [ u_{\ell} - \psi_{\ell-1}( \theta, \mathbf{u}_{\ell-1} ) ]_{i}  >0 \}, \\
			&
			I_{-}^{\ell}(z):=\{ i \in [ N_{\ell}] \mid [ u_{\ell} - \psi_{\ell-1}( \theta, \mathbf{u}_{\ell-1} ) ]_{i}  <0 \} ,  \\
			&
			I_{0}^{\ell}(z):= [N_{\ell}] \backslash \left( I_{+}^{\ell}(z)  \cup I_{-}^{\ell}(z) \right)  . 
		\end{aligned}
	\end{equation} 
	To show $ u_{L} = \psi_{L-1} ( \theta, u_{1}, \dots, u_{L-1} ) $ by contradiction, we use $ I_{+}^{L}(z) , I_{-}^{L}(z) $ and $I_{0}^{L}(z)$ defined in \eqref{eq:I_+-0}.  
	Suppose $ u_{L} \neq \psi_{L-1} ( \theta, u_{1}, \dots, u_{L-1} ) $, then $ I_{+}^{L}(z) \cup I_{-}^{L}(z)  \neq \emptyset $. Let $ \bar{z}:=( \bar{\theta}^{\top},  \bar{u}_{1}^{\top}, \dots, \bar{u}_{L}^{\top} )^{\top} $ with
    \begin{align*}
        \bar{\theta}:=\theta, ~
        \bar{u}_{1} := u_{1}, ~ \dots, ~ \bar{u}_{L-1} := u_{L-1}, 
        ~ \bar{u}_{L} : = \psi_{L-1} ( \theta, u_{1}, \dots, u_{L-1} ) \neq u_{L},
    \end{align*}
    and $ d:= \bar{z} -z= ( \mathbf{0}, \dots, \mathbf{0}, d_{u_{L}} ) \neq \mathbf{0} $. Then it follows from \cite[Proposition 3.4]{JWC2025} 
    that
    \begin{align}
        \Theta^{\prime}( z;d )
        &= F^{\prime}(u;d_{u} ) + \beta_{L} \left(
        \sum_{i \in I_{+}^{L}(z) } [ d_{u_{L}} ]_{i}  -  \sum_{i \in I_{-}^{L}(z) } [ d_{u_{L}} ]_{i} + \sum_{i \in I_{0}^{L}(z) } | [ d_{u_{L}} ]_{i} |
        \right) \notag \\
        &= g^{\prime}(u;d_{u} ) - \beta_{L}  \| d_{u_{L}} \|_{1} \label{eq:4-0.5new} \\
        &\leq ( K_{g}  - \beta_{L} )  \| d_{u_{L}} \|  , \notag \\
		&  <   - \| d_{u_{L}} \| , \notag  
	\end{align}
    where the two equalities use the definitions of $d$ and $ I_{+}^{L}(z), I_{-}^{L}(z), I_{0}^{L}(z)  $, the first inequality holds due to \eqref{eq:Lip} at $ \bar{d}=\mathbf{0} $ and $ \|\cdot\| \leq \|\cdot\|_{1} $, and the last inequality follows from \eqref{eq:br-threshold-3} and $ d_{u_{L}} \neq \mathbf{0} $. However, it follows from the positive homogeneity of $ \Theta^{\prime}(z;\cdot) $ and $ \Theta^{\prime}(z;\tilde{d}) \geq -  \epsilon $ for all $\tilde{d}$ with $ \|\tilde{d}\|=1 $ that 
	\begin{align*}
		\Theta^{\prime}(z;d) 
		= \Theta^{\prime}\left(z; \frac{d}{\|d\|} \right)  \|d\| 
		\geq -  \epsilon  \|d\| 
		= -  \epsilon  \|d_{u_{L}}\|
		> - \|d_{u_{L}}\|,
	\end{align*}
	where the last inequality comes from $ d_{u_{L}} \neq \mathbf{0} $ and $ \epsilon \in [0,1) $. It contradicts \eqref{eq:4-0.5new}, which implies that $ u_{L} = \psi_{L-1} ( \theta, u_{1}, \dots, u_{L-1} ) $.  

    For any $ \ell=L-1,\dots,1 $, we next show $ u_{\ell} = \psi_{\ell-1} ( \theta, u_{1}, \dots, u_{\ell-1} ) $ using $ I_{+}^{\ell}(z) $, $ I_{-}^{\ell}(z) $ and $I_{0}^{\ell}(z)$ in \eqref{eq:I_+-0}.  
    Suppose $ u_{\ell} \neq \psi_{\ell-1} ( \theta, u_{1}, \dots, u_{\ell-1} ) $, then $ I_{+}^{\ell}(z) \cup I_{-}^{\ell}(z) \neq \emptyset $. Let $ \bar{z}:=( \bar{\theta}^{\top},  \bar{u}_{1}^{\top}, \dots, \bar{u}_{L}^{\top} )^{\top} $ with
    \begin{align*}
        &\bar{\theta}:=\theta, ~
        \bar{u}_{1} := u_{1}, ~ \dots, ~ \bar{u}_{\ell-1} := u_{\ell-1}, 
        ~ \bar{u}_{\ell} = \psi_{\ell-1} ( \theta, u_{1}, \dots, u_{\ell-1} ) \neq u_{\ell}, \\
        & \bar{u}_{\ell+1}:= u_{\ell+1} + \psi_{\ell}^{\prime} ( \theta, u_{1}, \dots, u_{\ell}; \mathbf{0}, \dots, \mathbf{0}, \bar{u}_{\ell} - u_{\ell}  ), \\
        & \vdots \\
        & \bar{u}_{L}:= u_{L} + \psi_{L-1}^{\prime} ( \theta, u_{1}, \dots, u_{L-1}; \mathbf{0}, \dots, \mathbf{0}, \bar{u}_{\ell} - u_{\ell} , \dots , \bar{u}_{L-1} - u_{L-1}  ),
    \end{align*}
    and $ d:= \bar{z} -z= ( \mathbf{0}, \dots, \mathbf{0}, d_{ u_{\ell} }, \dots, d_{u_{L}} ) \neq \mathbf{0} $. Then it can be checked that
    \begin{equation}
        \label{eq:4-0new}
        \begin{aligned}
            d_{u_{\ell+1}} = \psi_{\ell}^{\prime} ( \theta,  \mathbf{u}_{\ell}; d_{\theta} ,   d_{ \mathbf{u}_{\ell} }  ), ~
            \dots, ~
            d_{u_{L}} = \psi_{L-1}^{\prime} ( \theta, \mathbf{u}_{L-1}; d_{\theta} ,  d_{ \mathbf{u}_{L-1} }  ).
        \end{aligned}
    \end{equation}
    Together with \cite[Proposition 3.4]{JWC2025}, 
    it implies that
    \begin{align}
        \Theta^{\prime}( z;d )
        &= F^{\prime}(u;d_{u} )  + \beta_{\ell}  \left(
        \sum_{i \in I_{+}^{\ell}(z) } [ d_{u_{\ell}} ]_{i}  -  \sum_{i \in I_{-}^{\ell}(z) } [ d_{u_{\ell}} ]_{i} + \sum_{i \in I_{0}^{\ell}(z) } | [ d_{u_{\ell}} ]_{i} | \right) \notag \\
        &= g^{\prime}(u;d_{u} )  - \beta_{\ell} \| d_{u_{\ell}} \|_{1} \notag \\
        &\leq K_{g} \|d_{u} \| - \beta_{\ell} \| d_{u_{\ell}} \|_{1}, \label{eq:4-1new}
    \end{align}
    where the two equalities use the definitions of $d$ and $ I_{+}^{\ell}(z), I_{-}^{\ell}(z), I_{0}^{\ell}(z)  $, the inequality holds due to \eqref{eq:Lip} at $ \bar{d}=\mathbf{0} $. Next we give an upper bound of $ \| d_{u} \| $ by estimating $ \{ d_{u_{j}} , j \in [L] \}$. Since $ d_{u_{1}}= \mathbf{0}, \dots, d_{u_{\ell-1}}= \mathbf{0} $, we only need to analyze $ \{ d_{u_{j}} , j = \ell,\dots,L \}$.
    For $j=\ell$, it follows from the definitions of $ d_{u_{\ell}} $ and $ \bar{u}_{\ell} $ that
    \begin{align}
        \label{eq:4-2new}
        \| d_{u_{\ell}} \| =  \| u_{\ell} - \psi_{\ell-1} ( \theta, u_{1}, \dots,u_{\ell-1} ) \| \neq  0.
    \end{align}
    For $ j=\ell+1 $, it follows from \eqref{eq:4-0new} that
    \begin{align}
        \label{eq:4-3new}
        \| d_{u_{\ell+1}} \|
        = \| \psi_{\ell}^{\prime} ( \theta, \mathbf{u}_{\ell}; d_{\theta} ,   d_{ \mathbf{u}_{\ell} }  ) \|
         \leq K_{\ell} (   \| d_{u_{1}}\| +  \dots +  \| d_{ u_{\ell} } \| )
         = K_{\ell} \| d_{ u_{\ell} } \|,
    \end{align}
    where the inequality holds due to \eqref{eq:Lip} at $ \bar{d}=\mathbf{0}, d_{\theta}=\mathbf{0} $,
    the last equality uses $ d_{u_{1}}= \mathbf{0}, \dots, d_{u_{\ell-1}}= \mathbf{0} $. And for $ j\geq \ell+2 $, assume that
    \begin{align}\label{eq:4-4new}
        \|d_{p}\| \leq \left( K_{p-1}\prod_{i=\ell+1}^{p-1} (1+ K_{i-1} ) \right) \|d_{u_{\ell}} \|
    \end{align}
    for all $ p=\ell+1, \dots, j-1 $.
    Then we can obtain that \eqref{eq:4-4new} also holds at $ p=j $:
    \begin{align*}
        \|d_{u_{j}}\|
        = \| \psi_{j-1}^{\prime} ( \theta, \mathbf{u}_{j-1}; d_{\theta} ,   d_{ \mathbf{u}_{j-1} }  ) \|  
        &\leq   K_{j-1} (   \| d_{u_{1}}\| +  \dots +  \| d_{ u_{j-1} } \| )  \\
        &=  K_{j-1} (   \| d_{u_{\ell}}\|+ \| d_{u_{\ell+1}}\| +  \dots +  \| d_{ u_{j-1} } \| ) \notag \\
        & \leq \left( K_{j-1}\prod_{i=\ell+1}^{j-1} (1+ K_{i-1} ) \right) \|d_{u_{\ell}} \| ,\notag
    \end{align*}
    where the first inequality holds due to \eqref{eq:Lip} at $ \bar{d}=\mathbf{0}, d_{\theta}=\mathbf{0} $,
    the second equality uses $  d_{u_{1}}= \mathbf{0}, \dots, d_{u_{\ell-1}}= \mathbf{0} $, and the second inequality uses \eqref{eq:4-4new} at $ p=\ell+1, \dots, j-1 $. By induction and \eqref{eq:4-3new}, it implies that \eqref{eq:4-4new} holds for all $ p=\ell+1,\dots,L $. 
	Based on these upper bounds for $ \{ \| d_{u_{j}} \| , j \in [L] \}$, we have 
	\begin{align}
		\|d_{u}\| 
		 &= \| ( \mathbf{0}, \dots, \mathbf{0}, d_{ u_{\ell} }, \dots, d_{u_{L}} ) \| \notag \\
		&\leq \|d_{u_{\ell}} \| + \|d_{u_{\ell+1}} \| + \dots + \| d_{u_{L}} \| \notag \\
		& \leq  \left( 1 + K_{\ell}
               + \dots +  K_{L-1}\prod_{i=\ell+1}^{L-1} (1+ K_{i-1} )  \right) \|d_{u_{\ell}} \| \notag \\
		& = \prod_{i=\ell+1}^{L} (1+ K_{i-1} )   \|d_{u_{\ell}} \|. \label{eq:d-bound}
	\end{align} 
	Applying \eqref{eq:d-bound} and $ \|\cdot \| \leq \|\cdot\|_{1} $ in \eqref{eq:4-1new}, we have
    \begin{align}
        \Theta^{\prime}( z;d )
        \leq  K_{g}   \|d_{u} \|  -\beta_{\ell} \|d_{u_{\ell}} \|   
        &\leq  \left( K_{g} \prod_{i=\ell+1}^{L} (1+ K_{i-1} )    - \beta_{\ell}  \right) \|d_{u_{\ell}} \| \notag \\
        & < - \prod_{i=\ell+1}^{L} (1+ K_{i-1} ) \|d_{u_{\ell}} \| , \label{eq:2.26}
    \end{align}
    where the last inequality uses \eqref{eq:br-threshold-3} and \eqref{eq:4-2new}. However, it follows from the positive homogeneity of $ \Theta^{\prime}(z;\cdot) $ and $ \Theta^{\prime}(z;\tilde{d}) \geq -  \epsilon $ for all $\tilde{d}$ with $ \|\tilde{d}\|=1 $ that 
	\begin{align}
		\Theta^{\prime}(z;d) 
		= \Theta^{\prime}\left(z; \frac{d}{\|d\|} \right)  \|d\| 
		\geq -  \epsilon  \|d\| 
		= -  \epsilon  \|d_{u }\|
		&> - \|d_{u }\| \notag \\
		&\geq - \prod_{i=\ell+1}^{L} (1+ K_{i-1} )   \|d_{u_{\ell}} \|, \label{eq:2.27}
	\end{align}
	where the second inequality uses \eqref{eq:4-2new} and $ \epsilon \in [0,1) $, 
	the last inequality comes from \eqref{eq:d-bound}. 
	The strict inequalities \eqref{eq:2.26}-\eqref{eq:2.27} lead to a contradiction, which implies that $ u_{\ell}=\psi_{\ell-1} ( \theta,u_{1}, \dots,u_{\ell-1} ) $. It yields the result due to the arbitrariness of $\ell$. 
   
\end{proof}
 
Lemma \ref{lem:eD1-F0} shows that larger values of $ (\beta_{1}, \beta_{2}) $ can ensure $ \mathcal{D}_{1}^{\epsilon} \cap lev_{\leq \Theta(z^{0})} \Theta \subseteq \mathcal{F}_{0} $, which plays an important role in the simplification of  $ \Theta^{(2)}(z;d) $ along certain directions as follows.

\begin{lemma}\label{lem:eD1-sd2}
    Set $ \{ \beta_{\ell} >0, \ell \in [L] \} $ satisfying \eqref{eq:br-threshold-3}. Then for any $ \epsilon \in [0,1) $, if $z$ satisfies \eqref{eq:con}, 
	we have
    \begin{itemize}
        \item[(a)] $ \Theta^{\prime}(z;d) > -  \epsilon \|d\| $ for all 
			$
				d \notin \mathcal{T}_{\mathcal{F}_{0}}(z)
        		:= \{ d \in \R^{\bar{N}} \mid  d_{u_{\ell}} = \psi_{\ell-1}^{\prime} ( \theta,  \mathbf{u}_{\ell-1}; d_{\theta}, \\ d_{\mathbf{u}_{\ell-1}} ),\, \ell\in[L]  \}  
			$, and
        \item[(b)] for any $ d $ with $ \|d\|=1 $ and $ \Theta^{\prime}(z;d) = -  \epsilon   $, $ \Theta^{(2)}(z;d) $ can be simplified as in \eqref{eq:sim_T'2}.
    \end{itemize} 
\end{lemma}
\begin{proof}
    $ (a) $. Firstly, it follows from Lemma \ref{lem:eD1-F0} that $ z \in \mathcal{F}_{0} $. Hence, the tangent cone $ \mathcal{T}_{\mathcal{F}_{0}}(\cdot) $ is well-defined at $ z $ and has the closed-form according to \cite[Theorem 3.6]{JWC2025}. 
	Together with \cite[Proposition 3.4]{JWC2025}, 
     $   z \in \mathcal{F}_{0} $  implies that for all $ d \in \R^{\bar{N}} $,
	\begin{equation}
		\begin{aligned}
			\Theta^{\prime}(z;d)
			= \, & F^{\prime}(z;d) + \sum_{\ell=1}^{L} \beta_{\ell} \| d_{u_{\ell}} - \psi_{\ell-1}^{\prime}( \theta, \mathbf{u}_{\ell-1} ; d_{\theta}, d_{\mathbf{u}_{\ell-1}} ) \|_{1}  .
		\end{aligned}
		\label{eq:D0=D1-1new}
	\end{equation}
    For any $ d \notin \mathcal{T}_{\mathcal{F}_{0} }(z)  $, we can construct a direction $\bar{d} \in \mathcal{T}_{\mathcal{F}_{0} }(z) $ as follows: set $ \bar{d}_{\theta}:= d_{\theta} $ and
    $ \bar{d}_{u_{\ell}}:= \psi_{\ell-1}^{\prime} ( \theta, u_{1},  \dots,  u_{\ell-1} ; \\  \bar{d}_{\theta}, \bar{d}_{u_{1}},  \dots , \bar{d}_{u_{\ell-1}} ) $
    in the order of $ \ell=1,\dots, L $. Then, it follows from \eqref{eq:D0=D1-1new} and $\Theta^{\prime}(z; \tilde{d}) \geq -\epsilon $ for all $ \tilde{d} $ with $ \|\tilde{d}\|=1 $  that
    \begin{align}
         & \Theta^{\prime}(z;d)  \notag \\
        = \, & \Theta^{\prime}(z;d) - \Theta^{\prime}(z;\bar{d}) + \Theta^{\prime}(z;\bar{d}) \notag \\
        \geq \, & \Theta^{\prime}(z;d) - \Theta^{\prime}(z;\bar{d})  - \epsilon \| \bar{d} \|   \notag \\
        = \, & \underbrace{F^{\prime}(z;d) - F^{\prime}(z;\bar{d})
            +  \sum_{\ell=1}^{L} \beta_{\ell} \| d_{u_{\ell}} - \psi_{\ell-1}^{\prime}( \theta, \mathbf{u}_{\ell-1} ; d_{\theta},   d_{\mathbf{u}_{\ell-1}} ) \|_{1}}_{ (I) }
			 - \epsilon \| \bar{d} \|  . \label{eq:D0=D1-3new}
    \end{align}
    Next we show that the right hand side of \eqref{eq:D0=D1-3new} is larger than $ -\epsilon \|d\| $. For $ (I) $, note that
    \begin{align}
        F^{\prime}(z;d) - F^{\prime}(z;\bar{d})
        = g^{\prime}(u;d_{u}) - g^{\prime}(u;\bar{d}_{u})
		  &\geq   -K_{g}  \| d_{u } - \bar{d}_{u } \| \notag \\ 
        & \geq -K_{g} \sum_{\ell=1}^{L} \| d_{u_{\ell}} - \bar{d}_{u_{\ell}} \|,\label{eq:4-4.5new}
    \end{align}
    where the equality follows from $ \bar{d}_{\theta} = d_{\theta} $ and \cite[(3.2)]{JWC2025}, 
    and the inequality uses \eqref{eq:Lip}. We only need to estimate $ \{  \| d_{u_{\ell}} - \bar{d}_{u_{\ell}} \| , \ell \in [L] \} $ by induction in the following.
    For $\ell=1$, it follows from the definition of $ \bar{d}_{u_{1}} $ and $ \bar{d}_{\theta} = d_{\theta} $ that
    \begin{align}\label{eq:4-5new}
        \| d_{u_{1}} - \bar{d}_{u_{1}} \| = \| d_{u_{1}} - \psi_{0}^{\prime} ( \theta;d_{\theta} ) \|.
    \end{align}
    For $ \ell=2,\dots,L $, assume that
    \begin{align}
        \label{eq:4-6new}
         &\|d_{u_{j}}- \bar{d}_{u_{j}} \| \\
        \leq \, & \|d_{u_{j}} -  \psi_{j-1}^{\prime}( \theta, \mathbf{u}_{j-1} ; d_{\theta} ,  d_{ \mathbf{u}_{j-1} } ) \| \notag \\
        & + K_{j-1} \sum_{p=1}^{j-1}  \left( \prod_{i=p+1}^{j-1} ( 1+K_{i-1} )  \right)
         \| d_{u_{p}} - \psi_{p-1}^{\prime} ( \theta, \mathbf{u}_{p-1}; d_{\theta},   d_{\mathbf{u}_{p-1}} ) \| \notag
    \end{align}
    holds for all $ j=1,\dots,\ell-1 $. Then we can deduce that \eqref{eq:4-6new} also holds at $j=\ell$:
    \begin{align*}
        &\|d_{u_{\ell}}- \bar{d}_{u_{\ell}} \| \notag \\
        = \, &  \|d_{u_{\ell}} -  \psi_{\ell-1}^{\prime}( \theta, \mathbf{u}_{\ell-1} ; \bar{d}_{\theta} ,   \bar{d}_{ \mathbf{u}_{\ell-1} } ) \| \notag\\
        \leq \, &  \|d_{u_{\ell}} -  \psi_{\ell-1}^{\prime}( \theta, \mathbf{u}_{\ell-1} ; d_{\theta} ,   d_{ \mathbf{u}_{\ell-1} } ) \|  \notag\\
        &+ K_{\ell-1} (   \| d_{u_{1}} - \bar{d}_{u_{1}} \| + \dots  + \| d_{u_{\ell-1}} - \bar{d}_{u_{\ell-1}} \| )  \\
        \leq \, & \|d_{u_{\ell}} -  \psi_{\ell-1}^{\prime}( \theta, \mathbf{u}_{\ell-1} ; d_{\theta} ,   d_{ \mathbf{u}_{\ell-1} } ) \| \notag \\
        &+ K_{\ell-1}
            \begin{pmatrix}
                \|d_{u_{1}}- \psi_{0}^{\prime}( \theta;d_{\theta} ) \| \cdot \left(  1 + K_{1} + \cdots+ K_{\ell-2} \prod_{i=2}^{\ell-2} ( 1+K_{i-1} ) \right) \\
                                +\dots + \|d_{u_{\ell-1}} -  \psi_{\ell-2}^{\prime}( \theta, \mathbf{u}_{\ell-2} ; d_{\theta} ,  d_{ \mathbf{u}_{\ell-2} } ) \|
            \end{pmatrix}\notag\\
        = \, & \|d_{u_{\ell}} -  \psi_{\ell-1}^{\prime}( \theta, \mathbf{u}_{\ell-1} ; d_{\theta} ,  d_{ \mathbf{u}_{\ell-1} } ) \|  \notag\\
        &+ K_{\ell-1}  
            \begin{pmatrix}
                \|d_{u_{1}}- \psi_{0}^{\prime}( \theta;d_{\theta} ) \| \cdot \prod_{i=2}^{\ell-1} ( 1+K_{i-1} ) +\dots\\
                                 + \|d_{u_{\ell-1}} -  \psi_{\ell-2}^{\prime}( \theta, \mathbf{u}_{\ell-2} ; d_{\theta} ,   d_{ \mathbf{u}_{\ell-2} } ) \|
            \end{pmatrix} , \notag
    \end{align*}
    where the first equality comes from the definition of $ \bar{d}_{u_{\ell}} $, the first inequality uses $ d_{\theta} = \bar{d}_{\theta} $ and \eqref{eq:Lip}, the last inequality follows from \eqref{eq:4-6new} at $ j=1,\dots,\ell-1 $.
    Together with \eqref{eq:4-5new}, it implies that \eqref{eq:4-6new} holds for all $ \ell\in [L] $. 
	Summing up these upper bounds for $ \{ \|d_{u_{\ell}} - \bar{d}_{ u_{\ell} }  \|, \ell \in [L]  \} $, we have 
	\begin{align}
		\sum_{\ell =1}^{L}\|d_{u_{\ell}} - \bar{d}_{ u_{\ell} }  \|
		\leq \, &\|d_{u_{1}}- \psi_{0}^{\prime}( \theta;d_{\theta} ) \| \cdot \left(  1 + K_{1} + \cdots+ K_{L-1 } \prod_{i=2}^{L-1} ( 1+K_{i-1} ) \right)   \notag \\
		& +\dots+ \|d_{u_{L }} -  \psi_{L-1 }^{\prime}( \theta, \mathbf{u}_{L-1} ; d_{\theta} ,  d_{ \mathbf{u}_{L-1} } ) \| \notag \\
		= \, & \sum_{\ell=1}^{L} \left( \prod_{j=\ell+1}^{L} (1+ K_{j-1}) \right) \| d_{u_{\ell }} -  \psi_{\ell-1 }^{\prime}( \theta, \mathbf{u}_{\ell-1} ; d_{\theta} ,  d_{ \mathbf{u}_{\ell-1} } )\| . \label{eq:d-dbar}
	\end{align}
	Plugging \eqref{eq:d-dbar} into \eqref{eq:4-4.5new}, we have
    \begin{align*}
        & F^{\prime}(z;d)  -  F^{\prime}(z;\bar{d})  \\
        \geq \,&  - K_{g} \sum_{\ell=1}^{L} \left( \prod_{j=\ell+1}^{L} (1+ K_{j-1}) \right)
        \| d_{u_{\ell }} -  \psi_{\ell-1 }^{\prime}( \theta, \mathbf{u}_{\ell-1} ; d_{\theta} ,  d_{ \mathbf{u}_{\ell-1} } )\|,
    \end{align*}
	which implies that 
    \begin{align*}
        (I) \geq \sum_{\ell=1}^{L} \left( \beta_{\ell} - K_{g}  \left( \prod_{j=\ell+1}^{L} (1+ K_{j-1}) \right) \right)
        \cdot  \| d_{u_{\ell }} -  \psi_{\ell-1 }^{\prime}( \theta, \mathbf{u}_{\ell-1} ; d_{\theta} ,  d_{ \mathbf{u}_{\ell-1} } )\| .
    \end{align*} 
	On the other hand, it follows from the definition of $ \bar{d} $ and \eqref{eq:d-dbar} that
	\begin{align*}
		\| \bar{d} \|  
		& \leq \| d\| + \| d-\bar{d} \| \\
		& \leq  \| d\| + \sum_{\ell =1}^{L}\|d_{u_{\ell}} - \bar{d}_{ u_{\ell} }  \| \\
		& \leq  \| d\| + \sum_{\ell=1}^{L} \left( \prod_{j=\ell+1}^{L} (1+ K_{j-1}) \right) \| d_{u_{\ell }} -  \psi_{\ell-1 }^{\prime}( \theta, \mathbf{u}_{\ell-1} ; d_{\theta} ,  d_{ \mathbf{u}_{\ell-1} } )\| .
	\end{align*} 
    Together with \eqref{eq:D0=D1-3new}, it implies that
    \begin{align*}
        & \Theta^{\prime}(z;d) \\
        \geq \, &\sum_{\ell=1}^{L} \left( \beta_{\ell} -  (K_{g} +\epsilon )   \left( \prod_{j=\ell+1}^{L} (1+ K_{j-1}) \right) \right)
          \| d_{u_{\ell }} -  \psi_{\ell-1 }^{\prime}( \theta, \mathbf{u}_{\ell-1} ; d_{\theta} ,  d_{ \mathbf{u}_{\ell-1} } )\|    \\ 
          & - \epsilon \|d\|   \\
		> \, &  - \epsilon \|d\|,  
    \end{align*}
    where the last inequality uses \eqref{eq:br-threshold-3}, $ \epsilon \in [0,1) $, $d \notin \mathcal{T}_{\mathcal{F}_{0}}(z) $.  

	(b). It follows from \cite[Proposition 3.5]{JWC2025} 
    that for any $ d \in \R^{\bar{N}} $, 
	\begin{align*}
        \Theta^{(2)} (z;d)  
         =\,& F^{(2)}(z;d) 
            + \sum_{\ell=1}^{L} \beta_{\ell} \left(
             - \sum_{i \in I_{+}^{\ell}(z) \cup I_{0,+}^{\ell}(z;d)  } [  \psi_{\ell-1}^{(2)} ( \theta, \mathbf{u}_{\ell-1} ; d_{\theta},  d_{\mathbf{u}_{\ell-1}} )  ]_{i} \right. \notag \\
             &  + \sum_{i \in I_{-}^{\ell}(z) \cup I_{0,-}^{\ell}(z;d) } [   \psi_{\ell-1}^{(2)} ( \theta, \mathbf{u}_{\ell-1} ; d_{\theta},   d_{\mathbf{u}_{\ell-1}} )  ]_{i} 
			   \notag \\
             &\left. 
				+\sum_{i \in I_{0,0}^{\ell}(z;d)  } \left| [ \psi_{\ell-1}^{(2)} ( \theta, \mathbf{u}_{\ell-1} ; d_{\theta},   d_{\mathbf{u}_{\ell-1}} )  ]_{i} \right|
             \right)  , \notag
     \end{align*}
where for all $ \ell \in [L] $, $ I_{+}^{\ell}(z), I_{-}^{\ell}(z), I_{0}^{\ell}(z) $ are defined in \eqref{eq:I_+-0}, and 
	\begin{align*}
		& I_{0,+}^{\ell}(z;d):= \{ i \in I_{0}^{\ell}(z) \mid [ d_{u_{\ell}} - \psi_{\ell-1}^{\prime} ( \theta, \mathbf{u}_{\ell-1} ; d_{\theta},  d_{\mathbf{u}_{\ell-1}} )  ]_{i} >0 \} , \\
		& I_{0,-}^{\ell}(z;d):= \{ i \in I_{0}^{\ell}(z) \mid [ d_{u_{\ell}} - \psi_{\ell-1}^{\prime} ( \theta, \mathbf{u}_{\ell-1} ; d_{\theta},   d_{\mathbf{u}_{\ell-1}} )  ]_{i} <0 \} ,  \\
    & I_{0,0}^{\ell}(z;d):= I_{0}^{\ell}(z) \backslash \left( I_{0,+}^{\ell}(z;d)  \cup I_{0,-}^{\ell}(z;d) \right) .
	\end{align*}
Based on Lemma \ref{lem:eD1-F0}, we have $ I_{+}^{\ell}(z) = I_{-}^{\ell}(z) = \emptyset $ for all $ \ell \in [L] $. Furthermore, for any $ d $ with $ \|d\|=1 $ and $ \Theta^{\prime}(z;d) = -  \epsilon   $, it follows from the result (a) that $ d \in \mathcal{T}_{\mathcal{F}_{0}}(z)
= \{ d \in \R^{\bar{N}} \mid  d_{u_{\ell}} = \psi_{\ell-1}^{\prime} ( \theta,  \mathbf{u}_{\ell-1}; d_{\theta},   d_{\mathbf{u}_{\ell-1}} ), \ell\in[L]  \} $, which implies that $ I_{0,+}^{\ell}(z;d) = I_{0,-}^{\ell}(z;d) = \emptyset $ for all $ \ell \in [L] $. Thus, $ I_{0,0}^{\ell}(z;d) = [N_{\ell}] $ for all $ \ell \in [L] $, which completes the proof.
   
\end{proof}

Finally, we obtain $ \mathcal{D}_{1}^{\epsilon} \cap lev_{\leq \Theta(z^{0}) } \Theta \subseteq \mathcal{SD}_{1}^{\epsilon} $ for $ \epsilon \in [0,1) $, where $ \mathcal{D}_{1}^{\epsilon} , \mathcal{SD}_{1}^{\epsilon} $ are defined in \eqref{def:eD_eSD}, and $ z^{0} $ is defined in \eqref{def:z0}.
\begin{theorem}
	\label{thm:eD-eSD}
	Set $ \{ \beta_{\ell} >0, \ell \in [L] \} $ satisfying \eqref{eq:br-threshold-3} and assume that $ g $ is convex. 
	Then for any $ \epsilon \in [0,1) $, we have $ \mathcal{D}_{1}^{\epsilon} \cap lev_{\leq \Theta(z^{0}) } \Theta \subseteq \mathcal{SD}_{1}^{\epsilon} $.
\end{theorem}
\begin{proof}
	For any $ z \in \mathcal{D}_{1}^{\epsilon} \cap lev_{\leq \Theta(z^{0}) } \Theta $,  it follows from Lemma \ref{lem:eD1-sd2} that for any $ d $ with $ \|d\|=1 $ and $ \Theta^{\prime}(z;d) = -  \epsilon   $, 
	\begin{align*}
		\Theta^{(2)}(z;d) \geq  F^{(2)}(z;d)  
		\geq 0
		\geq -\epsilon,
	\end{align*}
	where the second inequality holds due to the convexity of $ g $ and $ F(z)= g(u) + \lambda \|\theta \|^{2} $ for all $ z = ( \theta^{\top}, u^{\top} )^{\top} $. 
   
\end{proof} 

Theorem \ref{thm:eD-eSD} demonstrates that the inequalities \eqref{eq:br-threshold-3} and the convexity of $ g $ allow $ y \in \mathcal{SD}_{1}^{\epsilon} $ to be verified simply by checking $ \Theta(y) \leq \Theta(z^{0}) $ and $ y \in \mathcal{D}_{1}^{\epsilon} $.
It enables us to obtain an $ \epsilon $-approximate second-order d-stationary point via the first-order algorithm presented in subsequent sections.

\section{SRPA}
\label{sec:r-SRPA}  

To find an $ \epsilon $-approximate d-stationary point of \eqref{l1pen_r} with complexity bounds, we propose the algorithm SRPA in Algorithm \ref{alg2}. 

In $k$th iteration, a piecewise affine approximation of $ \Theta $ is constructed by substituting the smooth functions $ F  $ and $c$ with their linear expansions at $ z^{k} $, respectively. By incorporating a quadratic regularization term, we obtain a function $ \Phi_{k}(s) \approx \Theta(z^{k} +s ) $ for sufficiently small $ s $. 
Then, $ \Phi_{k} $ is minimized within a neighborhood of $ \mathbf{0} $ to generate a trial step.  
If $ \Phi_{k} $ has the same value at the trial step and $ \mathbf{0}  $, SRPA terminates. 
Otherwise, the algorithm checks whether the value of $ \Theta $ decreases sufficiently and updates the iteration point and the regularization parameter correspondingly. 

The selection of the neighborhood for each subproblem is critical for computational efficiency.
Inspired by the idea of active-set in \cite{PRA2017B-sta}, we select the neighborhood $ \mathcal{F}_{k}^{\delta} := \mathcal{F}^{\delta}(z^{k}) $ for a given $ \delta>0 $ as follows.
For any $ z \in \R^{\bar{N}} $, denote
\begin{equation}
	\label{def:AdFd}
	\begin{aligned}
		  \mathcal{A}^{\delta}(z) := \{ i \in [I] \mid \mathbf{d}(z,P_{i}) \leq \delta \}, \,
		  \mathcal{F}^{\delta}(z) := 
		\{ \bar{z}-z \mid \bar{z} \in \cup_{ i \in \mathcal{A}^{\delta}(z) } P_{i} \}, 
	\end{aligned}
\end{equation}
where the polyhedra $ P_{i}, \, i \in [I] $ are defined in the discussion after Definition \ref{def:PA}.
It can be observed that for any $ z $, $ \mathcal{A}^{\delta}(z) $ and $ \mathcal{F}^{\delta}(z) $ contain
\begin{equation}
	\label{def:AF}
	\begin{aligned}
		 \mathcal{A}(z) := \{ i \in [I] \mid   z \in P_{i}   \}, \,
		 \mathcal{F} (z) := 
		\{ \bar{z}-z \mid \bar{z} \in \cup_{ i \in \mathcal{A}(z) } P_{i} \} ,
	\end{aligned}
\end{equation}
respectively, which implies that $ \mathcal{F}^{\delta}(z) $ constitutes a neighborhood of $ \mathbf{0}  $ as follows.
\begin{lemma}
	\label{lem:intFd}
	For any $ z \in \R^{\bar{N}} $, $ z \in int (\cup_{ i \in \mathcal{A}(z) } P_{i} ) \subseteq int (\cup_{ i \in \mathcal{A}^{\delta}(z) } P_{i} ) $ and 
	$ \mathbf{0} \in int ( \mathcal{F}(z) ) \subseteq int ( \mathcal{F}^{\delta}(z) ) $.
\end{lemma}
\begin{proof} 
	If $ z \notin int (\cup_{ i \in \mathcal{A}(z) } P_{i} ) $, then there exists a sequence $ \{ z_{t} 
	\in \R^{\bar{N}} \backslash \cup_{ i \in \mathcal{A}(z) } P_{i}   \} $ converging to $ z $.
	Together with $\R^{\bar{N}} = \cup_{i \in [I]} P_{i}$, it implies that $ \{z_{t} \} \subseteq \cup_{ i \in [I] \backslash \mathcal{A}(z) } P_{i} $. 
	Since $ | [I] \backslash \mathcal{A}(z) | $ is finite, there exists an $ \bar{i} \in [I] \backslash \mathcal{A}(z) $ and a subsequence $ \{z_{t_{j}} \} $ such that $ \{z_{t_{j}} \} \subseteq P_{\bar{i}} $. 
	Together with the closedness of $ P_{\bar{i}} $ and $ z_{t_{j}} \rightarrow z $, it implies that $ z \in P_{\bar{i}} $.
	According to \eqref{def:AF}, it indicates that $ \bar{i} \in \mathcal{A}(z) $, which leads to a contradiction. 
	Hence, $ z \in int (\cup_{ i \in \mathcal{A}(z) } P_{i} ) $.
	Together with $ \cup_{ i \in \mathcal{A}(z) } P_{i}  \subseteq  \cup_{ i \in \mathcal{A}^{\delta}(z) } P_{i} $ and definitions of $ \mathcal{F} (z), \mathcal{F}^{\delta}(z) $, we yield the result. 
   
\end{proof} 
By the definition of $ \mathcal{F}_{k}^{\delta} $, finding a solution of $ \min_{s \in \mathcal{F}_{k}^{\delta} } \Phi_{k}(s) $ is equivalent to minimizing $ \Phi_{k} $ on polyhedron 
$ \{z-   z^{k} \mid z\in P_{i} \} $ for all
$ i \in \mathcal{A}^{\delta}(z^{k})  $ and selecting the solution at which $ \Phi_{k} $ has the minimum value. 
Besides, $ \Phi_{k} $ is strongly convex on each polyhedron $ \{z-   z^{k} \mid z\in P_{i} \} $  since $ H $ is affine on each $ P_{i} $ and $ ( \Phi_{k}(\cdot) - H(z^{k} +\cdot ) ) $ is strongly convex. 
Consequently, the subproblem is split as $ | \mathcal{A}^{\delta}(z^{k}) | $ strongly convex optimization problems that can be effectively solved by existing algorithms. 
Thus, a small $ \delta $ can reduce the computational cost for each subproblem. 
However, the subsequent convergence analysis demonstrates that an excessively small $ \delta $ may increase the number of iterations. 
Therefore, $ \delta $ should be set properly to balance the computational cost of each subproblem and the number of iterations.

\begin{algorithm}[htbp] 
  \caption{SRPA for \eqref{l1pen_r}}
  \label{alg2} 
  \begin{algorithmic}[1] 
    \REQUIRE Given $ z^{1} \in \R^{\bar{N}} $, $\rho_{1}>0 $, $ \eta_{1} \in (0,1), \eta_{2} >1 $, $ \delta>0 $. 
	\FOR{$ k=1,\dots $}
		\STATE Identify the $ \delta $-active index set $ \mathcal{A}_{k}^{\delta}:= \mathcal{A}^{\delta}(z^{k}) $ by \eqref{def:AdFd}.
		\STATE For each $ i \in \mathcal{A}_{k}^{\delta} $, calculate
            \begin{align*}
                s^{k,i} = \argmin_{s \in  \{z-   z^{k} \mid z\in P_{i} \}    } \Phi_{k}(s):=\, & F(z^{k}) + \nabla F (z^{k})^{\top} s + H (z^{k} + s ) \\
				          &+ \sum_{j \in \mathcal{J}} | c_{j}( z^{k} )  + \nabla c_{j}( z^{k} ) ^{\top} s  |+ \frac{\rho_{k}}{2} \|s\|^{2} .
            \end{align*} 
		\STATE Set $ s^{k} \in \argmin_{ i \in \mathcal{A}_{k}^{\delta} } \Phi_{k}( s^{k,i} ) $.
		\STATE If $ \Phi_{k}(s^{k})=\Phi_{k}( \mathbf{0} ) $, then terminate and output $ z^{k} $.
		\STATE Update $ z^{k+1} $ and $ \rho_{k+1} $ by
			\begin{align*}
				(z^{k+1} , \rho_{k+1} ):=
        \left\{
				\begin{array}{lll}
					 (z^{k} +s^{k} , \rho_{k} ), & \text{if } \Theta(z^{k}) -  \Theta(z^{k} +s^{k}) \geq \frac{\eta_{1} \rho_{k} }{2 } \|s^{k} \|^{2},    \\
					 (z^{k}  , \eta_{2} \rho_{k} ), & \text{otherwise}.   
				\end{array} 
				\right.
			\end{align*} 
	\ENDFOR 
  \end{algorithmic} 
\end{algorithm}

\begin{remark}
	\label{remark:subp}
	In practice, the process of expressing $ H $ as an affine function on each polyhedron can be simplified. 
	Consider $ H(z):=\| [z]_{1:M}- [[z]_{M+1:2M}]_{+} \|_{1} $ for example, where $ M\in \mathbb{Z}_{++}, z \in \R^{2M} $. 
	For any $ \boldsymbol{\nu}  :=(\nu_{1},\dots, \nu_{M})^{\top} \in \{ -1,1 \}^{M} $, we have 
	\begin{align*}
		H(z)= \bar{H}_{\boldsymbol{\nu}}(z)
		:= \sum_{m=1}^{M} \left| z_{m} - [\nu_{m}]_{+} z_{M+m}  \right|
	\end{align*}
	for all $ z \in \bar{P}_{\boldsymbol{\nu}}:= \{ z \mid \nu_{m} \,  z_{M+m} \geq 0 \text{ for all } m \in [M] \} $, which expresses $ H $ as a convex function on each polyhedron $ \bar{P}_{\boldsymbol{\nu}} $.
	 
	Correspondingly, the process of selecting positive polyhedron can also be simplified. 
	Denote the indicator set at $ k $th iteration as $ \bar{\mathcal{A}}_{k}^{\delta}:=\{ \boldsymbol{\nu} \in \{ -1,1 \}^{M} \mid \mathbf{d}( z^{k}, \bar{P}_{\boldsymbol{\nu}} ) \leq \delta \} $. 
	It follows from the definition of $ \bar{P}_{\boldsymbol{\nu}} $ that for any $ \boldsymbol{\nu} \in \{ -1,1 \}^{M} $, 
	\begin{align*}
		&\mathbf{d}( z^{k}, \bar{P}_{\boldsymbol{\nu}} ) \\
		=\, &  \min_{ \bar{z} \in \bar{P}_{\boldsymbol{\nu}} } \left(\sum_{i=1}^{2M} (z_{i}^{k} - \bar{z}_{i}  )^{2}\right)^{\frac{1}{2}} \\
		=\, & \min \left\{ \left(\sum_{m=1}^{M} (z_{M+m}^{k} - \bar{z}_{M+m}  )^{2}\right)^{\frac{1}{2}} \mid  \nu_{m} \,  \bar{z}_{M+m} \geq 0 \text{ for all } m \in [M] \right\} \\
		=\, & \left( \sum_{m=1}^{M} \min \{ (z_{M+m}^{k} - \bar{z}_{M+m}  )^{2} \mid \nu_{m} \,  \bar{z}_{M+m} \geq 0 \}  \right)^{\frac{1}{2}}.
	\end{align*}
	Therefore, if $ | z_{M+m}^{k} | >\delta $, then for any $ \bar{P}_{\bar{\boldsymbol{\nu}}} $ corresponding to $ \bar{\boldsymbol{\nu}} $ with $ \bar{\nu}_{M+m}= - sign( z_{M+m}^{k} ) $, we have $ \mathbf{d}( z, \bar{P}_{\bar{\boldsymbol{\nu}}} ) > \delta $, which implies that $ \bar{\boldsymbol{\nu}} \notin \bar{\mathcal{A}}_{k}^{\delta} $. 
	Thus, 
	\begin{align*}
		\bar{\mathcal{A}}_{k}^{\delta}
		\subseteq
		\left\{ \boldsymbol{\nu} \in \R^{M} \Bigg|  
			\nu_{m} = \left\{
			\begin{array}{ll}
				1, & z_{M+m}^{k} >\delta, \\
				-1, & z_{M+m}^{k} < -\delta, \\
				\pm 1 , &  | z_{M+m}^{k} | \leq \delta, 
			\end{array}
			\,\right.
			\forall m \in [M]
		\right\}
		=: \hat{\mathcal{A}}_{k}^{\delta}.
	\end{align*}
	Through more refined calculations, one can further identify $ \bar{\mathcal{A}}_{k}^{\delta} $ within $ \hat{\mathcal{A}}_{k}^{\delta} $. In fact, the subsequent proof demonstrates that convergence still holds when $ \bar{\mathcal{A}}_{k}^{\delta} $ is replaced by its superset in each iteration. Therefore, in $k$th iteration, we can directly consider the subproblem
	$
		\min_{ s \in \{ z-z^{k} \mid z \in \bar{P}_{\boldsymbol{\nu}} \} } \Phi_{k}(s)
	$
	for all $ \boldsymbol{\nu} \in \hat{\mathcal{A}}_{k}^{\delta} $. 
	It is equivalent to solving the strongly convex optimization problem
	\begin{align} 
		\label{eq:rem_sub}
		\min_{s  }\, \Phi_{k}(s), ~ \text{s.t. } \nu_{m}  \,  [z^{k} + s ]_{M+m} \geq 0, \, m \in [M],
	\end{align}
	for all $ \boldsymbol{\nu} \in \R^{M}  $ with $ \nu_{m} = \left\{
			\begin{array}{ll}
				1, & z_{M+m}^{k} >\delta, \\
				-1, & z_{M+m}^{k} < -\delta, \\
				1 \text{ or } -1 , &  | z_{M+m}^{k} | \leq \delta, 
			\end{array}
			\,\right.$ for all $m \in [M] $, and finding the one solution at which $ \Phi_{k} $ has the minimum value. 
	Furthermore, \eqref{eq:rem_sub} corresponding to different $ \boldsymbol{\nu} $ can be solved in parallel to save time. 
\end{remark}

\subsection{Global convergence}
\label{subsec:1-d.1-r}
Next we show the convergence of the sequence generated by Algorithm \ref{alg2}.
Firstly, it can be observed that if Algorithm \ref{alg2} terminates at $k$th iteration, $ z^{k} $ is a d-stationary point of \eqref{l1pen_r}. 
\begin{lemma}
	\label{lem:2.1-r}
	If Algorithm \ref{alg2} terminates at $ k $th iteration, then $ \Theta^{\prime} (z^{k} ;d )\geq 0 $ for all $ d $.
\end{lemma}
\begin{proof}
	It follows from $ \Phi_{k}(s^{k})=\Phi_{k}( \mathbf{0} ) $, $ s^{k} \in \argmin_{  s \in \mathcal{F}_{k}^{\delta}  } \Phi_{k}(s) $ and $ \mathbf{0} \in \mathcal{F}_{k}^{\delta} $ from Lemma \ref{lem:intFd} that 
    \begin{align*}
        \mathbf{0} \in \argmin_{  s \in \mathcal{F}_{k}^{\delta}  } \Phi_{k}(s) .
    \end{align*}
    Hence, it follows from the optimality condition and \cite[Proposition 4.1.2]{CP-book} that for all $ d \in   \mathcal{T}_{\mathcal{F}_{k}^{\delta}} (\mathbf{0})  $,
	\begin{align*}
		0 \leq \Phi_{k}^{\prime} ( \mathbf{0} ; d ) 
		=  \nabla F (z^{k})^{\top} d + H^{\prime} (z^{k} ;d ) + \sum_{j \in \mathcal{J}} f_{j}^{\prime} (z^{k} ;d ) +0
		= \Theta^{\prime} (z^{k} ;d ),
	\end{align*}
	which yields the result by $ \mathcal{T}_{\mathcal{F}_{k}^{\delta}} (\mathbf{0}) = \R^{\bar{N}} $ from Lemma \ref{lem:intFd}.
   
\end{proof}

Next, we will focus on the convergence of Algorithm \ref{alg2} in the case where it does not terminate in a finite number of steps. 
It follows from the continuity of $  \nabla F   $ and the level-boundedness of $ \Theta $ \cite[proof of Lemma 2.2]{JWC2025} that, there exists $ K_{F}  >0 $ such that 
	\begin{align}
		\label{eq:KF-r}
		 \| \nabla F(z) \| \leq K_{F}  
		\text{ for all } z \in lev_{\leq \Theta(z^{1}) } \Theta.
	\end{align}   
	Since $ \nabla F$ and $ \nabla c$ are locally Lipschitz continuous on $ \R^{ \bar{N} }  $, there exist $  \bar{K}_{F}, \bar{K}_{c}>0 $ such that 
	\begin{align}
		\label{eq:Kbar-r}
		\|\nabla F(z) - \nabla F( y )  \| \leq \bar{K}_{F} \|z-y\|, \,
		\|\nabla c(z) - \nabla c( y )  \| \leq \bar{K}_{c} \|z-y\|   
	\end{align}
	for all $ z, y  \in lev_{\leq \Theta(z^{1}) } \Theta + \mathbb{B}(\mathbf{0} ; 2 \rho_{1}^{-1} K_{F} + \sqrt{  2 \rho_{1}^{-1} \Theta(z^{1}) } ) $, where $ \sqrt{  2 \rho_{1}^{-1} \Theta(z^{1}) }   $ is well-defined since $ \rho_{1}>0,\Theta(z^{1}) \geq 0 $. 
	The following lemma shows that \eqref{eq:KF-r} and \eqref{eq:Kbar-r} hold for some sequences generated by Algorithm \ref{alg2}.
\begin{lemma}
	\label{lem:2.2-r}
	For $ \{ z^{k} \} $ and $ \{ s^{k} \} $ generated by Algorithm \ref{alg2}, we have  
	\begin{align*} 
		& \{ z^{k} \}  \subseteq lev_{\leq \Theta(z^{1}) } \Theta, \text{ and }\\
		& \{ z^{k} + \tau s^{k} \mid \tau \in [0,1], k\geq 1 \} \subseteq lev_{\leq \Theta(z^{1}) } \Theta + \mathbb{B}(\mathbf{0} ; 2 \rho_{1}^{-1} K_{F} + \sqrt{  2 \rho_{1}^{-1} \Theta(z^{1}) } ) .
	\end{align*}  
\end{lemma}
\begin{proof}
	Firstly, it follows from step 6 of Algorithm \ref{alg2} that $ \Theta( z^{k} ) \geq \Theta(z^{k+1}) $ and $ \rho_{k}\leq \rho_{k+1} $ for all $k$. Hence, $ \Theta( z^{k} ) \leq \Theta (z^{1}) $ and $ \rho_{k}\geq \rho_{1} $ for all $k$.
	Together with $ \Phi_{k}(\mathbf{0})= \Theta(z^{k}) $, $ s^{k} \in \argmin_{s \in \mathcal{F}_{k}^{\delta}  } \Phi_{k}(s) $ and $ \mathbf{0} \in \mathcal{F}_{k}^{\delta} $, it implies that for any $k$,
	\begin{align*}
		\Theta(z^{1}) \geq \Phi_{k}(\mathbf{0}) 
		\geq \Phi_{k}(s^{k})
		&\geq \nabla F(z^{k})^{\top} s^{k} + \frac{\rho_{k}}{2} \|s^{k} \|^{2} \\
		&= \frac{\rho_{k}}{2} \|s^{k} + \frac{1}{\rho_{k}} \nabla F(z^{k}) \|^{2} - \frac{1}{2 \rho_{k}} \| \nabla F(z^{k})  \|^{2},
	\end{align*}
	where the last inequality uses the nonnegativity of $F, H $ and $ |\cdot| $. After reorganization, we have
	\begin{equation}
		\label{eq:sk_bound-r}
		\begin{aligned}
			\|s^{k}\|
            & \leq \|s^{k} + \frac{1}{\rho_{k}} \nabla F(z^{k}) \| + \frac{1}{\rho_{k}} \|  - \nabla F(z^{k}) \| \\
			& \leq \sqrt{ \frac{2}{\rho_{k}} \Theta(z^{1}) + \frac{1}{\rho_{k}^{2}} \| \nabla F(z^{k})  \|^{2} } + \frac{1}{\rho_{k}} \|   \nabla F(z^{k}) \| \\
			& \leq \sqrt{ \frac{2}{\rho_{1}} \Theta(z^{1}) + \frac{1}{\rho_{1}^{2}} K_{F}^{2} } + \frac{1}{\rho_{1}} K_{F} \\
			& \leq \sqrt{ 2 \rho_{1}^{-1} \Theta(z^{1})} + 2 \rho_{1}^{-1} K_{F} ,
		\end{aligned}
	\end{equation}
	where the third inequality follows from $ \rho_{k} \geq \rho_{1} , 0 \leq \Theta( z^{k} ) \leq \Theta (z^{1}) $ and \eqref{eq:KF-r}, the last inequality uses $ \sqrt{a+b}  \leq \sqrt{a} +\sqrt{b} $ for all $ a,b \in \R_{+} $. Thus, for any $ \tau \in[0,1] $ and any $k$, we have $ z^{k} \in lev_{\leq \Theta(z^{1})} \Theta $ and $ \| \tau s^{k}\| \leq 2 \rho_{1}^{-1} K_{F} + \sqrt{ 2 \rho_{1}^{-1} \Theta(z^{1})}    $, which completes the proof. 
   
\end{proof}

Lemma \ref{lem:2.2-r} implies that the infinite sequence $ \{ z^{k} \} $ generated by Algorithm \ref{alg2} is bounded. Thus, there is at least one accumulation point.
For any $ k\geq 1 $, denote the sets of successful and unsuccessful iterations as
\begin{align*}
	& S_{k} := \left\{ \hat{k} \in \mathbb{Z}_{++} \mid \hat{k} \leq k, \Theta(z^{\hat{k}}) -  \Theta(z^{\hat{k}} +s^{\hat{k}}) \geq \frac{\eta_{1} \rho_{\hat{k}} }{2 } \|s^{\hat{k}} \|^{2} \right\}, \\ 
	& \bar{S}_{k} := \{ 1, \dots, k \} \backslash S_{k} = \left\{ \hat{k} \in \mathbb{Z}_{++} \mid  \hat{k} \leq k, \Theta(z^{\hat{k}}) -  \Theta(z^{\hat{k}} +s^{\hat{k}}) < \frac{\eta_{1} \rho_{\hat{k}} }{2 } \|s^{\hat{k}} \|^{2} \right\}.
\end{align*}
In the following lemma, we show that $ | \bar{S}_{k} | $ is uniformly bounded, i.e., the number of unsuccessful iterations is finite. 
\begin{lemma}
	\label{lem:2.3-r}
	For any $ k \geq 1 $, $ |  \bar{S}_{k} | \leq    [  \lceil \log_{\eta_{2}}  ( \frac{\bar{K}_{F} + \sqrt{J} \bar{K}_{c} }{\rho_{1} ( 1-\eta_{1}  ) }  )  \rceil  ]_{+} +1 =: \kappa_{1}  $.
\end{lemma}
\begin{proof}
	It follows from \eqref{eq:Kbar-r} and Lemma \ref{lem:2.2-r} that  
	\begin{align*}
		& \left| F(z+s) - F(z) - \nabla F(z)^{\top} s  \right| \leq \frac{\bar{K}_{F}}{2} \|s\|^{2}, \\
		& \left| \sum_{j \in \mathcal{J}} (|c_{j}(z+s)| - | c_{j}( z )  + \nabla c_{j}( z  ) ^{\top} s  |) \right| 
			\leq \|  c(z+s) - c(z) - \nabla c(z) ^{\top} s \|_{1} \\
            &~~~~~~~~~~~~~~~~~~~~~~~~~~~~~~~~~~~~~~~~~~~~~~~~~\,
			\leq \frac{\sqrt{J} \bar{K}_{c}}{2} \|s\|^{2},  
	\end{align*}
	for all $ (z,s) = ( z^{k} , s^{k} ), \, k\geq 1  $. Hence, for all $k \geq 1$, 
	\begin{equation}
		\label{eq:new2.3-r}
		\begin{aligned}
			& \Theta(z^{k}) - \Theta(   z^{k}+s^{k}   ) \\
			= \,& \Phi_{k}(\mathbf{0}) - \Phi_{k}( s^{k} ) + \Phi_{k}( s^{k} ) - \Theta(z^{k} + s^{k} ) \\
			> \,&  0 
				+  \left( F(z^{k}) + \nabla F (z^{k})^{\top} s^{k} - F( z^{k} + s^{k} ) \right)   \\
				&+ \sum_{j \in \mathcal{J}} \left( | c_{j}( z^{k} )  + \nabla c_{j}( z^{k} ) ^{\top} s^{k}  | - | c_{j}( z^{k} + s^{k} ) |  \right)+ \frac{\rho_{k}}{2} \|s^{k}\|^{2} 
				 \\
			\geq \,&  \frac{ \rho_{k}-\bar{K}_{F} - \sqrt{J} \bar{K}_{c} }{2} \|s^{k} \|^{2}.
		\end{aligned}
	\end{equation}
	If $ | \bar{S}_{k} | >   [  \lceil \log_{\eta_{2}}  ( \frac{\bar{K}_{F} + \sqrt{J} \bar{K}_{c} }{\rho_{1} ( 1-\eta_{1}  ) }  )  \rceil  ]_{+} +1 $, we have $ | \bar{S}_{k} | \geq  [  \lceil \log_{\eta_{2}}  ( \frac{\bar{K}_{F} + \sqrt{J} \bar{K}_{c} }{\rho_{1} ( 1-\eta_{1}  ) }  )  \rceil  ]_{+} +2   $ since $  [ \lceil \log_{\eta_{2}}  ( \frac{\bar{K}_{F} + \sqrt{J} \bar{K}_{c} }{\rho_{1} ( 1-\eta_{1}  ) }  )  \rceil  ]_{+}   $ and  $ | \bar{S}_{k} | $ are integers, which implies that $ | \bar{S}_{k} | \geq 2 $. Denoting the second largest and the largest elements in $ \bar{S}_{k} $ as $ k_{1} $ and $ k_{2} $ respectively, we have $ 1\leq k_{1} <k_{2} \leq k $ and $ | \bar{S}_{k_{1}} | = | \bar{S}_{k} \backslash \{ k_{2} \} | = | \bar{S}_{k} |-1 \geq  [  \lceil \log_{\eta_{2}}  ( \frac{\bar{K}_{F} + \sqrt{J} \bar{K}_{c} }{\rho_{1} ( 1-\eta_{1}  ) }  )  \rceil  ]_{+} +1 $. Together with step 6 of Algorithm \ref{alg2}, it implies that
	\begin{align*}
		\rho_{k_{2}}= \dots = \rho_{k_{1}+1} 
		&= \rho_{1} \eta_{2}^{ |  \bar{S}_{k_{1}} | } \\
		&\geq \rho_{1} \eta_{2} \frac{\bar{K}_{F} + \sqrt{J} \bar{K}_{c} }{\rho_{1} ( 1-\eta_{1}  ) } 
		= \eta_{2} \frac{\bar{K}_{F} + \sqrt{J} \bar{K}_{c} }{  ( 1-\eta_{1}  ) }
		> \frac{\bar{K}_{F} + \sqrt{J} \bar{K}_{c} }{  ( 1-\eta_{1}  ) }.
	\end{align*}
	Combining it with \eqref{eq:new2.3-r}, we have
	\begin{align*}
		\Theta(z^{k_{2}}) - \Theta( z^{k_{2}} + s^{k_{2}}  )  
		\geq   \frac{ \rho_{k_{2}}-\bar{K}_{F} - \sqrt{J} \bar{K}_{c} }{2} \|s^{k_{2}} \|^{2}
		> \frac{\eta_{1} \rho_{k_{2}} }{ 2 } \|s^{k_{2}} \|^{2}.
	\end{align*}
	According to step 6 of Algorithm \ref{alg2}, it indicates that the $  k_{2} $th iteration is successful, which contradicts $ k_{2} \in \bar{S}_{k} $. 
   
\end{proof}

Actually, it can be observed from \eqref{eq:new2.3-r} that if the values of $ ( \bar{K}_{F} ,\bar{K}_{c} ) $ are available, we can set $ \rho_{1}= \frac{ \bar{K}_{F} + \sqrt{J} \bar{K}_{c} }{ 1-\eta_{1} } $ to simplify Algorithm \ref{alg2} as an algorithm with fixed regularization coefficient since the iteration is always successful. 
In general, the step 6 of Algorithm \ref{alg2}, \eqref{eq:sk_bound-r} and  Lemma \ref{lem:2.3-r} imply the following corollary.
\begin{corollary}
	\label{cor:2.1-r}
	For any $ k\geq 1 $, 
	\begin{align}
		& \rho_{1} \leq  \rho_{k} \leq  \eta_{2} \max\left\{ \rho_{1} , \frac{ \eta_{2} ( \bar{K}_{F} + \sqrt{J} \bar{K}_{c} ) }{ 1 - \eta_{1} }  \right\} =: \rho  , \label{eq:rhok_bound-r} \\
		&   \sum_{   \hat{k} \in  S_{k}   } \|s^{\hat{k}} \|^{2}   \leq \frac{ 2 \left( \Theta(z^{1}) - \Theta^{*} \right) }{ \eta_{1} \rho_{1} }, \label{eq:ssk2_bound-r}
	\end{align}
	where $ \Theta^{*}:=\min_{z} \Theta(z) $. And $ \lim_{k \rightarrow \infty} s^{k}  = \mathbf{0} $.
\end{corollary} 
\begin{proof}
    Firstly, $ \rho_{k}\geq \rho_{1} $ for all $ k \geq 1 $ as established in the proof of Lemma \ref{lem:2.2-r}. 
	Since $ \eta_{2}>1 $, \eqref{eq:rhok_bound-r} naturally holds at $ k=1 $. For $ k \geq 2 $, step 6 of Algorithm \ref{alg2} and Lemma \ref{lem:2.3-r} imply that 
	\begin{align*}
		\rho_{k} 
		= \rho_{1} \eta_{2}^{  | \bar{S}_{ k-1 } | }  
		&\leq \rho_{1} \eta_{2}^{  \left[ \left\lceil \log_{\eta_{2}} \left( \frac{\bar{K}_{F} + \sqrt{J} \bar{K}_{c} }{\rho_{1} ( 1-\eta_{1}  ) } \right) \right\rceil \right]_{+} +1  } \\
		& = \rho_{1} \eta_{2} \max \left\{ 1,  \, \eta_{2}^{   \left\lceil \log_{\eta_{2}} \left( \frac{\bar{K}_{F} + \sqrt{J} \bar{K}_{c} }{\rho_{1} ( 1-\eta_{1}  ) } \right) \right\rceil   }\right\}  \\
		& \leq \rho_{1} \eta_{2} \max \left\{ 1,  \,   \frac{ \eta_{2}  ( \bar{K}_{F} + \sqrt{J} \bar{K}_{c}) }{\rho_{1} ( 1-\eta_{1}  ) }    \right\}, 
	\end{align*} 
	where the second inequality follows from $ \lceil a \rceil \leq a+1 $ for all $ a \in \R $. 

	Then, it follows from \eqref{eq:sk_bound-r} and Lemma \ref{lem:2.3-r} that 
	\begin{align}
		& \sum_{   \hat{k} \in \bar{S}_{k}   } \|s^{\hat{k}} \|^{2} \notag \\
		\leq \, & | \bar{S}_{k} | \cdot \left( \sqrt{ 2 \rho_{1}^{-1} \Theta(z^{1})} + 2 \rho_{1}^{-1} K_{F} \right)^{2} \notag \\
		\leq \, & \left(  \left[ \left\lceil \log_{\eta_{2}} \left( \frac{\bar{K}_{F} + \sqrt{J} \bar{K}_{c} }{\rho_{1} ( 1-\eta_{1}  ) } \right) \right\rceil \right]_{+} +1  \right) \left( \sqrt{ 2 \rho_{1}^{-1} \Theta(z^{1})} + 2 \rho_{1}^{-1} K_{F} \right)^{2}. \label{eq:2.7-r}
	\end{align}
	At the same time, $ \rho_{1} \leq \rho_{k} $ and the step 6 of Algorithm \ref{alg2} indicate that 
	\begin{align*}
		\frac{ \eta_{1} \rho_{1} }{ 2 } \sum_{   \hat{k} \in  S_{k}   } \|s^{\hat{k}} \|^{2}
		\leq  \sum_{   \hat{k} \in  S_{k}   } \frac{ \eta_{1} \rho_{\hat{k}} }{ 2 } \|s^{\hat{k}} \|^{2} 
		&\leq \sum_{   \hat{k} \in  S_{k}   } \left( \Theta(z^{\hat{k}}) - \Theta(z^{\hat{k}+1}) \right) \\
		&= \sum_{ \hat{k} \leq k   } \left( \Theta(z^{\hat{k}}) - \Theta(z^{\hat{k}+1}) \right) \\
		&= \Theta(z^{1}) - \Theta(z^{k+1}) \\
		& \leq \Theta(z^{1}) - \Theta^{*}, 
	\end{align*}
	which yields \eqref{eq:ssk2_bound-r}.
	Together with \eqref{eq:2.7-r}, it implies that 
		\begin{align*}
			\sum_{ \hat{k} \leq k } \|s^{\hat{k}} \|^{2} 
			\leq \, & \left(  \left[ \left\lceil \log_{\eta_{2}} \left( \frac{\bar{K}_{F} + \sqrt{J} \bar{K}_{c} }{\rho_{1} ( 1-\eta_{1}  ) } \right) \right\rceil \right]_{+} +1  \right) \left( \sqrt{ 2 \rho_{1}^{-1} \Theta(z^{1})} + 2 \rho_{1}^{-1} K_{F} \right)^{2} \\
                &+ \frac{ 2 (\Theta(z^{1}) - \Theta^{*}) }{\eta_{1} \rho_{1} },
		\end{align*} 
	which further deduces $ \lim_{k \rightarrow \infty} s^{k}  = \mathbf{0} $, since its righthand side is independent of $k$.
   
\end{proof}

Note that $\rho$ defined in \eqref{eq:rhok_bound-r} is independent of $k$, since it is determined by the input parameters $ (\rho_{1},\eta_{1}, \eta_{2}) $ and the parameters $ ( \bar{K}_{F}, J,\bar{K}_{c} ) $ concerning \eqref{l1pen_r}. 
Define 
\begin{align*}
	\Phi(z,s):=  F(z ) + \nabla F (z )^{\top} s + H(z  + s ) + \sum_{j \in \mathcal{J}} | c_{j}( z  )  + \nabla c_{j}( z  ) ^{\top} s  |+ \frac{\rho }{2} \|s\|^{2}.
\end{align*}
Similar to $ \Phi_{k}(s) $, $ \Phi(z,s) $ is proper, continuous and level-bounded in $s$ for any $z$. 
Together with nonemptiness and closedness of $ \mathcal{F}(z) $ defined in \eqref{def:AF}, $ \argmin_{ s \in \mathcal{F}(z) } \Phi(z ,s) $ is nonempty and compact for any $ z $.  
Next we show that $ \mathbf{0} \in \argmin_{s \in \mathcal{F}(z) } \Phi(z,s)  $ and $ \Theta^{\prime}(z;d)\geq 0 $ for all $d$ at any accumulation point.
\begin{theorem}
	\label{th:2.1-r}
	For any accumulation point $ \bar{z} $ of the infinite sequence $ \{ z^{k} \} $ generated by Algorithm \ref{alg2}, it holds that
	\begin{align*}
		\Theta^{\prime}(\bar{z};d)\geq 0, \, \forall d.
	\end{align*}
\end{theorem}
\begin{proof}
	First we show $ \mathbf{0} \in \argmin_{s \in \mathcal{F}(\bar{z}) } \Phi(\bar{z},s)  $.
		Denote the subsequence converging to $ \bar{z} $ as $ \{ z^{k_{t}}, t \geq 1 \} $. 
		For any $ i \in \mathcal{A}(\bar{z}) = \{ i \in [I] \mid \bar{z} \in P_{i} \} $, we have $ \mathbf{d}(\bar{z} , P_{i} )=0 $. 
		By the continuity of $ \mathbf{d}( \cdot , P_{i} ) $  and $ z^{k_{t}} \rightarrow \bar{z} $, there exists $ \bar{t}_{i} \geq 1 $ such that $ \mathbf{d}( z^{k_{t}} , P_{i} ) \leq \delta $ for all $ t \geq \bar{t}_{i} $. 
		Define $ \bar{t}:= \max_{i \in \mathcal{A}(\bar{z})} \bar{t}_{i} < \infty $. Then we have $ \mathbf{d}( z^{k_{t}} , P_{i} ) \leq \delta $ for any $ t \geq \bar{t}  $ and any $ i \in \mathcal{A}(\bar{z}) $, which implies that $ \mathcal{A}(\bar{z}) \subseteq \mathcal{A}_{k_{t}}^{\delta}= \{ i \in [I] \mid \mathbf{d}( z^{k_{t}} , P_{i} ) \leq \delta  \} $ for any $ t \geq \bar{t}  $. 
		For any $ k \geq 1 $, it is similar to \eqref{eq:new2.3-r} that 
		\begin{align*}
			\Theta ( z^{k} + s^{k} )
			& \leq \Phi_{k}(s^{k}) + \frac{ \bar{K}_{F} + \sqrt{J} \bar{K}_{c} - \rho_{k} }{2} \|s^{k} \|^{2} \\
			& \leq \Phi_{k}(s ) + \frac{ \bar{K}_{F} + \sqrt{J} \bar{K}_{c} - \rho_{k} }{2} \|s^{k} \|^{2} 
		\end{align*}
		for all $ s \in  \{z-   z^{k} \mid z\in P_{i} \}   $ and all $ i \in \mathcal{A}_{k}^{\delta} $,
		where the last inequality follows from steps 2-4 of Algorithm \ref{alg2}.
		It implies that for any $ t \geq \bar{t} $ and any $ i \in \mathcal{A}(\bar{z}) \subseteq \mathcal{A}_{k_{t}}^{\delta} $, 
		\begin{align*}
			&\Theta ( z^{k_{t}} + s^{k_{t}} ) \\
			 \leq \, & \Phi_{k_{t}}(s ) + \frac{ \bar{K}_{F} + \sqrt{J} \bar{K}_{c} - \rho_{k_{t}} }{2} \|s^{k_{t}} \|^{2} \\
			= \, & F(z^{k_{t}}) + \nabla F (z^{k_{t}})^{\top} s + H (z^{k_{t}} + s ) + \sum_{j \in \mathcal{J}} | c_{j}( z^{k_{t}} )  + \nabla c_{j}( z^{k_{t}} ) ^{\top} s  | \\
       & + \frac{\rho_{k_{t}}}{2} \|s\|^{2} 
				+ \frac{ \bar{K}_{F} + \sqrt{J} \bar{K}_{c} - \rho_{k_{t}} }{2} \|s^{k_{t}} \|^{2}
		\end{align*}
		for all $ s \in  \{z-   z^{k_{t}} \mid z\in P_{i} \}  $.
		Hence, for any $ t \geq \bar{t} $ and any $ i \in \mathcal{A}(\bar{z})   $, 
		\begin{align*}
			&\Theta ( z^{k_{t}} + s^{k_{t}} ) \\
			 \leq   \, & F(z^{k_{t}}) + \nabla F (z^{k_{t}})^{\top} ( z- z^{k_{t}} ) + H (z ) + \sum_{j \in \mathcal{J}} | c_{j}( z^{k_{t}} )  + \nabla c_{j}( z^{k_{t}} ) ^{\top} ( z- z^{k_{t}} )  | \\
				& + \frac{\rho_{k_{t}}}{2} \| z- z^{k_{t}} \|^{2}+ \frac{ \bar{K}_{F} + \sqrt{J} \bar{K}_{c} - \rho_{k_{t}} }{2} \|s^{k_{t}} \|^{2} \\
			\leq   \, & F(z^{k_{t}}) + \nabla F (z^{k_{t}})^{\top} ( z- z^{k_{t}} ) + H (z ) + \sum_{j \in \mathcal{J}} | c_{j}( z^{k_{t}} )  + \nabla c_{j}( z^{k_{t}} ) ^{\top} ( z- z^{k_{t}} )  | \\
				& + \frac{\rho }{2} \| z- z^{k_{t}} \|^{2}+ \frac{ \bar{K}_{F} + \sqrt{J} \bar{K}_{c} - \rho_{k_{t}} }{2} \|s^{k_{t}} \|^{2}
		\end{align*}
		for all $ z \in P_{i} $, where the last inequality uses $ \rho_{k} \leq \rho $ from Corollary \ref{cor:2.1-r}.
		Take the limit $ t \rightarrow \infty $, we yield that for any $ i \in \mathcal{A}(\bar{z})   $ and any $ z \in P_{i} $, 
		\begin{align*}
			\Theta ( \bar{z} )  
			 \leq \, &    F(\bar{z}) + \nabla F (\bar{z})^{\top} ( z- \bar{z} ) + H (z ) \\
       & + \sum_{j \in \mathcal{J}} | c_{j}( \bar{z} )  + \nabla c_{j}( \bar{z} ) ^{\top} ( z- \bar{z} )  |+ \frac{\rho }{2} \| z- \bar{z} \|^{2} 
		\end{align*}
		by $ z^{k_{t}} \rightarrow \bar{z} $, $ s^{k_{t}} \rightarrow \mathbf{0} $ from Corollary \ref{cor:2.1-r}, the continuity of $ \Theta,F, \nabla F, c_{j}, \nabla c_{j}, |\cdot|, \|\cdot \|^{2} $, and the boundedness of $ \rho_{k} $ from Corollary \ref{cor:2.1-r}. 
		It implies that for any $ i \in \mathcal{A}(\bar{z})   $ and any $ s \in 
		\{ z -\bar{z} \mid z \in P_{i} \}$, 
		\begin{align*}
			\Theta ( \bar{z} )  
			 \leq \, &    F(\bar{z}) + \nabla F (\bar{z})^{\top} s + H (\bar{z} + s )  + \sum_{j \in \mathcal{J}} | c_{j}( \bar{z} )  + \nabla c_{j}( \bar{z} ) ^{\top} s  |+ \frac{\rho }{2} \| s \|^{2} \\
			 = \, & \Phi(\bar{z},s).
		\end{align*} 
		Together with the arbitrariness of $ i \in \mathcal{A}(\bar{z})   $ and $ s \in 
		\{ z -\bar{z} \mid z \in P_{i} \} $, we have $ \Theta(\bar{z}) 
		\leq \min_{s \in 
		\{ z -\bar{z} \mid z \in \cup_{i \in  \mathcal{A}(\bar{z})}  P_{i} \}
		} \Phi(\bar{z},s) 
		= \min_{s \in \mathcal{F}(\bar{z}) }\Phi(\bar{z},s) $. Combining it with $ \Theta(\bar{z})= \Phi(\bar{z}, \mathbf{0}) $ and $ \mathbf{0} \in \mathcal{F}(\bar{z}) $ from Lemma \ref{lem:intFd}, we yield $ \mathbf{0} \in \argmin_{s \in \mathcal{F}(\bar{z}) } \Phi(\bar{z},s)  $.

	Hence, similar to Lemma \ref{lem:2.1-r}, it follows from the optimality condition and \cite[Proposition 4.1.2]{CP-book} that for all $d \in   \mathcal{T}_{\mathcal{F}(\bar{z})}(\mathbf{0})   $, 
	\begin{align*}
		0 \leq \left( \Phi(\bar{z}, \cdot) \right)^{\prime}( \mathbf{0};d ) 
		= \Theta^{\prime}( \bar{z} ; d ),
	\end{align*}
	which completes the proof by $ \mathcal{T}_{\mathcal{F}(\bar{z})}(\mathbf{0}) = \R^{\bar{N}} $ from Lemma \ref{lem:intFd}.
   
\end{proof}

By the proof of Theorem \ref{th:2.1-r}, we can also figure out the reason why using $ \mathcal{A}^{\delta}(z^{k}) $ rather than $ \mathcal{A} (z^{k}) $ in Algorithm \ref{alg2}.
Essentially, we require $ \mathcal{A} (\bar{z}) $ to be covered by the $ k $th index set used in each subproblem when the $ z^{k} $ approaches $ \bar{z} $.
Take $ H(z):=|z| $ for all $ z \in \R $, where $ I=2 , P_{1}=\R_{+}, P_{2}=\R_{-}  $, as an example. For $ z^{k}:=1/k \rightarrow 0 =: \bar{z} $, $ \mathcal{A} (z^{k}) \equiv \{ 1 \} \not\supseteq \{1,2\} = \mathcal{A} (\bar{z}) $. 
In the next subsection, we could further illustrate the influence of $ \delta>0 $ by complexity bounds. 

Finally, we extend the convergence results to second-order d-stationarity under additional assumptions.
\begin{corollary}
	Let $ \{ \beta_{\ell}, \, \ell \in [L] \} $ satisfy \eqref{eq:br-threshold-2}.
	Set the initial point $ z^{1} $ satisfying $ \Theta(z^{1}) \leq \Theta(z^{0}) $ with $ z^{0} $ defined in \eqref{def:z0}.  
    Assuming that $g$ is convex.
	\begin{itemize}
        \item[(a)] If Algorithm \ref{alg2} terminates at $k$th iteration, then $ z^{k} $ is a second-order d-stationary point of \eqref{l1pen_r}.
        \item[(b)] If Algorithm \ref{alg2} does not terminate in a finite number of steps, then any accumulation point of the infinite sequence $ \{ z^{k} \} $ is a second-order d-stationary point of \eqref{l1pen_r}.
    \end{itemize} 
\end{corollary}
\begin{proof}
	According to Theorem \ref{th:2.1-r} and Lemmas \ref{lem:2.1-r}, any termination point $z$ after a finite number of iterations or accumulation point $z$ of the infinite sequence satisfies $ \Theta^{\prime}(z;d)\geq 0 $ for all $d$.
	Additionally, it follows from Lemma \ref{lem:2.2-r} and  the continuity of $ \Theta $ that any termination point $z$ after a finite number of iterations or accumulation point $z$ of the infinite sequence satisfies 
    \begin{align*}
        \Theta(z) \leq  \Theta(z^{1}) \leq \Theta(z^{0}).
    \end{align*}
    Together with Corollary \ref{cor:D-SD}, we yield the result.
   
\end{proof}

\subsection{Complexity bounds}
\label{subsec:1-d.2-r}
In this subsection, we give the complexity bounds for finding $\epsilon$-approximate first-order and second-order d-stationary points of \eqref{l1pen_r} defined in Definitions \ref{def:eps-1d} and \ref{def:eps-2d}, respectively. 
The following theorem provides the complexity bound for finding an $\epsilon$-approximate first-order d-stationary point.
\begin{theorem}
	\label{th:2.2-r}
	For any $ \epsilon >0 $, there exists $ k \leq  \bar{k}(\epsilon) $, where
	\begin{align}
		\label{def:kbar_eps-r}
		\bar{k}(\epsilon) := \left\lceil \kappa_{1} + 6  \gamma^{2} \eta_{1}^{-1} \rho_{1}^{-1} \left( \Theta(z^{1})  - \Theta^{*} \right) \epsilon^{-2} \right\rceil, 
	\end{align}  
	such that $ z^{k} $ is an $ \epsilon $-approximate d-stationary point for \eqref{l1pen_r}, where the positive integer $ \kappa_{1} $ is defined in Lemma \ref{lem:2.3-r} and
	 \begin{align*} 
		\gamma:= \max & \Biggl\{  
	 		\frac{ 2 \left( \Theta(z^{1}) - \Theta^{*} \right) }{ \left( \sqrt{ 2 \rho_{1}^{-1} \Theta(z^{1}) } + 2 \rho_{1}^{-1} K_{F} \right)^{2}} ,\,
			 \frac{ 8   \left( \Theta(z^{1}) - \Theta^{*} \right)  }{ \delta^{2} }  , \,\Biggr. \\
	 		&~ \Biggl.  
			\rho_{1} \eta_{2} + 2( \bar{K}_{F} + \sqrt{J} \bar{K}_{c} ), \, 
      \left( \frac{\eta_{2}^{2}}{1-\eta_{1}} +2 \right)( \bar{K}_{F} + \sqrt{J} \bar{K}_{c} ), \, 
			1
			\Biggr\} . 
	 \end{align*} 
\end{theorem} 
\begin{proof}
	For any $ \epsilon >0 $, we denote $ \bar{k}(\epsilon) $ as $ \bar{k} $ for brevity.  
	If $ \Theta(z^{1})  = \Theta^{*} $, then it follows from the optimality condition that $ z^{1} $ is a d-stationary point for \eqref{l1pen_r}. Thus, the conclusion holds naturally since $ \kappa_{1} \geq 1 $.

	If $ \Theta(z^{1})  > \Theta^{*} $, by the definition of $ \bar{k} $ and Lemma \ref{lem:2.3-r}, we have 
		\begin{align}
			\label{eq:2.11-r}
			| S_{\bar{k}} |
			= \bar{k} - | \bar{S}_{\bar{k}} | 
			\geq \bar{k} - \kappa_{1}
			= \left\lceil    6    \gamma^{2}  \eta_{1}^{-1} \rho_{1}^{-1} \left( \Theta(z^{1})  - \Theta^{*} \right) \epsilon^{-2}   \right\rceil  \geq 1,
		\end{align}
	which indicates $ S_{\bar{k}} \neq \emptyset $. Denote $ k \in  \argmin_{t \in S_{\bar{k}} }  \|s^{t} \|^{2}    $. Then it follows from \eqref{eq:ssk2_bound-r} that
	\begin{align*}
		\|s^{k} \|^{2} 
		= \min_{t \in S_{\bar{k}} }   \|s^{t} \|^{2}   
		\leq | S_{\bar{k}} |^{-1} \sum_{t \in S_{\bar{k}} } \|s^{t} \|^{2}   
		\leq  2 | S_{\bar{k}} |^{-1}   \eta_{1}^{-1} \rho_{1}^{-1}  \left( \Theta(z^{1})  - \Theta^{*} \right)  
		 \leq \frac{ \epsilon^{2} }{ 3   \gamma^{2}   },
	\end{align*} 
	where the last inequality is derived from \eqref{eq:2.11-r}. Taking the square root of both sides yields 
	\begin{align}
		\label{eq:(a)}
		\| s^{k} \| \leq \frac{\epsilon}{\sqrt{3} \gamma}  .
	\end{align}
	For $ k \in  \argmin_{t \in S_{\bar{k}} }  \|s^{t} \|^{2}   $, define 
		\begin{align}
      \label{def:yk}
			 y^{k+1} \in \argmin_{   y \in \cup_{ i \in \mathcal{A}_{k}^{\delta} } P_{i}   } \left( \Theta(y) + \frac{\gamma }{2} \| y - z^{k} \|^{2}  \right) , 
      ~ \Delta:= \left( \frac{2 \left( \Theta(z^{1}) - \Theta^{*} \right) }{ \gamma } \right)^{\frac{1}{2}}.
		\end{align}
		Since $ \cup_{ i \in \mathcal{A}_{k}^{\delta} } P_{i} $ is nonempty and closed, and $  ( \Theta(\cdot) + \frac{\gamma }{2} \| \cdot - z^{k} \|^{2}  ) $ is proper, continuous and level-bounded, we know that $ y^{k+1} $ is well-defined.  
		Then it follows from the definition of $ y^{k+1} $ and $ s^{k} \in \mathcal{F}_{k}^{\delta} = \{ z-z^{k} \mid z \in \cup_{ i \in \mathcal{A}_{k}^{\delta} } P_{i} \} $ that
		\begin{align}
			& \Theta(y^{k+1}) + \frac{ \gamma }{2} \| y^{k+1} - z^{k} \|^{2} \notag \\
			\leq \,& \Theta(z^{k} + s^{k} ) + \frac{ \gamma }{2} \|s^{k} \|^{2} \notag\\
			\leq \,& \Phi_{k}(s^{k} ) + \frac{ \gamma  - \rho_{k} + \bar{K}_{F} +\sqrt{J} \bar{K}_{c} }{2} \|s^{k} \|^{2} \notag \\
			\leq \,& \min_{  s \in \mathcal{F}_{k}^{\delta} ,  \|s\|\leq  \Delta }\Phi_{k}(s  ) + \frac{ \gamma  - \rho_{k} + \bar{K}_{F} +\sqrt{J} \bar{K}_{c} }{2} \|s^{k} \|^{2} \notag \\
			\leq \,& \min_{   s \in \mathcal{F}_{k}^{\delta} ,  \|s\|\leq  \Delta }\left( \Theta ( z^{k} + s  ) + \frac{ \rho_{k} + \bar{K}_{F} +\sqrt{J} \bar{K}_{c} }{2} \|s  \|^{2}\right) \notag \\ 
        &+ \frac{ \gamma  - \rho_{k} + \bar{K}_{F} +\sqrt{J} \bar{K}_{c} }{2} \|s^{k} \|^{2} \notag \\
			\leq \,&  \Theta ( y^{k+1}   ) + \frac{ \rho_{k} + \bar{K}_{F} +\sqrt{J} \bar{K}_{c} }{2} \|y^{k+1} - z^{k}  \|^{2} + \frac{ \gamma  - \rho_{k} + \bar{K}_{F} +\sqrt{J} \bar{K}_{c} }{2} \|s^{k} \|^{2}, \label{eq:2.12-r}
		\end{align}
		where the second inequality uses \eqref{eq:new2.3-r}, 
		the third inequality uses 
		$$ \Phi_{k}(s^{k} ) = \min_{   s \in \mathcal{F}_{k}^{\delta}   }\Phi_{k}(s  ) \leq \min_{   s \in \mathcal{F}_{k}^{\delta} ,  \|s\|\leq \Delta  }\Phi_{k}(s  ) $$ 
		from the definition of $ s^{k} $, 
		the fourth inequality applies \eqref{eq:Kbar-r}, Lemma \ref{lem:2.2-r} and 
		\begin{align*}
			 \Delta  
			&\leq \left( \frac{2 \left( \Theta(z^{1}) - \Theta^{*} \right) }{ 2 \left( \Theta(z^{1}) - \Theta^{*} \right) } \left( \sqrt{ 2 \rho_{1}^{-1} \Theta(z^{1}) } + 2 \rho_{1}^{-1} K_{F} \right)^{2} \right)^{\frac{1}{2}} \\
			&= \sqrt{ 2 \rho_{1}^{-1} \Theta(z^{1}) } + 2 \rho_{1}^{-1} K_{F}
		\end{align*}
		from definitions of $ ( \Delta ,  \gamma   ) $, 
		the last inequality uses $( y^{k+1} - z^{k}) \in \mathcal{F}_{k}^{\delta} $ and  
		\begin{align}
			\label{eq:ykzyD}
			\| y^{k+1} - z^{k} \| \leq \Delta
		\end{align}
		from
		$
			\Theta^{*}  + \frac{ \gamma }{2} \| y^{k+1} - z^{k} \|^{2}
			\leq \Theta( y^{k+1} )  + \frac{  \gamma  }{2} \| y^{k+1} - z^{k} \|^{2}
			\leq \Theta (z^{k})
			\leq \Theta (z^{1}) 
		$.
		Reorganizing \eqref{eq:2.12-r} and applying $   \gamma \geq \rho  + 2( \bar{K}_{F} +\sqrt{J} \bar{K}_{c})    \geq \rho_{k} + 2( \bar{K}_{F} +\sqrt{J} \bar{K}_{c}) $, we have
		\begin{align}
			\label{eq:2.13-r}
			\|y^{k+1} - z^{k} \|
			\leq \left( \frac{  \gamma  - \rho_{k} + \bar{K}_{F} +\sqrt{J} \bar{K}_{c} }{  \gamma  - \rho_{k} - \bar{K}_{F} -\sqrt{J} \bar{K}_{c} } \right)^{\frac{1}{2}} \|s^{k} \|   
			\leq \sqrt{3} \|s^{k} \| 
			\leq \frac{ \epsilon  }{   \gamma   }
			\leq \epsilon ,
		\end{align}
		where the third inequality is derived from result (a), the last inequality comes from $ \gamma\geq 1 $.
		The remaining is to show $ \Theta^{\prime}( y^{k+1} ;d ) \geq -\epsilon $ for all $d$ with $ \|d\|=1 $. It follows from the optimality condition of $ y^{k+1} $ \cite[Lemma 3.1]{JWC2025} 
        that for any $ d \in   \mathcal{T}_{ \cup_{ i \in \mathcal{A}_{k}^{\delta} } P_{i} } ( y^{k+1} )   $, 
		\begin{align}
			\label{eq:4.15}
			0 \leq \Theta^{\prime}( y^{k+1} ; d ) +  \gamma  ( y^{k+1} - z^{k} )^{\top} d.
		\end{align}
			We can simplify $ \mathcal{T}_{ \cup_{ i \in \mathcal{A}_{k}^{\delta} } P_{i} } ( y^{k+1} )    $ as $ \R^{\bar{N}} $
			by proving $ y^{k+1} \in   int( \cup_{i \in \mathcal{A}_{k}^{\delta} } P_{i} )  $ as follows. 
			For any $ z \in \mathbb{B}( z^{k} ; \delta ) $, it follows from the definition of $ \{ P_{i}, \, i \in [I] \} $ that there exists $ \bar{i} \in [I] $ such that $ z \in P_{\bar{i}} $. 
			Thus, we have $ \mathbf{d}( z^{k} , P_{\bar{i}} ) \leq \| z^{k} - z \| \leq \delta  $, which implies that $ \bar{i} \in \mathcal{A}_{k}^{\delta}= \{ i \in [I] \mid \mathbf{d}( z^{k} , P_{i} ) \leq \delta \} $ and $ z \in   \cup_{ i \in \mathcal{A}_{k}^{\delta} } P_{i}  $.
			Together with the arbitrariness of $ z \in \mathbb{B}( z^{k} ; \delta ) $, it indicates that $ \mathbb{B}( z^{k} ; \delta ) \subseteq \cup_{ i \in \mathcal{A}_{k}^{\delta} } P_{i} $.
			On the other hand, it follows from \eqref{eq:ykzyD} and the definitions of $ \Delta, \gamma $ that
			\begin{align*}
				\| y^{k+1} - z^{k} \| 
				\leq \left( \frac{2 \left( \Theta(z^{1}) - \Theta^{*} \right) }{ \gamma } \right)^{\frac{1}{2}}  
				\leq \left( \frac{2 \delta^{2} \left( \Theta(z^{1}) - \Theta^{*} \right) }{  8   \left( \Theta(z^{1}) - \Theta^{*} \right) } \right)^{\frac{1}{2}} 
				= \frac{\delta}{2 },
			\end{align*}
			which implies that for any $ z\in \mathbb{B}( y^{k+1} ; \frac{\delta}{2  } ) $, 
			\begin{align*}
				\|z-z^{k} \| 
				\leq \|z-y^{k+1} \| + \| y^{k+1} - z^{k} \| 
				\leq \frac{\delta}{2  } + \frac{\delta}{2 } 
				= \delta ,   
			\end{align*}
			i.e., $ \mathbb{B}( y^{k+1} ; \frac{\delta}{2  } ) \subseteq  \mathbb{B}( z^{k} ; \delta ) $.
			Together with $ \mathbb{B}( z^{k} ; \delta ) \subseteq \cup_{ i \in \mathcal{A}_{k}^{\delta} } P_{i} $, 
			it yields that $ \mathbb{B}( y^{k+1} ; \frac{\delta}{2 } ) \subseteq \cup_{ i \in \mathcal{A}_{k}^{\delta} } P_{i} $, which implies that 
			$$ 
				\mathcal{T}_{ \cup_{ i \in \mathcal{A}_{k}^{\delta} } P_{i} } ( y^{k+1} )    
				= \R^{\bar{N}} .$$
			Applying it on \eqref{eq:4.15}, we yield $ 0 \leq \Theta^{\prime}( y^{k+1} ; d ) +  \gamma  ( y^{k+1} - z^{k} )^{\top} d $ for all $ d \in \R^{\bar{N}} $. 
		Together with the second to last inequality in \eqref{eq:2.13-r}, 
		it implies that for any $d$ with $ \|d\|=1 $, 
		\begin{align}
      \label{eq:op-res-2}
			\Theta^{\prime}( y^{k+1} ; d ) 
			\geq -  \gamma  ( y^{k+1} - z^{k} )^{\top} d
			\geq - \gamma  \| y^{k+1} - z^{k} \|
			\geq -\epsilon ,
		\end{align}
		which yields the result. 
   
\end{proof}
  
Theorem \ref{th:2.2-r} indicates that for any $ \epsilon>0 $, to obtain an $ \epsilon $-approximate d-stationary point of \eqref{l1pen_r}, the number of iterations is no more than 
\begin{align*}
	\bar{k} \left( \epsilon  \right)
	 = \left\lceil \kappa_{1} +   6  \gamma^{2} \eta_{1}^{-1} \rho_{1}^{-1}   \left( \Theta(z^{1})  - \Theta^{*} \right) \epsilon^{-2} \right\rceil   
	 = \mathcal{O}(\kappa_{1} + \epsilon^{-2}   )  ,
\end{align*}
and the number of successful iterations is no more than 
\begin{align}
	\label{eq:com_suc-r}
	\left\lceil     6  \gamma^{2} \eta_{1}^{-1} \rho_{1}^{-1}   \left( \Theta(z^{1})  - \Theta^{*} \right) \epsilon^{-2} \right\rceil   
	 =  \mathcal{O}(  \epsilon^{-2}   )  ,
\end{align}
which is consistent with the existing results \cite{CGT2011,GST2021} for the case where $ H $ is convex.  
Notice that $ \gamma =\mathcal{O}(\delta^{-2}) $ for sufficiently small $ \delta>0 $. A small $ \delta $ reduces the computational cost per iteration while increasing the number of iterations required. 
Hence, Theorem \ref{th:2.2-r} demonstrates that $ \lim_{\delta \downarrow 0} \bar{k}  ( \epsilon   ) = \infty $ for any $ \epsilon>0 $, which also explains from the perspective of complexity why using $ \mathcal{A}^{\delta}(z^{k}) $ rather than $ \mathcal{A} (z^{k}) $ (i.e., $ \mathcal{A}^{\delta}(z^{k}) $ with $ \delta=0 $) in Algorithm \ref{alg2}.   

Next, we extend the convergence results to $ \epsilon $-approximate second-order d-stationarity under additional assumptions.
\begin{corollary}
	\label{cor:2.4-r}
	Set $ \{ \beta_{\ell} >0, \ell \in [L] \} $ satisfying \eqref{eq:br-threshold-3} and $z^{1} \in lev_{\leq \Theta(z^{0}) } \Theta  $. 
	Assume that $ g $ is convex. 
	For any $ \epsilon \in (0,1) $, there exists $ k \leq \bar{k}(\epsilon) $, where $ \bar{k}(\epsilon) $ is defined in \eqref{def:kbar_eps-r}, 
	such that $ z^{k} $ is an $ \epsilon $-approximate second-order d-stationary point for \eqref{l1pen_r}.
\end{corollary}
\begin{proof}
	It follows from Theorem \ref{th:2.2-r} that there exists $ k \leq  \bar{k}(\epsilon) $ such that $ \|z^{k} - y^{k+1} \| \leq \epsilon $ and $ \Theta^{\prime}(y^{k+1} ;d ) \geq - \epsilon $ for any $d$ with $ \|d\|=1 $, where 
    \begin{align*}
        y^{k+1} \in \argmin_{   y \in \cup_{ i \in \mathcal{A}_{k}^{\delta} } P_{i}    } \left( \Theta(y) + \frac{\gamma}{2} \|y-z^{k} \|^{2} \right) 
    \end{align*}
    with $ \gamma $ defined in Theorem \ref{th:2.2-r}. 
    The definition of $ y^{k+1} $ and $ z^{k} \in \cup_{ i \in \mathcal{A}_{k}^{\delta} } P_{i} $ imply that 
	\begin{align*}
		\Theta( y^{k+1} ) 
		\leq \Theta(y^{k+1}) + \frac{\gamma}{2} \|y^{k+1}-z^{k} \|^{2}
		\leq \Theta(z^{k})  
		\leq \Theta(z^{1})
		\leq \Theta(z^{0}),
	\end{align*}
	where the third inequality follows from Lemma \ref{lem:2.2-r}, the last inequality uses $z^{1} \in lev_{\leq \Theta(z^{0}) } \Theta  $.
	Together with Theorem \ref{thm:eD-eSD} and $ \Theta^{\prime}(y^{k+1} ;d ) \geq - \epsilon $ for any $d$ with $ \|d\|=1 $, it implies that $ \Theta^{(2)}(y^{k+1};d ) \geq -\epsilon $ for any $ d $ with $ \|d\|=1 $ and $ \Theta^{\prime}(y^{k+1} ;d ) = - \epsilon $, 
	which completes the proof.
   
\end{proof}

\section{Numerical experiments}
\label{sec:4.4} 
This section carries out numerical experiments for training the Elman RNN with a single unidirectional hidden layer and Tikhonov regularizers 
\begin{equation} 
    \label{multicom-RNN-test}
    \begin{aligned} 
        \min_{\theta} ~&  \frac{1}{NT} \sum_{i=1}^{N} \sum_{t=1}^{T} \left\| y_{t}^{i} - \left(A \sigma\left( W (\dots \sigma( V x_{1}^{i} +b ) \dots) + V x_{t}^{i} +b \right) + c \right)\right\|^{2}  \\
        &+  \lambda_{1}\|A\|_{\rm F}^{2} + \lambda_{2}\|W\|_{\rm F}^{2} +\lambda_{3}\|V\|_{\rm F}^{2} + \lambda_{4} \|b\|^{2} + \lambda_{5} \|c\|^{2}  ,
    \end{aligned}
\end{equation}
where $  \lambda_{1},\lambda_{2},\lambda_{3},\lambda_{4},\lambda_{5} \in \R_{++} $, the activator $ \sigma(\cdot) = [\cdot]_{+} $, the variables 
\begin{align*}
    &\theta:=  \left( \operatorname{vec}(A)^{\top} ,\operatorname{vec}(W)^{\top} ,\operatorname{vec}(V)^{\top} ,b^{\top} ,c^{\top}\right) ^{\top} \\
    &\text{with } A \in \R^{N_{y} \times N_{h} }, \, W \in \R^{N_{h} \times N_{h} }, \, V \in \R^{N_{h} \times N_{x} }, \, b \in \R^{N_{h}}, \,  c \in \R^{N_{y}}, 
\end{align*}
the dataset $ \{ (x_{t}^{i} \in \R^{N_{x}} , y_{t}^{i} \in \R^{N_{y}} ), \, i\in [N], t \in [T] \}   $, and $ N, T,N_{x},N_{h},N_{y} \in \mathbb{Z}_{++}  $. 

By introducing auxiliary variables $ h_{t}^{i}, u_{t}^{i} \in \R^{N_{h}}$ and $ v_{t}^{i} \in \R^{N_{y}} $ for all $ i \in [N], t\in [T] $, we obtain
the following constrained form 
\begin{equation*} 
    \begin{aligned} 
        \min_{z}~ & F(z) :=  \frac{\|v-y \|^{2}}{N T } +  \lambda_{1}\|A\|_{\rm F}^{2} + \lambda_{2}\|W\|_{\rm F}^{2} +\lambda_{3}\|V\|_{\rm F}^{2} + \lambda_{4} \|b\|^{2} + \lambda_{5} \|c\|^{2}  
         \\
        s.t. ~ & \bar{c}_{t}^{i}(z):= u_{t}^{i}  - W h_{t-1}^{i}  - V x_{t}^{i}  - b = \mathbf{0}, \\
        &   \hat{c}_{t}^{i}(z):= h_{t}^{i}  -  [ u_{t}^{i} ]_{+}  = \mathbf{0}   \\
        &   \tilde{c}_{t}^{i}(z):= v_{t}^{i} - A h_{t }^{i}  - c = \mathbf{0} , \\ 
        & \forall i \in [N], \, t \in [T],
    \end{aligned}
\end{equation*}
where the variables 
\begin{align*}
    &z=(\theta^{\top}, u^{\top},h^{\top}, v^{\top} )^{\top} \in \R^{\bar{N}} \text{ with }  \\
	&u=((u^{1})^{\top}, \dots, (u^{N})^{\top})^{\top},~ u^{i}=((u_{1}^{i})^{\top}, \dots, (u_{T}^{i})^{\top})^{\top} \text{ for all } i\in [N],  \\
	&h=((h^{1})^{\top}, \dots, (h^{N})^{\top})^{\top},~ h^{i}=((h_{1}^{i})^{\top}, \dots, (h_{T}^{i})^{\top})^{\top} \text{ for all } i\in [N],  \\
	&v=((v^{1})^{\top}, \dots, (v^{N})^{\top})^{\top},~ v^{i}=((v_{1}^{i})^{\top}, \dots, (v_{T}^{i})^{\top})^{\top} \text{ for all } i\in [N],   \\
	&\bar{N}:= N_{h}(N_{x} + N_{h}+ N_{y} ) + N_{h} +N_{y} + N T(2N_{h}+N_{y}),  
\end{align*}
the label $ y=((y^{1})^{\top}, \dots, (y^{N})^{\top})^{\top}$ with $ y^{i}=((y_{1}^{i})^{\top}, \dots, (y_{T}^{i})^{\top})^{\top} $ for all $ i\in [N] $, and $ h_{0}^{i}=\mathbf{0} $ for all $ i \in [N] $. 
Correspondingly, the $ \ell_{1} $-penalized reformulation of \eqref{multicom-RNN-test} is defined as 
\begin{equation}
    \label{RNN-test}
    \begin{aligned} 
        \min_{z} ~\Theta(z) :=  F(z) + \beta_{1} \left( \| \bar{c}(z) \|_{1} + \| \hat{c}(z) \|_{1}  \right) + \beta_{2} \| \tilde{c}(z) \|_{1}  ,
    \end{aligned}
\end{equation} 
where $  \beta_{1}, \beta_{2} \in \R_{++} $, and $ \bar{c}:= ( (\bar{c}_{1}^{1})^{\top} , \dots, (\bar{c}_{T}^{N})^{\top} )^{\top} $, $ \hat{c}:= ( (\hat{c}_{1}^{1})^{\top} , \dots, (\hat{c}_{T}^{N})^{\top} )^{\top} $, $ \tilde{c}:= ( (\tilde{c}_{1}^{1})^{\top} , \dots, (\tilde{c}_{T}^{N})^{\top} )^{\top} $.  

We test and compare the performance of Algorithm \ref{alg2} (SRPA) on \eqref{RNN-test}, and the state-of-the-art (stochastic) gradient descent-based algorithms on \eqref{multicom-RNN-test}. 
The performance is evaluated by training error and test error at each iteration point $ \theta $ or the component $ \theta $ from iteration point $ z= (\theta^{\top}, u^{\top},h^{\top},v^{\top} )^{\top} $ defined as follows. 
For the case where $ N=1 $, i.e., a dataset of a single sequence, we remove the superscripts in $ x_{t}^{i} $ and $ y_{t}^{i}  $ for simplicity and define 
{\small \begin{align*}
	&\textbf{TrainErr}(\theta):= \frac{1}{T} \sum_{t=1}^{T} \| y_{t} - A \left( W (\cdots [Vx_{1} + b ]_{+} \cdots ) + V x_{t} +b \right) -c \|^{2} ,   \\
	&\textbf{TestErr}(\theta):= \frac{1}{T_{test}} \sum_{t=T+1}^{T+T_{test}} \| y_{t} - A \left( W (\cdots [V x_{1} + b ]_{+} \cdots ) + V x_{t} +b \right) -c \|^{2}.  
\end{align*}}  
For the case where $ N>1 $, i.e., a dataset with multiple sequences, we define 
{\small \begin{align*}
	&\textbf{TrainErr}(\theta):= \frac{1}{N T} \sum_{i=1}^{N}\sum_{t=1}^{T} \| y_{t}^{i} - A \left( W (\cdots [Vx_{1}^{i} + b ]_{+} \cdots ) + V x_{t}^{i} +b \right) -c \|^{2} ,   \\
	&\textbf{TestErr}(\theta):= \frac{1}{N_{test} T} \sum_{i=N+1}^{N+N_{test}}   \sum_{t=1}^{T} \| y_{t}^{i} - A \left( W (\cdots [Vx_{1}^{i} + b ]_{+} \cdots ) + V x_{t}^{i} +b \right) -c \|^{2}  .  
\end{align*}}

\subsection{Datasets and coefficients}
\label{sec:4.4.1}
Experiments are conducted on a synthetic dataset ($N=1$), the Volatility of S\&P Index dataset \cite{WZC2024} ($N=1$), and the TIMIT dataset \cite{garofolo1993darpa} ($ N>1 $).

The synthetic training dataset $ \{ (x_{t}   , y_{t}   ), t \in [T] \}   $ and test dataset $ \{ (x_{t}   , y_{t}   ), \\ t \in [T+T_{test}]\backslash [T] \}   $ are randomly generated according to Section 5.1 of \cite{WZC2024}. Specifically, we set $ T=8 , T_{test}=2, N_{x}=5,N_{h}=4,N_{y}=3 $ and $ \bar{N}=143 $. Then we randomly generate matrices $ \hat{A},\hat{W},\hat{V},\hat{b},\hat{c} $ from a normal distribution with mean 0 and variance 0.8, inputs $ \{ x_{t}, t \in [T+T_{test}] \}  $ from a uniform distribution on the interval $[-1,1]$, and noises $   \{ e_{t}, t \in [T+T_{test}] \} $ from a normal distribution with mean 0 and variance $10^{-3}$. The labels $ \{ y_{t}, t \in [T+T_{test}] \}  $ are subsequently generated by 
\begin{align*}
    \hat{y}_{t} = \hat{A} \sigma\left( \hat{W} (\dots \sigma( \hat{V}  x_{1}  +\hat{b} ) \dots) + \hat{V} x_{t} +\hat{b} \right) +\hat{c} + e_{t}
\end{align*}
and the standardization on $ \{ [\hat{y}_{t}]_{i}, t \in [T+T_{test}] \}  $ for all $ i \in [N_{y}] $. 
Experiments for the synthetic dataset are conducted in Python 3.10 on a laptop of 16GB RAM and AMD Ryzen 7 5800U with Radeon Graphics (1.90 GHz).

The dataset, Volatility of S\&P Index, is constructed according to Section 5.1 of \cite{WZC2024}. 
Specifically, $ y_{t} $ and $ x_{t} $ are the monthly realized volatility of the S\&P index and exogenous factors, respectively, where $ T=393, T_{test}=44,N_{x}=11,N_{h}=20,N_{y}=1 $ and $ \bar{N}=16774  $. 
Experiments for the Volatility of S\&P Index dataset are conducted in Python 3.10 on a laptop of 16GB RAM and AMD Ryzen 7 5800U with Radeon Graphics (1.90 GHz). 

The TIMIT dataset is widely used in  audio denoising tasks. We select and construct $ \{ (x_{t}^{i} \in N_{x} , y_{t}^{i} \in N_{y} ), \, i\in [N+N_{test}], t \in [T] \}   $ according to \cite[Section 3.2.2]{wang2025thesis}. Specifically, $ x_{t}^{i} $ and $y_{t}^{i} $ are the spectrograms of the audio with noise and the clean audio, respectively, where $ N=49, N_{test}=21, T=126, N_{x}=N_{y}=129, N_{h}=2  $ and $ \bar{N}=821793 $.
All the implementations are conducted in Python 3.10 on a laptop of 16GB RAM and AMD Ryzen 7 5800U with Radeon Graphics (1.90 GHz), a high-performance server (2 AMD EPYC 7763 CPUs and 1024GB RAM) at the Department of Applied Mathematics and a GPU enabled virtual machine (16 vCPU and 96GiB vRAM) at UBDA of the Hong Kong Polytechnic University.

Regarding regularization coefficients, we set
\begin{align*}
    \lambda_{1} = \frac{\tau}{N_{y} N_{h}}  , \,
    \lambda_{2} = \frac{\tau}{N_{h}^{2}}, \,
    \lambda_{3} = \frac{\tau}{ N_{h} N_{x} }, \,
    \lambda_{4} = \frac{\tau}{ N_{h} }, \,
    \lambda_{5} = \frac{\tau}{N_{y}}, 
\end{align*}
with 
\begin{itemize}
	\item $ \tau=1.2 $ for the synthetic dataset according to \cite{WZC2024}, 
	\item $ \tau=1 $ for the Volatility of S\&P Index dataset according to \cite{WZC2024},
	\item and $ \tau=5 \times 10^{-4} $ for the TIMIT dataset according to \cite{wang2025thesis}, 
\end{itemize}
respectively.

For penalty parameters $ \beta_{1} $ and $ \beta_{2} $, we set $ \beta_{1} = \beta_{2} =1, 0.1 $, and $0.04$ for the synthetic, Volatility of S\&P Index, and TIMIT datasets, respectively. 
    As shown in Section \ref{sec:4.4.3}, feasibility violation is controlled within acceptable levels under these penalty parameters.  

\subsection{Subalgorithm}
\label{sec:4.4.2}
This subsection elaborates on the algorithm for the subproblems of SRPA. 
When SRPA is applied on \eqref{RNN-test}, we have $ H(z)=\beta_{1} \| h-[u]_{+} \|_{1} $.
According to Remark \ref{remark:subp}, steps 3 and 4 of SRPA at $k$th iteration is equivalent to calculating the strongly convex optimization problem
\begin{align*} 
    \min_{d \in \R^{\bar{N}}}\, \Phi_{k}(d), ~ \text{s.t. } \nu_{j} \, [u^{k} + d_{u} ]_{j} \geq 0, \, j \in [N T N_{h}],
\end{align*}
for all indicators $ \boldsymbol{\nu} \in \R^{N T N_{h}}  $ with the components
\begin{align*} 
    \nu_{j}  
    = \left\{ \begin{array}{ll}
        1, & \text{if } [u^{k}]_{j} >\delta,\\
        -1, & \text{if } [u^{k}]_{j} <-\delta,\\
        1 \text{ or } -1 , & \text{if } | [u^{k}]_{j} | \leq \delta,\\
    \end{array}\right.
    ~ \forall  j \in [T N_{h}], 
\end{align*}
and finding the one solution at which $ \Phi_{k} $ has the minimum value. 
Here, 
\begin{align*}
	\Phi_{k}(d)
    = \, & F(z^{k}) + \nabla F(z^{k})^{\top} d + \frac{\rho_{k}}{2} \|d\|^{2}
      + \beta_{1} \left\| h^{k} + d_{h} - [u^{k} + d_{u} ]_{+} \right\|_{1}   \\
    & + \beta_{1}  \left\| \bar{c}(z^{k}) + \nabla \bar{c}(z^{k}) ^{\top} d \right\|_{1} 
      + \beta_{2}  \left\|  \tilde{c}(z^{k}) + \nabla \tilde{c}(z^{k}) ^{\top} d     \right\|_{1} 
\end{align*}  
with $ \bar{c} $ and $ \tilde{c} $ defined as in \eqref{RNN-test}. The variable $d \in \R^{\bar{N}} $ is divided into blocks as follows.
\begin{align*}
    &d=(d_{\theta}^{\top},  d_{u}^{\top}, d_{h}^{\top} , d_{v}^{\top} )^{\top}    \text{ with } 
      d_{\theta}=\left(d_{A}^{\top} ,d_{W}^{\top} ,d_{V}^{\top} ,d_{b}^{\top} , d_{c}^{\top}  \right) ^{\top}  , \\
        &d_{u}=(d_{u^{1}}^{\top}, \dots, d_{u^{N}}^{\top})^{\top}, \, d_{u^{i}} = (d_{u_{1}^{i}}^{\top}, \dots, d_{u_{N}^{i}}^{\top})^{\top} \text{ for all } i\in [N],   \\
        &d_{h}=(d_{h^{1}}^{\top}, \dots, d_{h^{N}}^{\top})^{\top}, \, d_{h^{i}} = (d_{h_{1}^{i}}^{\top}, \dots, d_{h_{N}^{i}}^{\top})^{\top} \text{ for all } i\in [N],   \\
        &d_{v}=(d_{v^{1}}^{\top}, \dots, d_{v^{N}}^{\top})^{\top}, \, d_{v^{i}} = (d_{v_{1}^{i}}^{\top}, \dots, d_{v_{N}^{i}}^{\top})^{\top} \text{ for all } i\in [N].   
\end{align*}
All dimensions are consistent with the corresponding components in $z$. 
In practice, we set $ \delta=10^{-15} $ and apply cvxpy solver to calculate a strongly convex optimization problem 
\begin{align}
    \label{subp_i}
    \min_{d \in \R^{\bar{N}}}\, \Phi_{k}(d), ~ \text{s.t. } \bar{\nu}_{k,j}  \, [u^{k} + d_{u} ]_{j} \geq 0, \, j \in [N T N_{h}],
\end{align}
at $k$th iteration with $\bar{\boldsymbol{\nu}}_{k} \in \R^{N T N_{h}}$ defined as
\begin{align}
	\label{def:nu}
    \bar{\nu}_{k,j} =  
    \left\{ \begin{array}{ll}
        1, & \text{if } [u^{k}]_{j} >\delta,\\
        -1, & \text{if } [u^{k}]_{j} <-\delta,\\
        X_{k,j} , & \text{if } | [u^{k}]_{j} | \leq \delta,\\
    \end{array}\right.
    ~ \forall  j \in [T N_{h}], 
\end{align}
where $ P(X_{k,j} = 1) = (\delta + [u^{k}]_{j})/(2 \delta), P(X_{k,j} = -1) = (\delta - [u^{k}]_{j})/(2 \delta) $.

\subsection{Feasibility and Optimality Measures}
\label{sec:4.4.3}

Since SRPA is applied on \eqref{RNN-test}, which is the $\ell_{1}$-penalized reformulation of \eqref{multicom-RNN-test}, this subsection investigates the feasibility violation 
\begin{equation}
    \label{eq:FeasVi}
    \begin{aligned}
      & \textbf{FeasVi}:=\max\{\textbf{FeasVi$_h$}, \textbf{FeasVi$_u$}, \textbf{FeasVi$_v$}  \}, \text{ with } \\
        &\textbf{FeasVi$_h$} = \frac{\left\| \hat{c}(z) \right\|}{NT} ,  
        \textbf{FeasVi$_u$} = \frac{\left\| \bar{c}(z) \right\|}{NT} , 
        \textbf{FeasVi$_v$} =\frac{\left\| \tilde{c}(z) \right\|}{NT}  
    \end{aligned}
\end{equation} 
at $ z=z^{k} $ for the sequence $ \{ z^{k} \} $ generated by SRPA, where $ \bar{c}, \hat{c} $ and $ \tilde{c} $ are defined as in \eqref{RNN-test}.  
To assess optimality, we investigate the objective function value $ \Theta(z^{k}) $ and the optimality residual 
\begin{align}
  \label{eq:res-k}
  \max\left\{ \left[-\min_{\|d\|=1} \Theta^{\prime}(y^{k+1} ; d)\right]_{+} , \left\| y^{k+1} - z^{k} \right\| \right\}
\end{align}
with $ y^{k+1} $ defined in \eqref{def:yk}.
However, $ y^{k+1} $ is hard to obtain. Even if $ y^{k+1} $ is available, the exact value of the optimality residual is difficult to compute since $ \Theta^{\prime}(y^{k+1} ; \cdot) $ is possibly nonconvex and the spherical surface $ \{d \mid \|d\|=1\} $ is nonconvex.
Fortunately, it follows from \eqref{eq:2.13-r} and \eqref{eq:op-res-2} that \eqref{eq:res-k} is in the order of $  \|s^{k} \|   $.
Thus, we use $  \|s^{k} \|   $ as a substitute measure of the optimality residual.

Before demonstrating the changes of feasibility violation and optimality measures with respect to the number of iterations, we tune the parameters $ \rho_{1},\eta_{1},\eta_{2} $ and search for a suitable initialization $ z^{1} $ for SRPA. 
Note that k-fold and leave-P-out cross validations are only applicable to independent and identically distributed samples. We fix the validation dataset as the test dataset rather than using cross validation when tuning parameters. 
We set the maximum number of iterations as $50$, $100$ and $100$ when tuning parameters over the synthetic, Volatility of S\&P Index, and TIMIT datasets, respectively. 

For the synthetic dataset, the values of $ \rho_{1},\eta_{1},\eta_{2} $ are searched from 
\begin{align*}
    \{ 0.5, 1, 2 \}
    , \,
    \{ 0.9, 0.99, 0.999 \}
    , \,
    \{ 1.1, 1.2, 1.3 \},
\end{align*}
respectively. 
The best initialization strategy for the initial values of ${A}^1$, ${W}^1$, ${V}^1$ is searched from 
\begin{itemize}
    \item random normal initialization \cite{bengio2009learning} with zero mean and standard deviations of $ \{ 10^{-3}, 10^{-2}, 10^{-1} \} $, 
    \item He initialization \cite{he2015delving},
    \item Glorot initialization \cite{glorot2010understanding}, 
    \item and LeCun initialization \cite{klambauer2017self}.
\end{itemize}
The initial values of $b^1$ and $c^1$ are set to $ \mathbf{0} $ according to \cite[p. 305]{goodfellow2016deep}. The initial values of $ \{ (u_{t}^{1}, h_{t}^{1}, v_{t}^{1} ), \, t \in [T] \} $ are set as
\begin{align}
    \label{eq:ini_u1}
    u_{t}^{1}:=    W^{1} h_{t-1}^{1}  + V^{1} x_{t}  + b^{1}, \, 
    h_{t}^{1}:= [ u_{t}^{1} ]_{+} , \, 
    v_{t}^{1}:=   A^{1} h_{t }^{1}  + c^{1} , \, 
        \forall t\in [T]
\end{align}
with $ h_{0}^{1} = \mathbf{0} $, which ensures the feasibility of the initial point $ z^{1} $.

For the Volatility of S\&P Index dataset, the values of $ \rho_{1},\eta_{1},\eta_{2} $ are searched from 
\begin{align*}
    \{ 0.01, 0.03, 0.05 \}
    , \,
    \{ 0.7, 0.8, 0.9 \}
    , \,
    \{ 1.1, 1.2, 1.3 \},
\end{align*}
respectively.
The initialization strategy for the initial values of ${A}^1$, ${W}^1$, ${V}^1$ is searched from random normal initialization with zero mean and standard deviations of $ \{ 10^{-3}, 10^{-2}, 10^{-1} \} $, Glorot initialization and LeCun initialization.
The initialization strategy for $ b^{1},c^{1}, \{ ( u_{t}^{1}, h_{t}^{1}, v_{t}^{1} ), \, t \in [T]  \} $ is the same with the one used in the synthetic dataset.
He initialization is no longer considered because the initial value $ \Theta (z^1)=2.975\times 10^{22} \gg \Theta(z^{0}) \approx 0.9364 $ under He initialization, which does not comply with our setting in Corollary \ref{cor:D-SD}.

For the TIMIT dataset, the values of $ \rho_{1}, \eta_{1}, \eta_{2} $ are searched from  
\begin{align*}
    \{ 0.001, 0.003, 0.005 \}
    , \,
    \{ 0.7, 0.8, 0.9 \}
    , \,
    \{ 1.1, 1.3, 1.5 \},
\end{align*}
respectively.
Based on \cite{wang2025thesis}, the initialization strategy for the initial values of ${A}^1$, ${W}^1$, ${V}^1$ is searched from random normal initialization with zero mean and standard deviations of $ \{ 10^{-3}, 10^{-2}, 10^{-1} \} $, and the initialization strategy for $ b^{1},c^{1}, \{ ( (u_{t}^{i})^{1}, (h_{t}^{i})^{1}, (v_{t}^{i})^{1} ), \, t \in [T] ,  i \in [N] \} $ is the same with the one used in the synthetic dataset.

Since these initialization strategies are random, we evaluate the performance using the averaged test error over 10, 3 and 3 repetitions for the synthetic, Volatility of S\&P Index and TIMIT datasets, respectively. 
The means and the standard deviations of test errors are listed in Tables \ref{tab:SRPA-tune-1}-\ref{tab:SRPA-tune-3} of Appendix \ref{sec:A1}, where each best test error is highlighted in bold. 
Based on these results, we set the parameters and initialization strategies as listed in Table \ref{tab:set_SRPA}. 
\begin{table}[htbp]
\centering
\caption{Parameters, initialization strategies and the number of iterations of SRPA.}
\label{tab:set_SRPA}

\resizebox{0.7\textwidth}{!}{%
\begin{tabular}{l c c c c}
    \toprule
    Datasets 
    & $\rho_{1}$ 
    & $(\eta_{1}, \eta_{2})$ 
    & Init. strategy 
    & No. of iters. \\ 
    \midrule
    Synthetic 
    & 0.5 
    & (0.9, 1.3) 
    & $\mathcal{N}(0, 10^{-1})$ 
    & 100 \\ 
    Volatility of S\&P Index 
    & 0.03 
    & (0.7, 1.1) 
    & Glorot 
    & 1000 \\ 
    TIMIT 
    & 0.003 
    & (0.7, 1.1) 
    & $\mathcal{N}(0, 10^{-2})$ 
    & 600 \\ 
    \bottomrule
\end{tabular}
} 

\end{table}

Figures \ref{fig:feas_res_syn}, \ref{fig:feas_res_SP500}, and \ref{fig:feas_res_TIMIT} correspond to the synthetic, Volatility of S\&P Index, and TIMIT datasets, respectively. Each figure shows the changes of feasibility violation \eqref{eq:FeasVi}, the optimality residual $ \|s^{k} \| $, and the objective function value $ \Theta(z^{k}) $ with respect to the number of iterations. 
Since the initial point $ z^{1} $ is feasible, \textbf{FeasVi} starts at a value very close to zero due to numerical precision errors introduced during the feedforward process \eqref{eq:ini_u1}.
In Figures \ref{fig:feas_res_syn}, \ref{fig:feas_res_SP500}, and \ref{fig:feas_res_TIMIT}, \textbf{FeasVi} rises to $ 10^{-3} $, $ 10^{-5} $, and $ 10^{-6}$ at early iterations, then steadily decrease to $ 10^{-6} $, $ 10^{-9} $, and $ 10^{-9}$, respectively.   
Across all three datasets, the optimality residual converges to zero with mild fluctuations, while the objective function value decreases monotonically and converges to a stable value.

\begin{figure}[htbp]
    \centering
    
    \begin{minipage}[b]{0.6\textwidth}
        \centering
        \includegraphics[width=\textwidth]{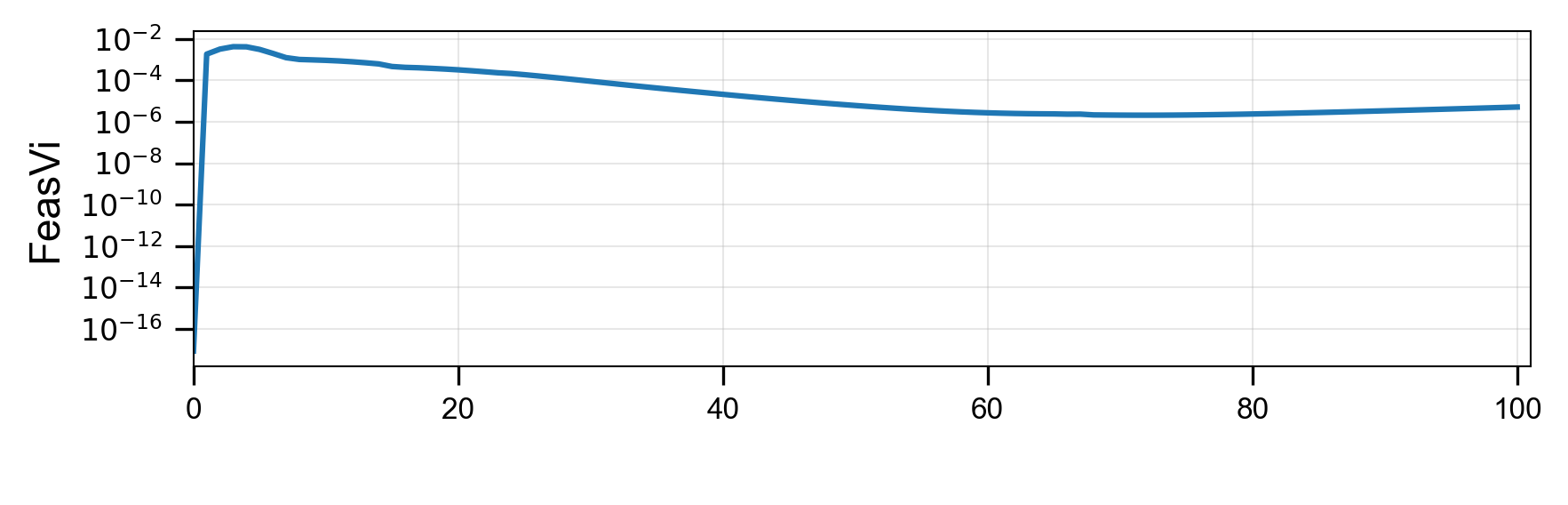} 
    \end{minipage} 
    
    \vspace{-0.5em} 

    \begin{minipage}[b]{0.6\textwidth}
        \centering
        \includegraphics[width=\textwidth]{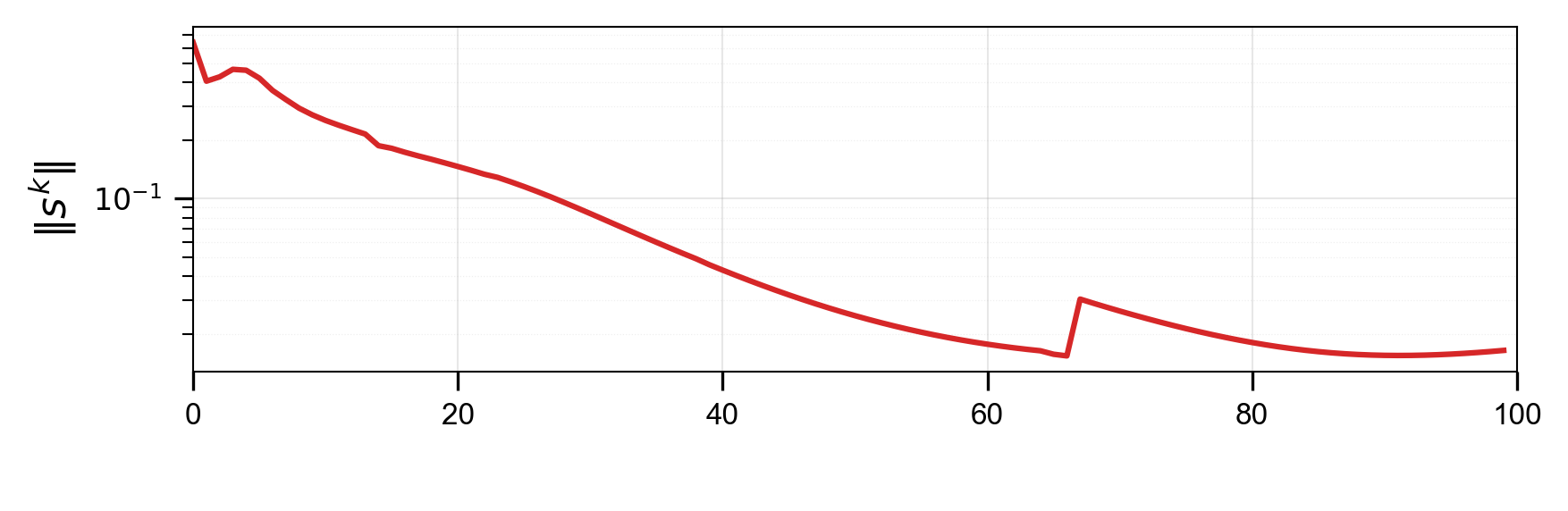} 
    \end{minipage} 

    \vspace{-0.5em} 

    \begin{minipage}[b]{0.6\textwidth}
        \centering
        \includegraphics[width=\textwidth]{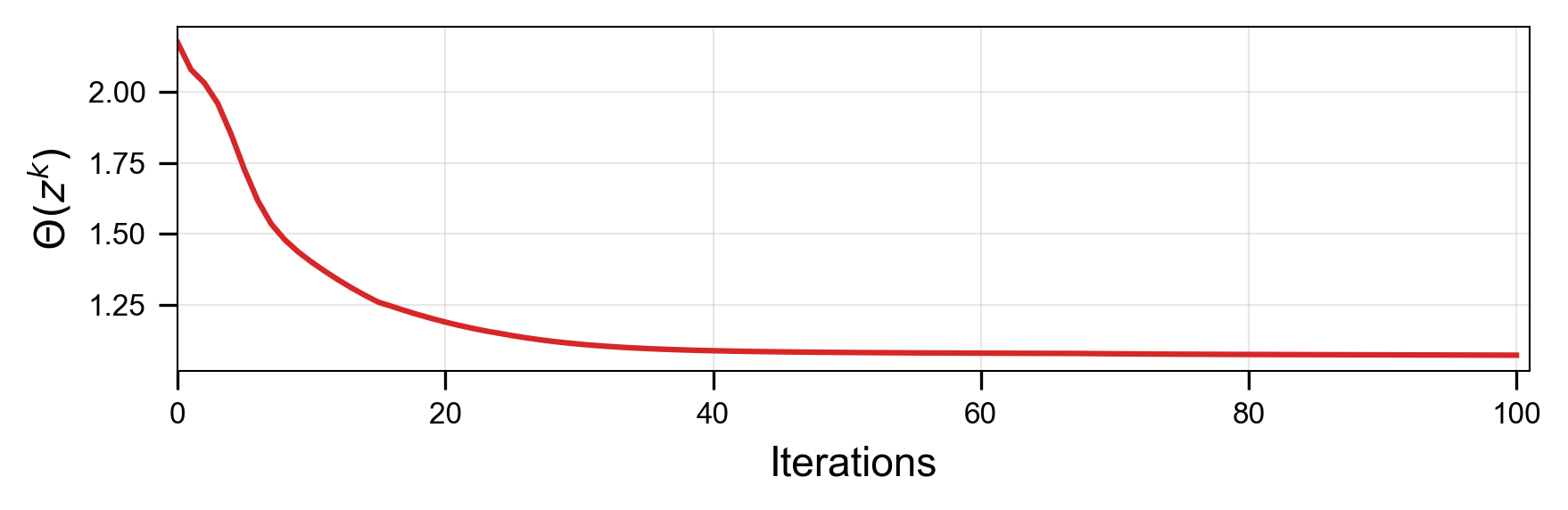} 
    \end{minipage} 

    \caption{Feasibility violation, optimality residual, and function value of SRPA in the synthetic dataset. }
    \label{fig:feas_res_syn}
\end{figure}
\begin{figure}[htbp]
    \centering
    
    \begin{minipage}[b]{0.6\textwidth}
        \centering
        \includegraphics[width=\textwidth]{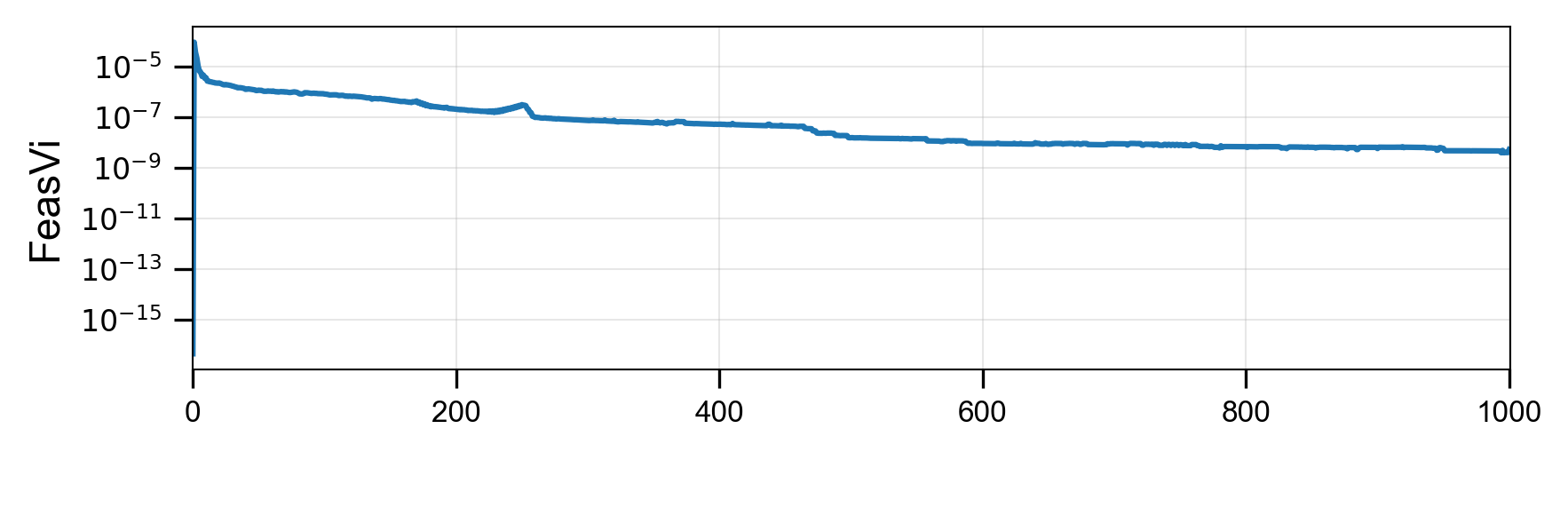} 
    \end{minipage} 
    
    \vspace{-0.5em} 

    \begin{minipage}[b]{0.6\textwidth}
        \centering
        \includegraphics[width=\textwidth]{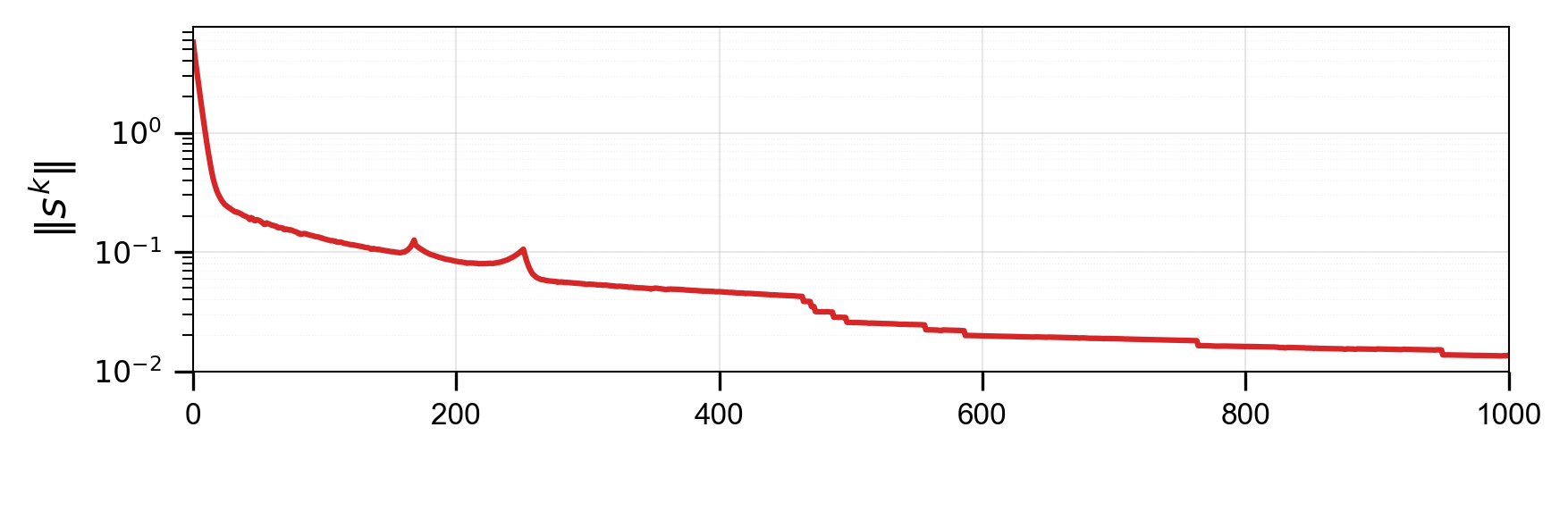} 
    \end{minipage} 

    \vspace{-0.5em} 

    \begin{minipage}[b]{0.6\textwidth}
        \centering
        \includegraphics[width=\textwidth]{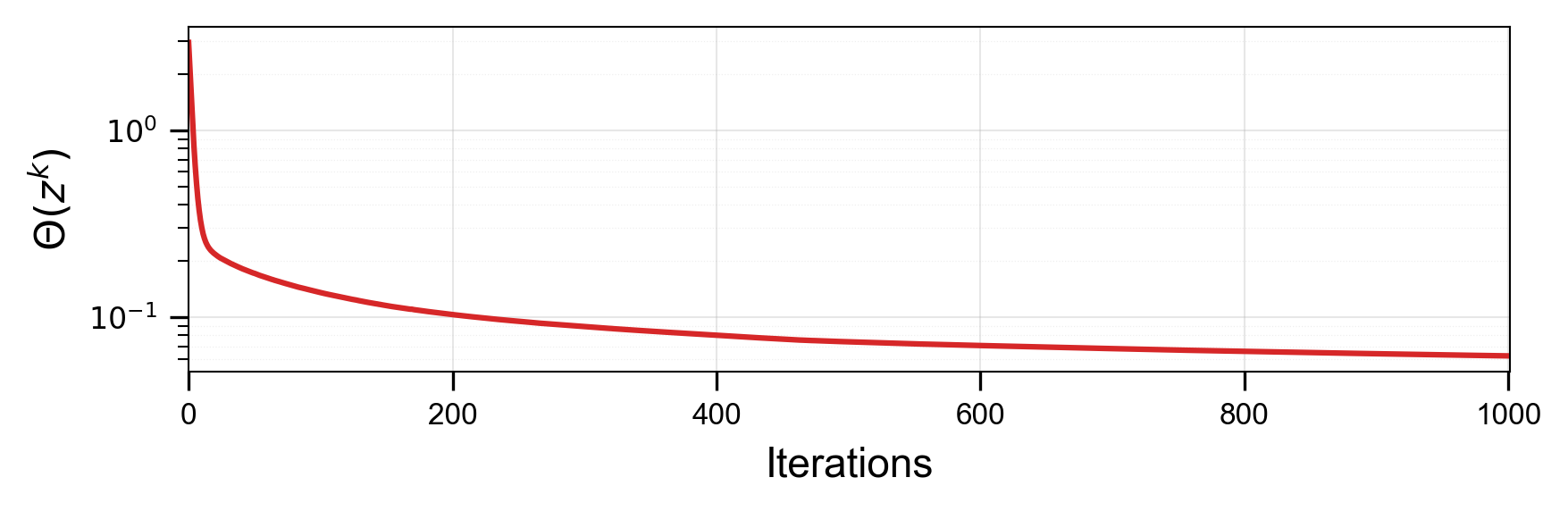} 
    \end{minipage} 

    \caption{Feasibility violation, optimality residual, and function value of SRPA in Volatility of S\&P Index dataset. }
    \label{fig:feas_res_SP500}
\end{figure}
\begin{figure}[htbp]
    \centering
    
    \begin{minipage}[b]{0.6\textwidth}
        \centering
        \includegraphics[width=\textwidth]{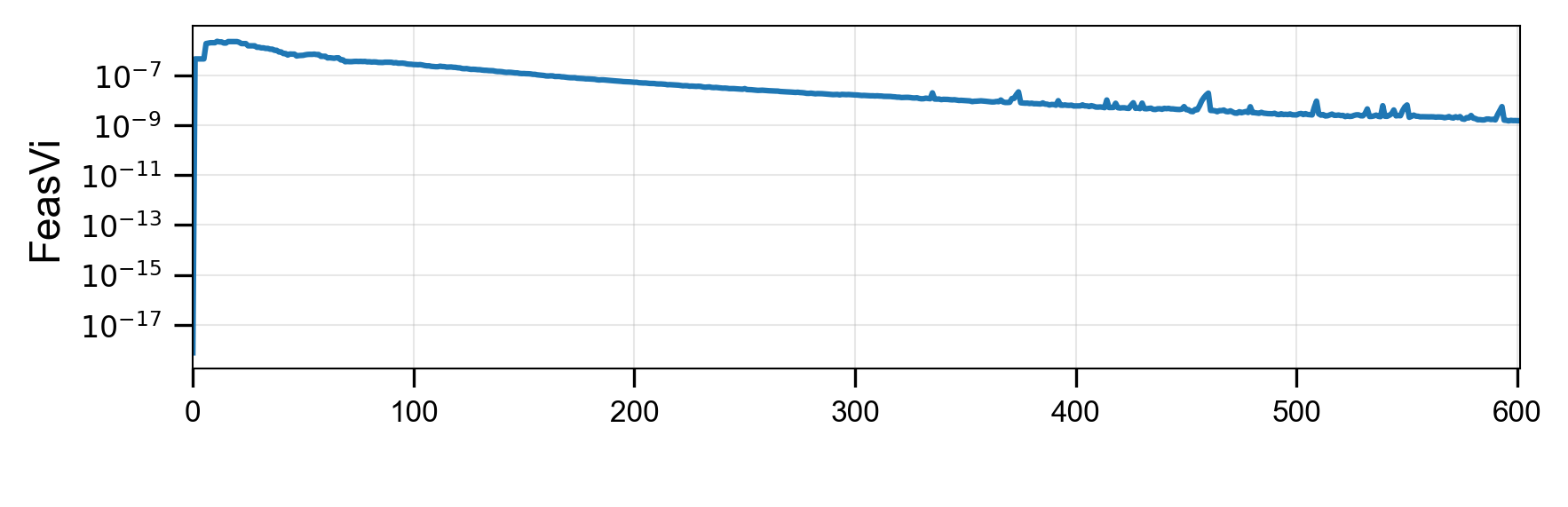} 
    \end{minipage} 
    
    \vspace{-0.5em} 

    \begin{minipage}[b]{0.6\textwidth}
        \centering
        \includegraphics[width=\textwidth]{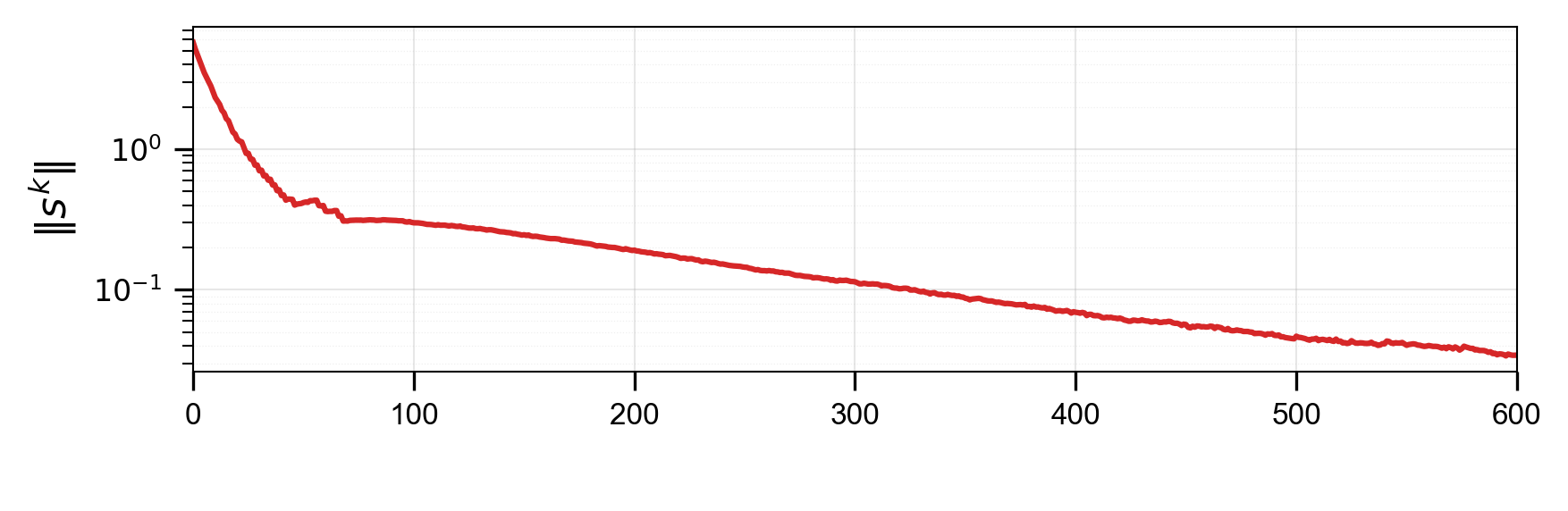} 
    \end{minipage} 

    \vspace{-1.5em} 

    \begin{minipage}[b]{0.6\textwidth}
        \centering
        \includegraphics[width=\textwidth]{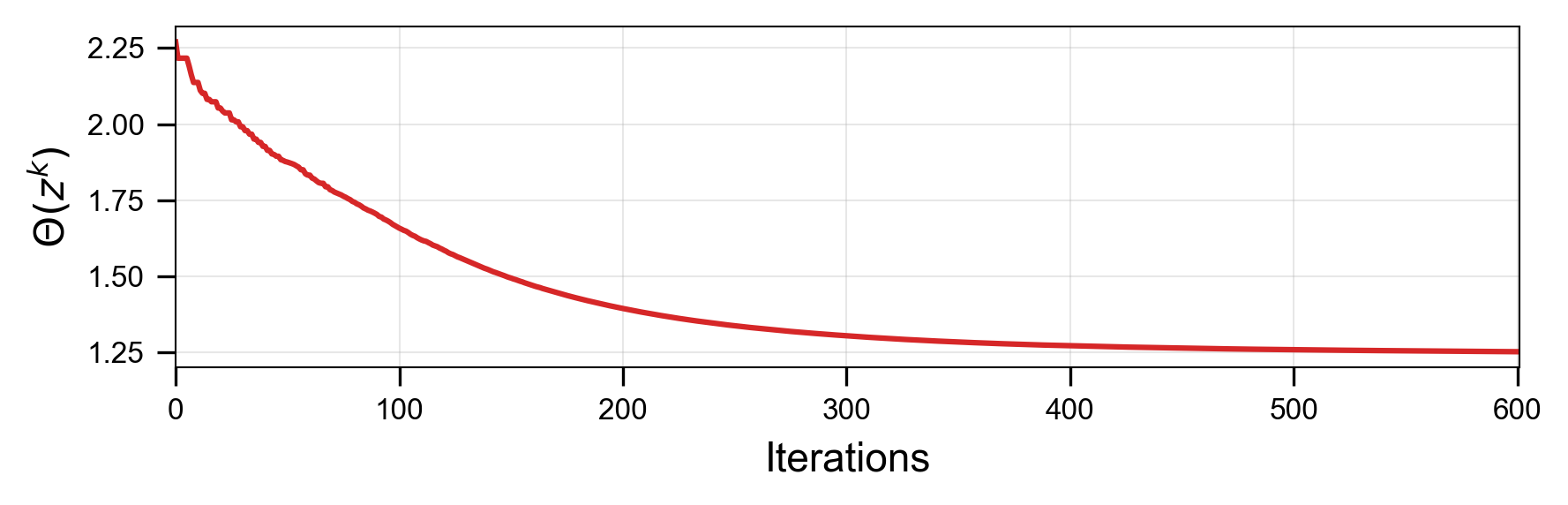} 
    \end{minipage} 

    \caption{Feasibility violation, optimality residual, and function value of SRPA in TIMIT dataset. }
    \label{fig:feas_res_TIMIT}
\end{figure}

\subsection{Training error and test error}
\label{sec:4.4.4} 
In this subsection, we compare SRPA with the state-of-the-art (S)GDs, including gradient descent (GD), gradient descent with gradient clipping (GDC), gradient descent with Nesterov momentum (GDNes), mini-batch stochastic gradient descent (SGD), and SGD with adaptive moment estimation (Adam), where the mini-batch stochastic sampling for the case $ N=1 $ refers to time-step sampling. We employ the Keras API \cite{chollet2015keras} running on TensorFlow 2 to implement the (S)GDs. 
All the implementations of SGDs are conducted in Python 3.10 on a laptop of 16GB RAM and AMD Ryzen 7 5800U with Radeon Graphics (1.90 GHz). 

Similar to the parameter tuning of SRPA in Subsection \ref{sec:4.4.3}, we search for the best parameters and initialization strategies of the aforementioned (S)GD algorithms. Specifically, the six initialization strategies in Subsection \ref{sec:4.4.3} are assessed. 
The learning rates of (S)GDs are searched from $ \{ 10^{-4}, 10^{-3}, 10^{-2}, \\10^{-1}, 1 \} $. 
Regarding the unique parameters of GDC, GDNes, SGD and Adam, 
the clipping norm of GDC is evaluated over $\{0.5, 1,  2,   4 \}$, 
the momentum coefficient of GDNes is set as $ 0.9 $, 
the batch sizes for SGD and Adam are searched from $ \{ 1,2,4 \} $ for the synthetic dataset, $ \{ 25,50,100 \} $ for the Volatility of S\&P Index dataset, and $ \{ 4,8,16 \} $ for the TIMIT dataset. 

The average test errors generated by GD, GDC, GDNes, SGD, and Adam are detailed in Appendix \ref{sec:A1}. 
Based on these results, we configure the parameters for (S)GDs as shown in Table \ref{tab:set_SGDs}, where an epoch is one complete pass of the entire training dataset through the learning algorithm, and compare their performance with SRPA under the settings in Table \ref{tab:set_SRPA}.

\begin{table}[htbp]
\centering
\caption{Parameter settings, initialization strategies, and maximum epochs of (S)GDs.}
\label{tab:set_SGDs}

\footnotesize

\renewcommand{\arraystretch}{1.2}

\begin{tabular}{l p{0.25\textwidth} p{0.25\textwidth} p{0.25\textwidth}}
\toprule
Algorithm 
& Synthetic (ep=100) 
& Volatility of S\&P Index (ep=1000) 
& TIMIT (ep=600) \\ 
\midrule
GD 
& \makecell[l]{lr=$10^{-4}$ \\ init=$\mathcal{N}(0,10^{-1})$}
& \makecell[l]{lr=$10^{-1}$ \\ init=$\mathcal{N}(0,10^{-2})$}
& \makecell[l]{lr=$1$ \\ init=$\mathcal{N}(0,10^{-2})$}
\\ 
\addlinespace
GDC 
& \makecell[l]{lr=$10^{-4}$ \\ cln=0.5 \\ init=$\mathcal{N}(0,10^{-1})$}
& \makecell[l]{lr=$10^{-1}$ \\ cln=0.5 \\ init=$\mathcal{N}(0,10^{-2})$}
& \makecell[l]{lr=$1$ \\ cln=0.5 \\ init=$\mathcal{N}(0,10^{-2})$}
\\ 
\addlinespace
GDNes 
& \makecell[l]{lr=$10^{-1}$ \\ mc=0.9 \\ init=Glorot}
& \makecell[l]{lr=$10^{-1}$ \\ mc=0.9 \\ init=$\mathcal{N}(0,10^{-1})$}
& \makecell[l]{lr=$1$ \\ mc=0.9 \\ init=$\mathcal{N}(0,10^{-2})$}
\\ 
\addlinespace
SGD 
& \makecell[l]{lr=$10^{-4}$ \\ bs=4 \\ init=$\mathcal{N}(0,10^{-1})$}
& \makecell[l]{lr=$10^{-2}$ \\ bs=25 \\ init=$\mathcal{N}(0,10^{-3})$}
& \makecell[l]{lr=$1$ \\ bs=4 \\ init=$\mathcal{N}(0,10^{-2})$}
\\ 
\addlinespace
Adam 
& \makecell[l]{lr=$10^{-1}$ \\ bs=1 \\ init=$\mathcal{N}(0,10^{-2})$}
& \makecell[l]{lr=$10^{-4}$ \\ bs=25 \\ init=$\mathcal{N}(0,10^{-3})$}
& \makecell[l]{lr=$10^{-3}$ \\ bs=8 \\ init=$\mathcal{N}(0,10^{-2})$}
\\ 
\bottomrule
\end{tabular}

\smallskip

\parbox{\textwidth}{
Note: lr, cln, mc, bs, init, and ep denote learning rate, clipping norm, momentum coefficient, batch size, initialization strategy, and maximum number of epochs, respectively.
}

\end{table}

For the synthetic dataset, we compare the performances of algorithms over 100 epochs in a single trial.
Figure \ref{fig:100epochs_epoch} plots the changes of training errors and test errors with respect to the number of epochs.
As shown in the figure, GD, GDC, and SGD stagnate after the early iterations, with their training and test errors remaining at relatively high levels.
The training error of Adam decreases in an oscillatory manner and fluctuates around $1.2$, while its test error oscillates around $6.2$.
The training error of GDNes asymptotically converges to $ 1.2 $, while its test error remains close to $6$ after a mild fluctuation. 
In contrast, SRPA demonstrates greater stability, and its training error and test error settle at approximately $ 0.5 $ and $ 5.5 $, respectively.
As indicated in Table \ref{tab:sp500-TIMIT}, SRPA achieves the best training error and test error among all the algorithms. 
In terms of computational time, the deterministic gradient descent algorithms (GD, GDC and GDNes) require approximately twice the CPU time of stochastic gradient descent algorithms (SGD and Adam), and the time spent by SRPA is between GDs and SGDs.  
 
\begin{figure}[htbp]
    \centering

    \begin{minipage}[b]{0.45\textwidth}
        \centering
        \includegraphics[width=\textwidth]{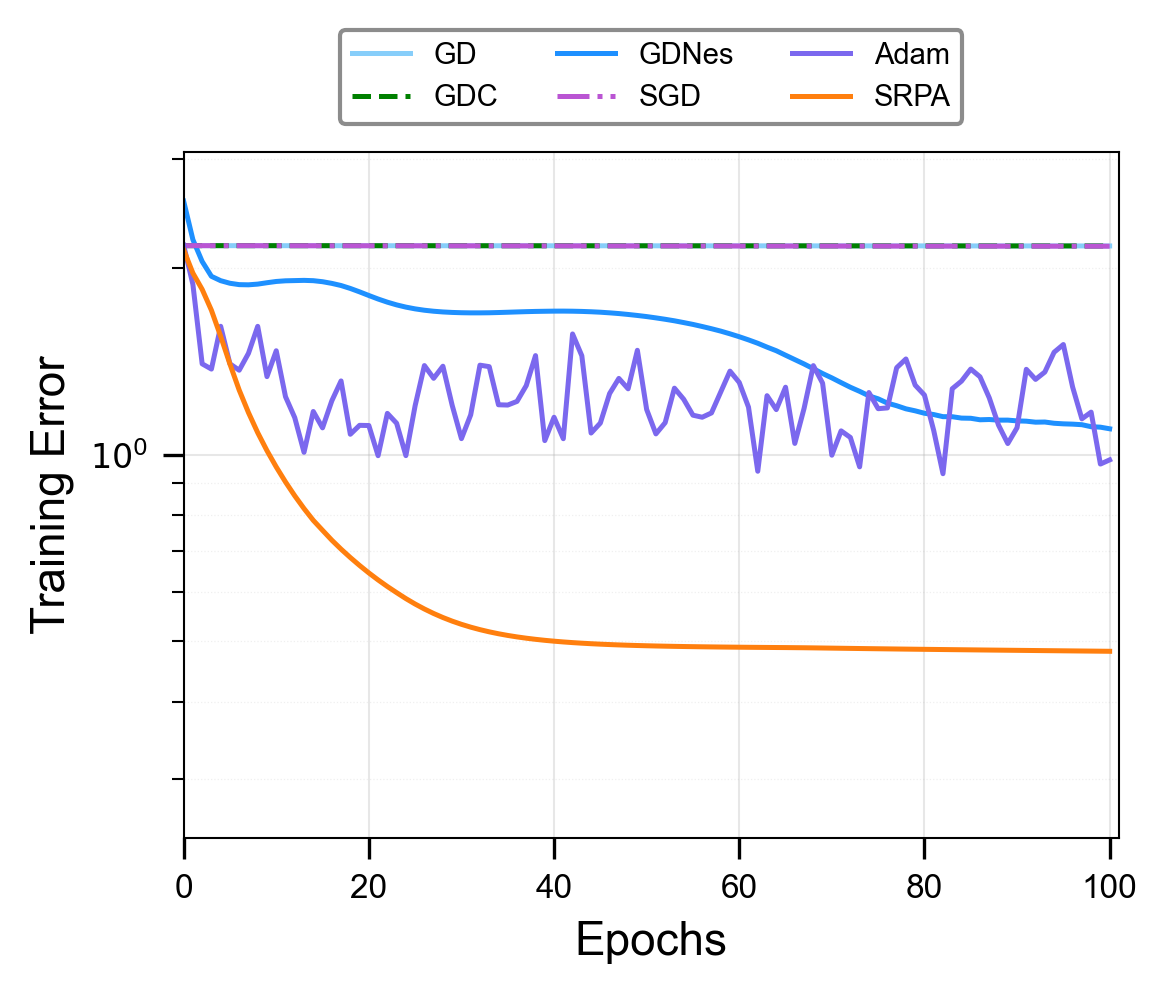}
    \end{minipage}
    \hfill
    \begin{minipage}[b]{0.45\textwidth}
        \centering
        \includegraphics[width=\textwidth]{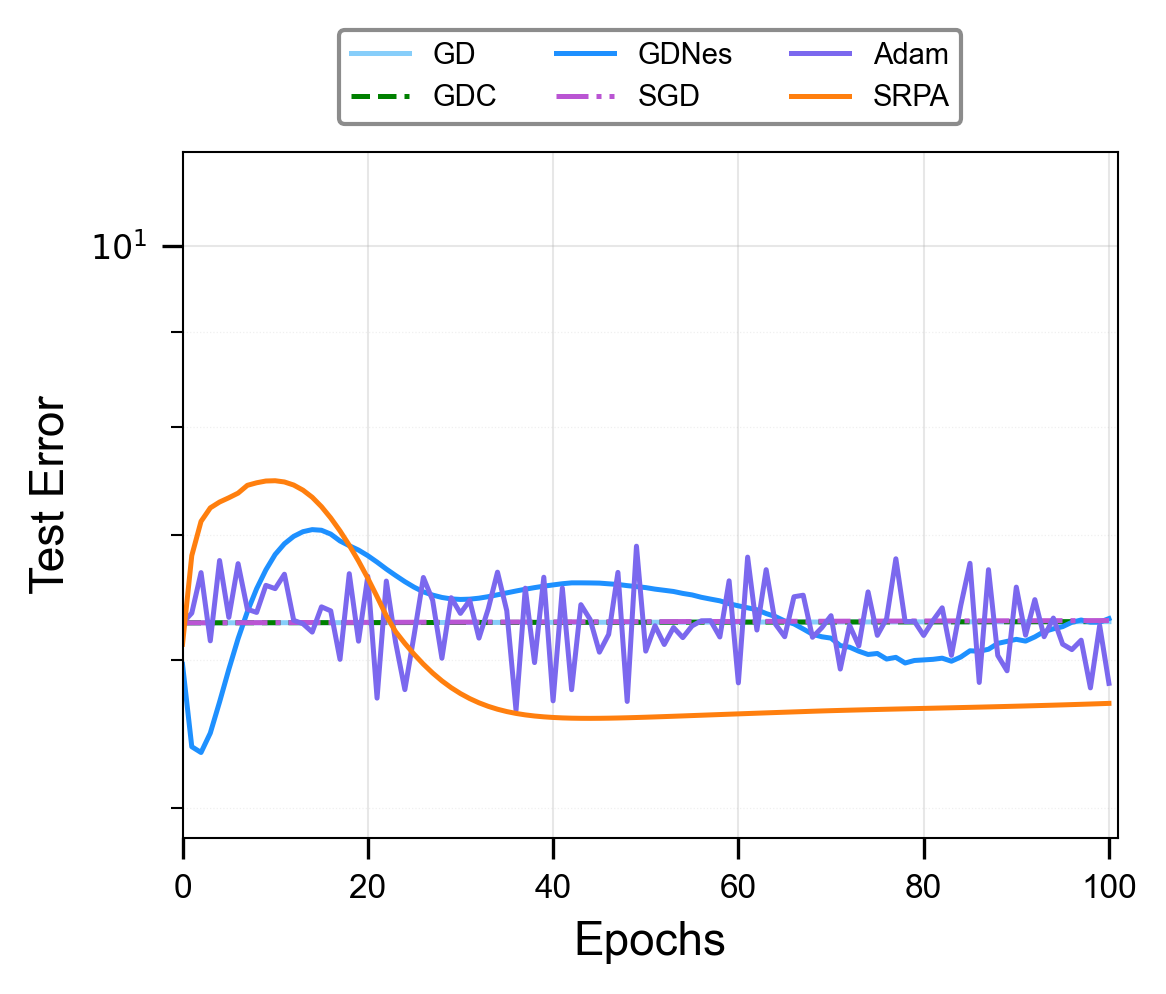} 
    \end{minipage} 

    \caption{Comparison of training and test errors on the synthetic dataset across epochs.}
    \label{fig:100epochs_epoch}
\end{figure}

For the Volatility of S\&P Index dataset, we compare the performances of algorithms over 1000 epochs in a single trial.
Figure \ref{fig:1000epochs_epoch} plots the changes of training errors and test errors with respect to the number of epochs.
As shown in the figure, SGD and Adam exhibit overfitting (training error decreases slowly and test error increases) after around 150 and 50 epochs, respectively, while GD, GDC and GDNes fluctuates violently after around 290, 280 and 40 epochs respectively. 
The minimum test error of each algorithm before significant fluctuations are listed in Table \ref{tab:sp500-TIMIT}.
As shown in Figure \ref{fig:1000epochs_epoch} and Table \ref{tab:sp500-TIMIT}, among all the algorithms, SRPA achieves the minimum test error while maintaining the training error at a reasonable level. 
Furthermore, SRPA demonstrates robustness against overfitting.
Regarding computational time, we acknowledge that our algorithm relies on an existing Python package (cvxpy) to solve the subproblem. Consequently, the computational cost increases significantly with the problem scale, i.e., $ \bar{N} $. 
Specifically, on the Volatility of S\&P Index dataset, the per-epoch time for (S)GD ranges from 0.1 to 2 seconds, while SRPA requires approximately 15 seconds. Therefore, developing efficient algorithms tailored to the specific structure of the nonsmooth convex subproblems in SRPA is a critical direction for future work. 
 
\begin{figure}[htbp]
    \centering

    \begin{minipage}[b]{0.45\textwidth}
        \centering
        \includegraphics[width=\textwidth]{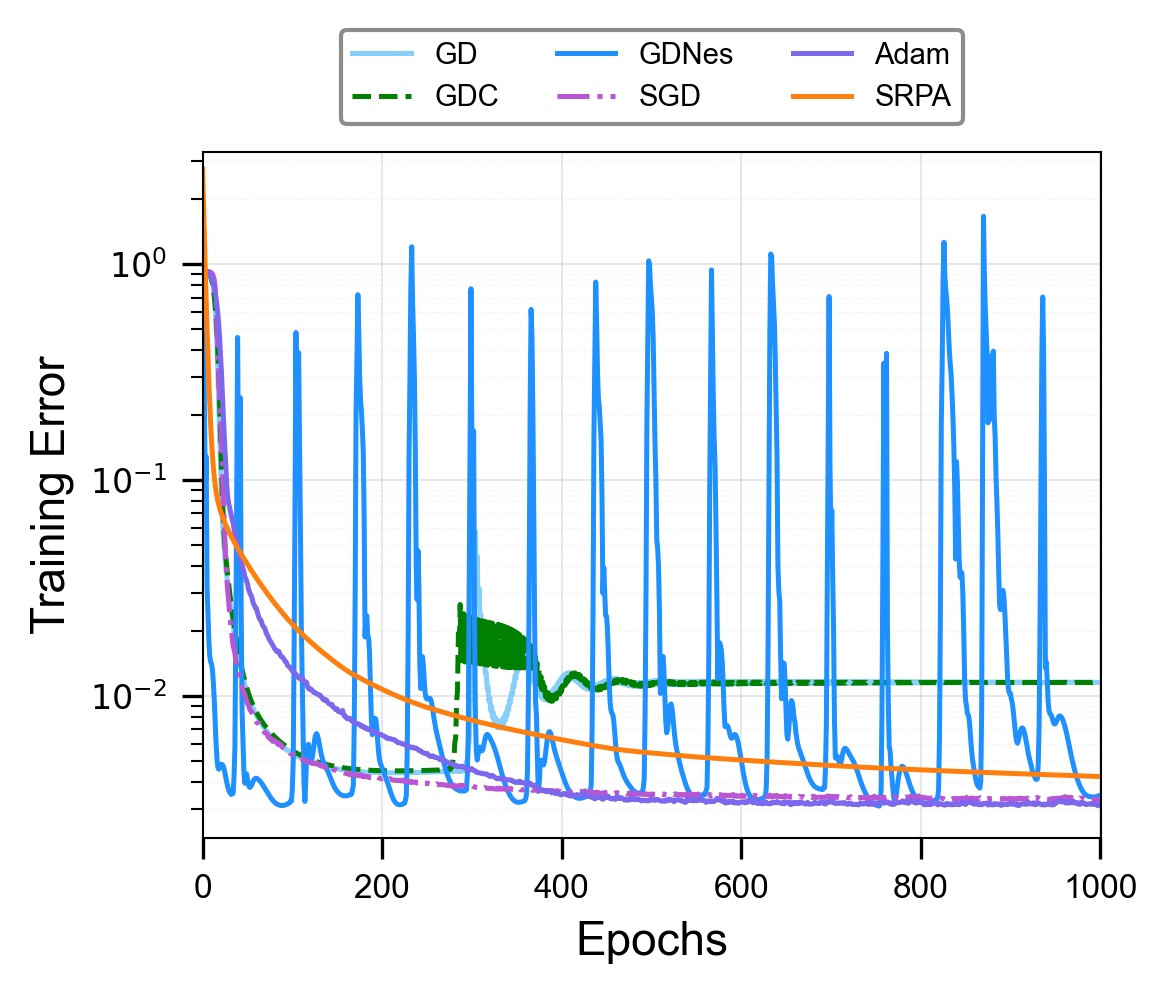}
    \end{minipage}
    \hfill
    \begin{minipage}[b]{0.45\textwidth}
        \centering
        \includegraphics[width=\textwidth]{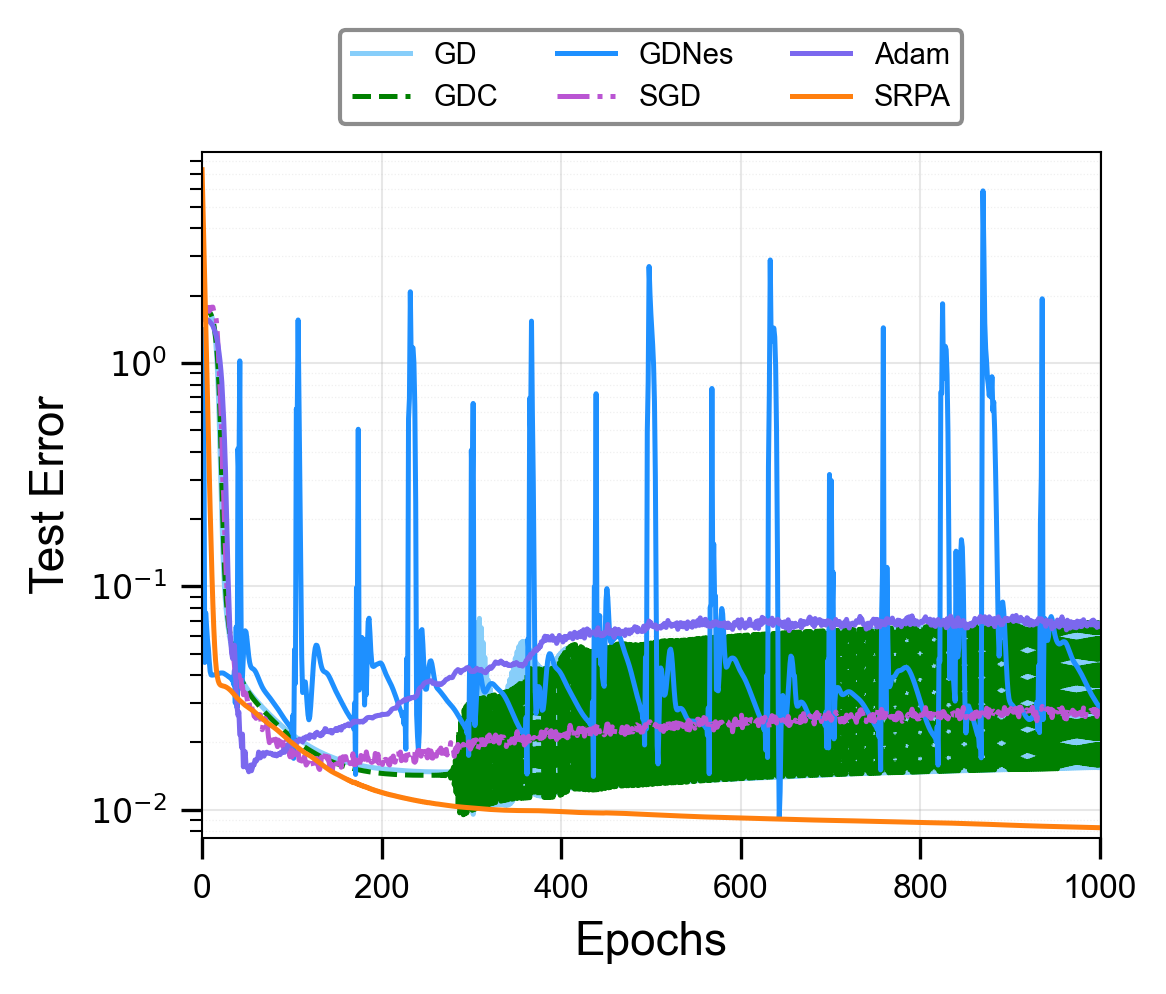} 
    \end{minipage} 

    \caption{Comparison of training and test errors on the Volatility of S\&P Index dataset across epochs.}
    \label{fig:1000epochs_epoch}
\end{figure}

For the TIMIT dataset, we compare the performances of algorithms over 600 epochs in a single trial.
Figure \ref{fig:600epochs_epoch-TIMIT} plots the changes of training errors and test errors with respect to the number of epochs.
As shown in the figure, GDNes, SGD, Adam, and SRPA exhibit overfitting after around 490, 340, 190, and 350 epochs, respectively.  
The minimum test error of each algorithm are listed in Table \ref{tab:sp500-TIMIT}.
As shown in Figure \ref{fig:600epochs_epoch-TIMIT} and Table \ref{tab:sp500-TIMIT}, among all the algorithms, SRPA achieves the minimum test error while maintaining the training error at a reasonable level.  
 
\begin{figure}[htbp]
    \centering

    \begin{minipage}[b]{0.45\textwidth}
        \centering
        \includegraphics[width=\textwidth]{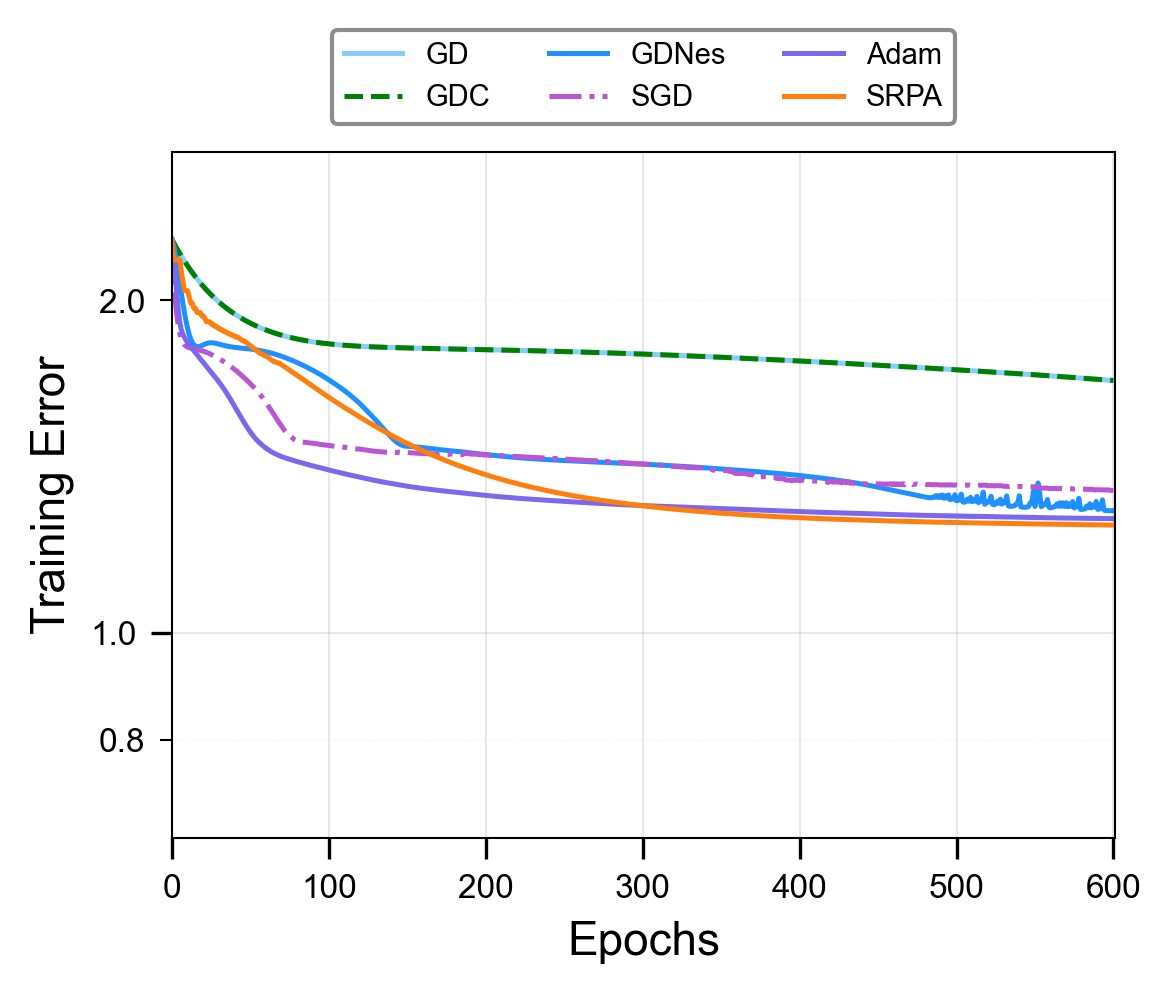}
    \end{minipage}
    \hfill
    \begin{minipage}[b]{0.45\textwidth}
        \centering
        \includegraphics[width=\textwidth]{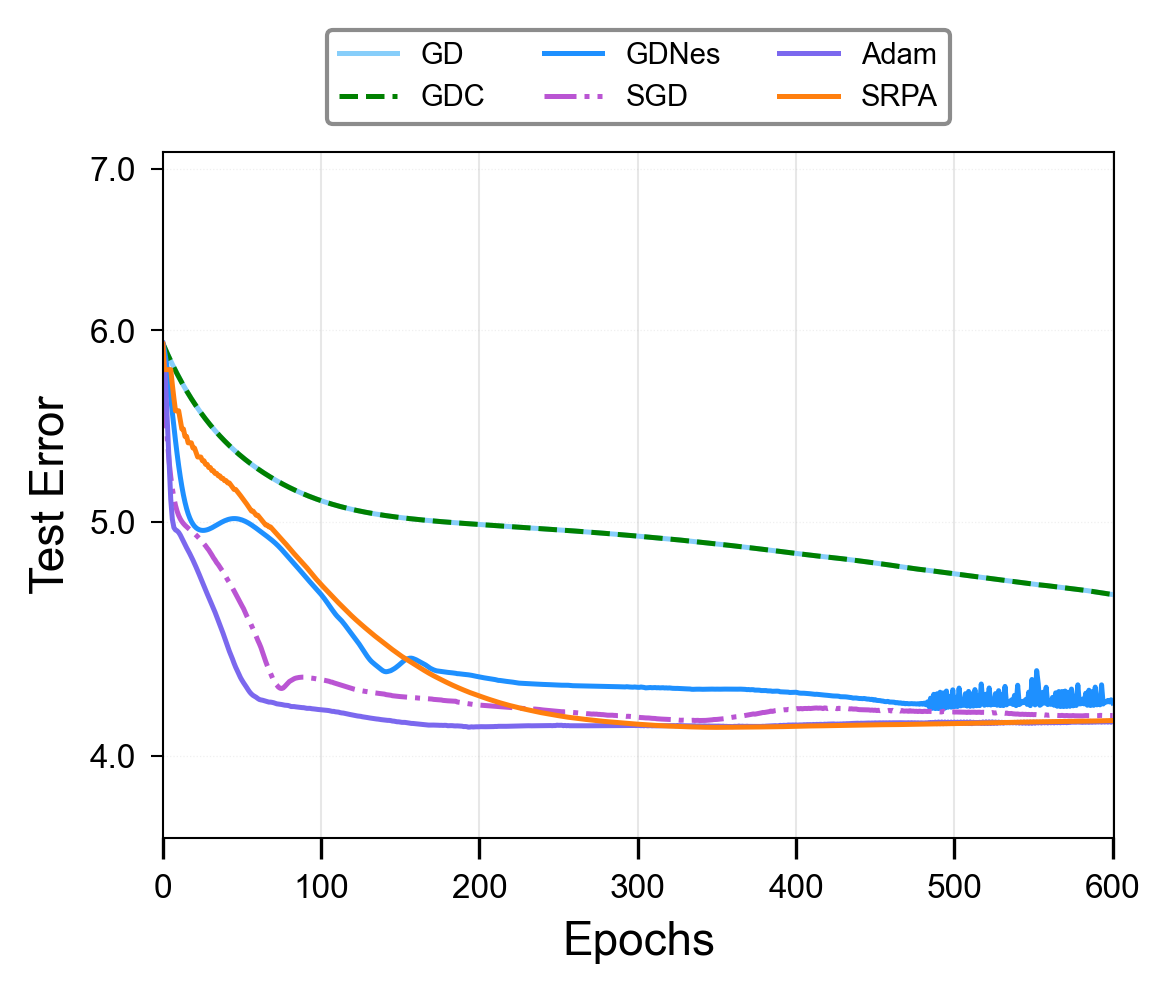} 
    \end{minipage} 

    \caption{Comparison of training and test errors on the TIMIT dataset across epochs.}
    \label{fig:600epochs_epoch-TIMIT}
\end{figure}

\begin{table}[t]
\centering
\caption{Minimum test errors and corresponding training errors generated by different optimization algorithms in the stable stage after convergence on different datasets. 
} 
    \begin{tabular}{l cc cc cc }
        \toprule
        \multirow{2}{*}{Algorithm} & \multicolumn{2}{c}{Synthetic} & \multicolumn{2}{c}{Volatility of S\&P Index} & \multicolumn{2}{c}{TIMIT} \\
        \cmidrule(lr){2-3} \cmidrule(lr){4-5} \cmidrule(lr){6-7}
        & {\textbf{TrainErr}} & {\textbf{TestErr}} & {\textbf{TrainErr}} & {\textbf{TestErr}} & {\textbf{TrainErr}} & {\textbf{TestErr}} \\
        \midrule
        GD   & 2.1747 & 6.2849 & 0.0045 & 0.0147 & 1.6922 & 4.6653 \\
        GDC  & 2.1747 & 6.2849 & 0.0051 & 0.0135 & 1.6922 & 4.6653 \\
        GDNes & 1.1863 & 5.9812 & 0.4572 & 0.0286 & 1.3258 & 4.1846 \\
        SGD  & 2.1747 & 6.2849 & 0.0046 & 0.0149 & 1.4099 & 4.1374 \\
        Adam & 1.2195 & 5.6419 & 0.0305 & 0.0148 & 1.3349 & 4.1113 \\
        SRPA & \textbf{0.4958} & \textbf{5.5864} & \textbf{0.0042} & \textbf{0.0083} & \textbf{1.2832} & \textbf{4.1105} \\
        \bottomrule
    \end{tabular} 
\label{tab:sp500-TIMIT}
\end{table}

\section{Conclusion}

This paper investigates nonconvex nonsmooth optimization problem \eqref{l1pen}, with a particular focus on the complexity bounds for finding approximate first-order and second-order d-stationary points of \eqref{l1pen_r}. 
First, we identify easily verifiable conditions, under which an approximate first-order d-stationary point of \eqref{l1pen} is also an approximate second-order d-stationary point of \eqref{l1pen}. 
Second, we propose the sequential regularized piecewise affine algorithm, SRPA, for finding an approximate second-order d-stationary point of \eqref{l1pen_r} with complexity guarantee. By exploiting the specific structure of the composite terms, our algorithm constructs piecewise affine approximations and solves quadratically regularized subproblems. 
The active-set strategy is incorporated to further alleviate the computational cost for each iteration. 
The algorithm SRPA is proven to generate an $ \epsilon $-approximate second-order d-stationary point of \eqref{l1pen_r} within $ \mathcal{O}(\epsilon^{-2}) $ iterations under certain conditions. 
Finally, numerical experiments on the synthetic, Volatility of S\&P Index, and TIMIT datasets for RNN training validate the practical effectiveness of SRPA.

\bibliographystyle{abbrv}
\bibliography{references}

\begin{thebibliography}{10}

\bibitem{ALT2025}
J.~H. Alcantara, C.-p. Lee, and A.~Takeda.
\newblock A four-operator splitting algorithm for nonconvex and nonsmooth
  optimization.
\newblock {\em SIAM J. Optim.}, 35(3):1846--1872, 2025.

\bibitem{bengio2009learning}
Y.~Bengio.
\newblock Learning deep architectures for {AI}.
\newblock {\em Found. Trends Mach. Learn.}, 2:1--127, 2009.

\bibitem{Burke1985}
J.~V. Burke.
\newblock Descent methods for composite nondifferentiable optimization
  problems.
\newblock {\em Math. Program.}, 33(3):260--279, 1985.

\bibitem{Burke1987}
J.~V. Burke.
\newblock Second order necessary and sufficient conditions for convex composite
  {NDO}.
\newblock {\em Math. Program.}, 38(3):287--302, 1987.

\bibitem{CGT2022}
C.~Cartis, N.~Gould, and P.~Toint.
\newblock The evaluation complexity of finding high-order minimizers of
  nonconvex optimization.
\newblock {\em Int. Congress of Mathematicians}, 7:5256--5289, 2022.

\bibitem{CGT2011}
C.~Cartis, N.~I.~M. Gould, and P.~L. Toint.
\newblock On the evaluation complexity of composite function minimization with
  applications to nonconvex nonlinear programming.
\newblock {\em SIAM J. Optim.}, 21(4):1721--1739, 2011.

\bibitem{chollet2015keras}
F.~Chollet et~al.
\newblock Keras, 2015.

\bibitem{CP-book}
Y.~Cui and J.-S. Pang.
\newblock {\em Modern Nonconvex Nondifferentiable Optimization}.
\newblock Society for Industrial and Applied Mathematics, Philadelphia, PA,
  2021.

\bibitem{DP2019}
D.~Drusvyatskiy and C.~Paquette.
\newblock Efficiency of minimizing compositions of convex functions and smooth
  maps.
\newblock {\em Math. Program.}, 178(1):503--558, 2019.

\bibitem{FWG2019}
S.~Fiege, A.~Walther, and A.~Griewank.
\newblock An algorithm for nonsmooth optimization by successive piecewise
  linearization.
\newblock {\em Math. Program.}, 177(1):343--370, 2019.

\bibitem{Fletcher1981PMO}
R.~Fletcher.
\newblock {\em Non-Smooth Optimization}, chapter~14, pages 357--416.
\newblock John Wiley $\&$ Sons, Ltd, Hoboken, NJ, USA, 1981.

\bibitem{Fletcher1982}
R.~Fletcher.
\newblock {\em A model algorithm for composite nondifferentiable optimization
  problems}, pages 67--76.
\newblock Springer Berlin Heidelberg, Berlin, Heidelberg, 1982.

\bibitem{garofolo1993darpa}
J.~S. Garofolo, L.~F. Lamel, W.~M. Fisher, J.~G. Fiscus, D.~S. Pallett, and
  D.~N. L.
\newblock Darpa {TIMIT} acoustic-phonetic continous speech corpus {CD-ROM}.
\newblock {\em NIST Interagency/Internal Report (NISTIR), National Institute of
  Standards and Technology, Gaithersburg, MD}, 1993.

\bibitem{glorot2010understanding}
X.~Glorot and Y.~Bengio.
\newblock Understanding the difficulty of training deep feedforward neural
  networks.
\newblock In {\em Proceedings of the 13th International Conference on
  Artificial Intelligence and Statistics}, pages 249--256, Chia Laguna Resort,
  Sardinia, Italy, 2010. PMLR.

\bibitem{goodfellow2016deep}
I.~Goodfellow, Y.~Bengio, and A.~Courville.
\newblock {\em Deep learning}.
\newblock MIT Press, Cambridge, MA, 2016.

\bibitem{GST2021}
S.~Gratton, E.~Simon, and P.~L. Toint.
\newblock An algorithm for the minimization of nonsmooth nonconvex functions
  using inexact evaluations and its worst-case complexity.
\newblock {\em Math. Program.}, 187(1):1--24, 2021.

\bibitem{Griewank2013}
A.~Griewank.
\newblock On stable piecewise linearization and generalized algorithmic
  differentiation.
\newblock {\em Optim. Methods Softw.}, 28(6):1139--1178, 2013.

\bibitem{GW2019}
A.~Griewank and A.~Walther.
\newblock Relaxing kink qualifications and proving convergence rates in
  piecewise smooth optimization.
\newblock {\em SIAM J. Optim.}, 29(1):262--289, 2019.

\bibitem{HT2023}
N.~Hallak and M.~Teboulle.
\newblock An adaptive lagrangian-based scheme for nonconvex composite
  optimization.
\newblock {\em Math. Oper. Research}, 48(4):2337--2352, 2023.

\bibitem{he2015delving}
K.~He, X.~Zhang, S.~Ren, and J.~Sun.
\newblock Delving deep into rectifiers: surpassing human-level performance on
  imagenet classification.
\newblock In {\em Proceedings of the 2015 IEEE International Conference on
  Computer Vision (ICCV)}, pages 1026--1034, Santiago, Chile, 2015. IEEE
  Computer Society.

\bibitem{Ioffe1979b}
A.~D. Ioffe.
\newblock Necessary and sufficient conditions for a local minimum. 3: Second
  order conditions and augmented duality.
\newblock {\em SIAM J. Control Optim.}, 17(2):266--288, 1979.

\bibitem{JWC2025}
L.~Jin, X.~Wang, and X.~Chen.
\newblock Nonconvex nonsmooth multicomposite optimization and its applications
  to recurrent neural networks.
\newblock {\em SIAM J. Optim.}, 35:2343--2371, 2025.

\bibitem{klambauer2017self}
G.~Klambauer, T.~Unterthiner, A.~Mayr, and S.~Hochreiter.
\newblock Self-normalizing neural networks.
\newblock In {\em Proceedings of the 31st International Conference on Neural
  Information Processing Systems}, pages 972--981, Red Hook, NY, USA, 2017.
  Curran Associates Inc.

\bibitem{LW2016}
A.~S. Lewis and S.~J. Wright.
\newblock A proximal method for composite minimization.
\newblock {\em Math. Program.}, 158(1):501--546, 2016.

\bibitem{LRP-LLC}
W.~Liu, X.~Liu, and X.~Chen.
\newblock Linearly constrained nonsmooth optimization for training
  autoencoders.
\newblock {\em SIAM J. Optim.}, 32(3):1931--1957, 2022.

\bibitem{NN2025}
Y.~Nabou and I.~Necoara.
\newblock Regularized higher-order taylor approximation methods for nonlinear
  least-squares.
\newblock {\em arxiv: 2503.02370}, 2025.

\bibitem{PRA2017B-sta}
J.-S. Pang, M.~Razaviyayn, and A.~Alvarado.
\newblock Computing b-stationary points of nonsmooth dc programs.
\newblock {\em Math. Oper. Research}, 42(1):95--118, 2017.

\bibitem{Penot1994a}
J.-P. Penot.
\newblock Optimality conditions in mathematical programming and composite
  optimization.
\newblock {\em Math. Program.}, 67(1-3):225--245, 1994.

\bibitem{PR1995}
R.~Poliquin and T.~Rockafellar.
\newblock {\em Second-Order Nonsmooth Analysis in Nonlinear Programming}, pages
  322--350.
\newblock World Scientific, USA, 1995.

\bibitem{powell1983general}
M.~Powell.
\newblock General algorithms for discrete nonlinear approximation calculations.
\newblock {\em Report DAMTP 1983/NA2}, 1983.

\bibitem{rockafellar2009variational}
R.~T. Rockafellar and R.~J.-B. Wets.
\newblock {\em Variational {A}nalysis}.
\newblock Springer, Berlin, Heidelberg, 2009.

\bibitem{Scholtes2012}
S.~Scholtes.
\newblock {\em Introduction to Piecewise Differentiable Equations}.
\newblock Springer New York, New York, NY, 2012.

\bibitem{wang2025thesis}
Y.~Wang.
\newblock {\em An augmented Lagrangian method for training recurrent neural
  networks with sample average approximation}.
\newblock {Ph.D.} thesis, The Hong Kong Polytechnic University, Hong Kong,
  2025.
\newblock https://theses.lib.polyu.edu.hk/handle/200/13783.

\bibitem{WZC2024}
Y.~Wang, C.~Zhang, and X.~Chen.
\newblock An augmented {Lagrangian} method for training recurrent neural
  networks.
\newblock {\em SIAM J. Sci. Comput.}, 47(1):C22--C51, 2025.

\bibitem{Y141}
Y.~Yuan.
\newblock Conditions for convergence of trust region algorithms for nonsmooth
  optimization.
\newblock {\em Math. Program.}, 31(2):220--228, 1985.

\bibitem{Proxlinear2021}
J.~Zhang and L.~Xiao.
\newblock Stochastic variance-reduced prox-linear algorithms for nonconvex
  composite optimization.
\newblock {\em Math. Program.}, 195:649--691, 2022.

\end{thebibliography}

\section*{Statements and Declarations}

%
%
%
%
\subsection*{Funding}
This paper is supported in part by Hong Kong Research Grant Council grants (PolyU15300124, PolyU15300625) and National Natural Science Foundation of China (No 11271278). 

\subsection*{Competing Interests}
The authors declare no competing interests.

\subsection*{Data Availability}
Data used for this work are either randomly generated or benchmark data, as described in the paper. 
The datasets generated during the current study are available at \url{https://github.com/LingziJIN-code/SRPA_RNN}.

\subsection*{Code availability}
The code and corresponding description are available at \url{https://github.com/LingziJIN-code/SRPA_RNN}.

\appendix 

\section{Parameter tuning for SRPA and SGDs in Section \ref{sec:4.4}}\label{sec:A1}

The means and the standard deviations of test errors generated by SRPA are listed in Tables \ref{tab:SRPA-tune-1}-\ref{tab:SRPA-tune-3}, where each best test error is highlighted in bold.

\begin{table}[htbp]
\centering
\caption{The averaged test errors generated by SRPA under different $ (\rho_{1}, \eta_{1}, \eta_{2}) $ and initialization strategies on the Synthetic dataset: the number of iterations and repetitions are 50 and 10, respectively.
Means and standard deviations are approximated as 3 decimal places and 2 significant figures, respectively.}
\label{tab:SRPA-tune-1}

\resizebox{\textwidth}{!}{%
\begin{tabular}{l l l l l l l l}
\toprule
$\rho_{1}$ & $(\eta_{1}, \eta_{2})$ & $\mathcal{N}(0, 10^{-3})$ & $\mathcal{N}(0, 10^{-2})$ & $\mathcal{N}(0, 10^{-1})$ & He & Glorot & LeCun \\ 
\midrule
0.5 & (0.9, 1.1) & $6.158 \pm 0.60$ & $6.134 \pm 0.72$ & $6.096 \pm 0.58$ & $6.962 \pm 1.2$ & $6.335 \pm 0.67$ & $6.328 \pm 0.62$ \\ 
    & (0.9, 1.2) & $6.213 \pm 0.60$ & $6.166 \pm 0.74$ & $6.089 \pm 0.58$ & $6.980 \pm 1.2$ & $6.366 \pm 0.69$ & $6.323 \pm 0.62$ \\ 
    & (0.9, 1.3) & $6.199 \pm 0.56$ & $6.166 \pm 0.70$ & $\mathbf{6.073 \pm 0.58}$ & $6.957 \pm 1.2$ & $6.365 \pm 0.68$ & $6.337 \pm 0.63$ \\ 
    & (0.99, 1.1) & $6.206 \pm 0.58$ & $6.180 \pm 0.73$ & $6.092 \pm 0.59$ & $6.975 \pm 1.2$ & $6.351 \pm 0.65$ & $6.350 \pm 0.64$ \\ 
    & (0.99, 1.2) & $6.248 \pm 0.58$ & $6.170 \pm 0.73$ & $6.081 \pm 0.59$ & $6.975 \pm 1.2$ & $6.331 \pm 0.67$ & $6.285 \pm 0.63$ \\ 
    & (0.99, 1.3) & $6.213 \pm 0.55$ & $6.207 \pm 0.74$ & $6.077 \pm 0.59$ & $6.994 \pm 1.2$ & $6.334 \pm 0.66$ & $6.326 \pm 0.64$ \\ 
    & (0.999, 1.1) & $6.223 \pm 0.57$ & $6.179 \pm 0.73$ & $6.079 \pm 0.59$ & $6.983 \pm 1.2$ & $6.379 \pm 0.68$ & $6.344 \pm 0.64$ \\ 
    & (0.999, 1.2) & $6.253 \pm 0.58$ & $6.185 \pm 0.73$ & $6.080 \pm 0.59$ & $6.975 \pm 1.2$ & $6.331 \pm 0.67$ & $6.301 \pm 0.65$ \\ 
    & (0.999, 1.3) & $6.213 \pm 0.55$ & $6.227 \pm 0.73$ & $6.077 \pm 0.59$ & $6.994 \pm 1.2$ & $6.334 \pm 0.66$ & $6.326 \pm 0.64$ \\ 
\addlinespace
1 & Any\textsuperscript{*} & $6.981 \pm 0.52$ & $6.505 \pm 0.58$ & $6.192 \pm 0.64$ & $6.966 \pm 0.89$ & $6.188 \pm 0.72$ & $6.229 \pm 0.72$ \\ 
\addlinespace
2 & Any\textsuperscript{*} & $7.289 \pm 0.032$ & $7.095 \pm 0.30$ & $6.714 \pm 0.60$ & $7.095 \pm 0.82$ & $6.377 \pm 0.79$ & $6.401 \pm 0.78$ \\ 
\bottomrule
\end{tabular}
}

\parbox{\textwidth}{\textsuperscript{*}``Any'' refers to any $ (\eta_{1}, \eta_{2}) \in \{ 0.9, 0.99, 0.999 \} \times \{ 1.1, 1.2, 1.3 \} $.}

\end{table}

\begin{table}[htbp]
\centering
\caption{The averaged test errors generated by SRPA under different $ (\rho_{1}, \eta_{1}, \eta_{2}) $ and initialization strategies on the Volatility of S\&P Index dataset: the number of iterations and repetitions are 100 and 3, respectively.
Means and standard deviations are approximated as 4 decimal places and 2 significant figures, respectively.}
\label{tab:SRPA-tune-2}

\resizebox{\textwidth}{!}{%
\begin{tabular}{l l l l l l l}
\toprule
$\rho_{1}$ & $(\eta_{1}, \eta_{2})$ & $\mathcal{N}(0, 10^{-3})$ & $\mathcal{N}(0, 10^{-2})$ & $\mathcal{N}(0, 10^{-1})$ & Glorot & LeCun \\ 
\midrule
0.01 & (0.7, 1.1) & $1.4745 \pm 0.070$ & $0.5947 \pm 0.13$ & $0.0391 \pm 0.0079$ & $0.0319 \pm 0.0031$ & $0.0372 \pm 0.0065$ \\ 
     & (0.7, 1.2) & $1.3364 \pm 0.10$ & $0.3689 \pm 0.018$ & $0.0386 \pm 0.0081$ & $0.0303 \pm 0.0045$ & $0.0371 \pm 0.0062$ \\ 
     & (0.7, 1.3) & $1.2752 \pm 0.15$ & $0.3396 \pm 0.12$ & $0.0394 \pm 0.0077$ & $0.0337 \pm 0.0051$ & $0.0352 \pm 0.0062$ \\ 
     & (0.8, 1.1) & $1.4787 \pm 0.076$ & $0.4585 \pm 0.11$ & $0.0395 \pm 0.0077$ & $0.0326 \pm 0.0038$ & $0.0376 \pm 0.0061$ \\ 
     & (0.8, 1.2) & $1.3387 \pm 0.10$ & $0.3196 \pm 0.088$ & $0.0388 \pm 0.0080$ & $0.0309 \pm 0.0039$ & $0.0369 \pm 0.0062$ \\ 
     & (0.8, 1.3) & $1.2752 \pm 0.15$ & $0.3396 \pm 0.12$ & $0.0389 \pm 0.0081$ & $0.0337 \pm 0.0051$ & $0.0362 \pm 0.0049$ \\ 
     & (0.9, 1.1) & $1.5066 \pm 0.086$ & $0.2853 \pm 0.088$ & $0.0388 \pm 0.0076$ & $0.0335 \pm 0.0050$ & $0.0376 \pm 0.0060$ \\ 
     & (0.9, 1.2) & $1.3449 \pm 0.11$ & $0.2559 \pm 0.043$ & $0.0388 \pm 0.0084$ & $0.0323 \pm 0.0060$ & $0.0370 \pm 0.0062$ \\ 
     & (0.9, 1.3) & $1.2752 \pm 0.15$ & $0.2877 \pm 0.063$ & $0.0393 \pm 0.0078$ & $0.0337 \pm 0.0051$ & $0.0362 \pm 0.0049$ \\ 
\addlinespace
0.03 & (0.7, 1.1) & $1.3106 \pm 0.10$ & $0.4616 \pm 0.13$ & $0.0372 \pm 0.0061$ & $\mathbf{0.0271 \pm 0.0056}$ & $0.0350 \pm 0.0079$ \\ 
     & (0.7, 1.2) & $1.2172 \pm 0.094$ & $0.3350 \pm 0.014$ & $0.0369 \pm 0.0063$ & $0.0277 \pm 0.0063$ & $0.0350 \pm 0.0079$ \\ 
     & (0.7, 1.3) & $1.1833 \pm 0.13$ & $0.2856 \pm 0.050$ & $0.0381 \pm 0.0069$ & $0.0282 \pm 0.0069$ & $0.0350 \pm 0.0079$ \\ 
     & (0.8, 1.1) & $1.3519 \pm 0.055$ & $0.3863 \pm 0.16$ & $0.0378 \pm 0.0068$ & $0.0279 \pm 0.0065$ & $0.0350 \pm 0.0079$ \\ 
     & (0.8, 1.2) & $1.2831 \pm 0.11$ & $0.2857 \pm 0.083$ & $0.0369 \pm 0.0063$ & $0.0277 \pm 0.0063$ & $0.0350 \pm 0.0079$ \\ 
     & (0.8, 1.3) & $1.2776 \pm 0.15$ & $0.2298 \pm 0.082$ & $0.0386 \pm 0.0075$ & $0.0282 \pm 0.0069$ & $0.0350 \pm 0.0079$ \\ 
     & (0.9, 1.1) & $1.4075 \pm 0.087$ & $0.2794 \pm 0.097$ & $0.0386 \pm 0.0081$ & $0.0286 \pm 0.0075$ & $0.0350 \pm 0.0079$ \\ 
     & (0.9, 1.2) & $1.2831 \pm 0.11$ & $0.2561 \pm 0.072$ & $0.0386 \pm 0.0087$ & $0.0291 \pm 0.0081$ & $0.0350 \pm 0.0079$ \\ 
     & (0.9, 1.3) & $1.2776 \pm 0.15$ & $0.2524 \pm 0.060$ & $0.0386 \pm 0.0075$ & $0.0282 \pm 0.0069$ & $0.0350 \pm 0.0079$ \\ 
\addlinespace
0.05 & (0.7, 1.1) & $0.8157 \pm 0.16$ & $0.4144 \pm 0.073$ & $0.0387 \pm 0.0072$ & $0.0422 \pm 0.012$ & $0.0542 \pm 0.021$ \\ 
     & (0.7, 1.2) & $0.8247 \pm 0.096$ & $0.3666 \pm 0.042$ & $0.0387 \pm 0.0072$ & $0.0422 \pm 0.012$ & $0.0542 \pm 0.021$ \\ 
     & (0.7, 1.3) & $0.7834 \pm 0.25$ & $0.3196 \pm 0.11$ & $0.0387 \pm 0.0072$ & $0.0422 \pm 0.012$ & $0.0542 \pm 0.021$ \\ 
     & (0.8, 1.1) & $0.8980 \pm 0.16$ & $0.3613 \pm 0.080$ & $0.0387 \pm 0.0072$ & $0.0422 \pm 0.012$ & $0.0542 \pm 0.021$ \\ 
     & (0.8, 1.2) & $0.8931 \pm 0.18$ & $0.3294 \pm 0.12$ & $0.0387 \pm 0.0072$ & $0.0422 \pm 0.012$ & $0.0542 \pm 0.021$ \\ 
     & (0.8, 1.3) & $0.8793 \pm 0.17$ & $0.3258 \pm 0.11$ & $0.0387 \pm 0.0072$ & $0.0422 \pm 0.012$ & $0.0542 \pm 0.021$ \\ 
     & (0.9, 1.1) & $1.0118 \pm 0.15$ & $0.2075 \pm 0.050$ & $0.0387 \pm 0.0072$ & $0.0422 \pm 0.012$ & $0.0542 \pm 0.021$ \\ 
     & (0.9, 1.2) & $0.8931 \pm 0.18$ & $0.2060 \pm 0.035$ & $0.0387 \pm 0.0072$ & $0.0422 \pm 0.012$ & $0.0542 \pm 0.021$ \\ 
     & (0.9, 1.3) & $0.8793 \pm 0.17$ & $0.2766 \pm 0.060$ & $0.0387 \pm 0.0072$ & $0.0422 \pm 0.012$ & $0.0542 \pm 0.021$ \\ 
\bottomrule
\end{tabular}
}

\end{table}

\begin{table}[htbp]
\centering
\caption{The averaged test errors generated by SRPA under different $ (\rho_{1}, \eta_{1}, \eta_{2}) $ and initialization strategies on the TIMIT dataset: the number of iterations and repetitions are 100 and 3, respectively. 
Means and standard deviations are approximated as 3 decimal places and 2 significant figures, respectively.}
\label{tab:SRPA-tune-3}
 
\resizebox{0.7\textwidth}{!}{%
\begin{tabular}{l l l l l}
\toprule
$\rho_{1}$ & $(\eta_{1}, \eta_{2})$ & $\mathcal{N}(0, 10^{-3})$ & $\mathcal{N}(0, 10^{-2})$ & $\mathcal{N}(0, 10^{-1})$ \\ 
\midrule
0.001 & (0.7, 1.1) & $5.042 \pm 0.12$ & $4.780 \pm 0.049$ & $5.052 \pm 0.078$ \\ 
      & (0.7, 1.3) & $5.149 \pm 0.059$ & $4.886 \pm 0.14$ & $5.042 \pm 0.073$ \\ 
      & (0.7, 1.5) & $5.182 \pm 0.058$ & $4.918 \pm 0.16$ & $5.020 \pm 0.080$ \\ 
      & (0.8, 1.1) & $5.040 \pm 0.12$ & $4.849 \pm 0.047$ & $5.069 \pm 0.084$ \\ 
      & (0.8, 1.3) & $5.168 \pm 0.089$ & $4.927 \pm 0.096$ & $5.046 \pm 0.074$ \\ 
      & (0.8, 1.5) & $5.186 \pm 0.055$ & $4.953 \pm 0.14$ & $5.066 \pm 0.037$ \\ 
      & (0.9, 1.1) & $5.058 \pm 0.13$ & $4.863 \pm 0.087$ & $5.093 \pm 0.085$ \\ 
      & (0.9, 1.3) & $5.189 \pm 0.074$ & $4.956 \pm 0.079$ & $5.059 \pm 0.070$ \\ 
      & (0.9, 1.5) & $5.196 \pm 0.048$ & $4.979 \pm 0.14$ & $5.101 \pm 0.083$ \\ 
\addlinespace
0.003 & (0.7, 1.1) & $5.031 \pm 0.20$ & $\mathbf{4.768 \pm 0.054}$ & $5.008 \pm 0.088$ \\ 
      & (0.7, 1.3) & $5.193 \pm 0.12$ & $4.863 \pm 0.16$ & $5.020 \pm 0.10$ \\ 
      & (0.7, 1.5) & $5.204 \pm 0.099$ & $4.926 \pm 0.21$ & $5.041 \pm 0.040$ \\ 
      & (0.8, 1.1) & $5.106 \pm 0.19$ & $4.774 \pm 0.081$ & $5.032 \pm 0.078$ \\ 
      & (0.8, 1.3) & $5.216 \pm 0.10$ & $4.902 \pm 0.20$ & $5.047 \pm 0.067$ \\ 
      & (0.8, 1.5) & $5.208 \pm 0.098$ & $4.963 \pm 0.17$ & $5.073 \pm 0.076$ \\ 
      & (0.9, 1.1) & $5.111 \pm 0.19$ & $4.800 \pm 0.097$ & $5.056 \pm 0.075$ \\ 
      & (0.9, 1.3) & $5.189 \pm 0.070$ & $4.939 \pm 0.19$ & $5.047 \pm 0.067$ \\ 
      & (0.9, 1.5) & $5.217 \pm 0.096$ & $5.028 \pm 0.10$ & $5.073 \pm 0.076$ \\ 
\addlinespace
0.005 & (0.7, 1.1) & $5.024 \pm 0.13$ & $4.882 \pm 0.13$ & $5.010 \pm 0.079$ \\ 
      & (0.7, 1.3) & $5.161 \pm 0.078$ & $4.889 \pm 0.22$ & $5.015 \pm 0.094$ \\ 
      & (0.7, 1.5) & $5.235 \pm 0.11$ & $4.951 \pm 0.17$ & $5.012 \pm 0.086$ \\ 
      & (0.8, 1.1) & $5.078 \pm 0.16$ & $4.896 \pm 0.10$ & $5.022 \pm 0.088$ \\ 
      & (0.8, 1.3) & $5.218 \pm 0.10$ & $4.940 \pm 0.19$ & $5.043 \pm 0.065$ \\ 
      & (0.8, 1.5) & $5.244 \pm 0.092$ & $4.956 \pm 0.18$ & $5.057 \pm 0.036$ \\ 
      & (0.9, 1.1) & $5.101 \pm 0.16$ & $4.871 \pm 0.14$ & $5.036 \pm 0.089$ \\ 
      & (0.9, 1.3) & $5.225 \pm 0.091$ & $4.945 \pm 0.20$ & $5.042 \pm 0.065$ \\ 
      & (0.9, 1.5) & $5.255 \pm 0.11$ & $4.989 \pm 0.20$ & $5.094 \pm 0.078$ \\ 
\bottomrule
\end{tabular}
} 

\end{table}

The average test errors generated by GD, GDC, GDNes, SGD, and Adam are listed in Tables \ref{tab:GD-tune-1}-\ref{tab:GD-tune-3}, \ref{tab:GDC-tune-1}-\ref{tab:GDC-tune-3}, \ref{tab:GDNes-tune-1}-\ref{tab:GDNes-tune-3}, \ref{tab:SGD-tune-1}-\ref{tab:SGD-tune-3}, and \ref{tab:Adam-tune-1}-\ref{tab:Adam-tune-3}, respectively, where an epoch is one complete pass of the entire training dataset through the learning algorithm.
In Tables \ref{tab:GD-tune-1}-\ref{tab:Adam-tune-3}, the best test errors are highlighted in bold, which indicate their best parameters and initialization strategies.

\begin{table}[htbp]
\centering
\caption{The averaged test errors generated by GD under different learning rates and initialization strategies on the Synthetic dataset. The number of epochs and repetitions are 50 and 10, respectively.
Means and standard deviations are approximated as 3 decimal places and 2 significant figures, respectively.}
\label{tab:GD-tune-1}
 
\resizebox{\textwidth}{!}{%
\begin{tabular}{l l l l l l l}
\toprule
\textbf{lr\textsuperscript{*}} 
& $\mathcal{N}(0, 10^{-3})$ & $\mathcal{N}(0, 10^{-2})$ & $\mathcal{N}(0, 10^{-1})$ & He & Glorot & LeCun \\ 
\midrule
$10^{-4}$ & $6.288 \pm 5.6\text{e-}6$ & $6.287 \pm 5.5\text{e-}4$ & $\mathbf{6.274 \pm 0.053}$ & $26.331 \pm 32$ & $7.090 \pm 1.6$ & $6.890 \pm 1.4$ \\ 
$10^{-3}$ & $6.327 \pm 5.5\text{e-}6$ & $6.327 \pm 5.5\text{e-}4$ & $6.314 \pm 0.052$ & $9.019 \pm 3.2$ & $6.673 \pm 1.4$ & $6.572 \pm 1.2$ \\ 
$10^{-2}$ & $6.617 \pm 4.8\text{e-}6$ & $6.617 \pm 4.8\text{e-}4$ & $6.603 \pm 0.047$ & $6.932 \pm 1.8$ & $6.516 \pm 1.1$ & $6.499 \pm 0.99$ \\ 
$10^{-1}$ & $6.939 \pm 4.8\text{e-}6$ & $6.939 \pm 4.8\text{e-}4$ & $6.909 \pm 0.046$ & $6.808 \pm 0.70$ & $6.558 \pm 0.52$ & $6.570 \pm 0.49$ \\ 
$1$ & $6.706 \pm 0.18$ & $6.466 \pm 0.31$ & $6.317 \pm 0.35$ & \text{NaN}\textsuperscript{**} & $6.440 \pm 0.47$ & $6.490 \pm 0.47$ \\ 
\bottomrule
\end{tabular}
}

\parbox{\textwidth}{
\textsuperscript{*} ``lr'' refers to the learning rate. \\
\textsuperscript{**} ``NaN'' indicates a numerical overflow has occurred.
}

\end{table}

\begin{table}[htbp]
\centering
\caption{The averaged test errors generated by GD under different learning rates and initialization strategies on the Volatility of S\&P Index dataset. The number of epochs and repetitions are 100 and 3, respectively.
Means and standard deviations are approximated as 4 decimal places and 2 significant figures, respectively.}
\label{tab:GD-tune-2}
 
\resizebox{\textwidth}{!}{%
\begin{tabular}{l l l l l l l}
\toprule
\textbf{lr\textsuperscript{*}} 
& $\mathcal{N}(0, 10^{-3})$ & $\mathcal{N}(0, 10^{-2})$ & $\mathcal{N}(0, 10^{-1})$ & He & Glorot & LeCun \\ 
\midrule
$10^{-4}$ & $1.5749 \pm 9.4\text{e-}6$ & $1.5748 \pm 9.5\text{e-}4$ & $1.5588 \pm 0.12$ & $7.9668 \pm 3.2$ & $3.1767 \pm 2.1$ & $2.7084 \pm 1.9$ \\ 
$10^{-3}$ & $1.6252 \pm 1.1\text{e-}5$ & $1.6232 \pm 0.0011$ & $1.4206 \pm 0.15$ & $2.7643 \pm 2.2$ & $1.6400 \pm 1.5$ & $1.6789 \pm 1.5$ \\ 
$10^{-2}$ & $1.7398 \pm 1.8\text{e-}4$ & $1.6192 \pm 0.018$ & $0.1719 \pm 0.13$ & \text{NaN}\textsuperscript{**} & $0.2174 \pm 0.10$ & $0.2195 \pm 0.13$ \\ 
$10^{-1}$ & $0.0234 \pm 0.0010$ & $\mathbf{0.0228 \pm 8\text{e-}4}$ & $0.0415 \pm 0.0024$ & \text{NaN}\textsuperscript{**} & $0.0439 \pm 0.0062$ & $0.0323 \pm 0.0067$ \\ 
$1$ & \text{NaN}\textsuperscript{**} & \text{NaN}\textsuperscript{**} & \text{NaN}\textsuperscript{**} & \text{NaN}\textsuperscript{**} & \text{NaN}\textsuperscript{**} & \text{NaN}\textsuperscript{**} \\ 
\bottomrule
\end{tabular}
}

\parbox{\textwidth}{
\textsuperscript{*} ``lr'' refers to the learning rate. \\
\textsuperscript{**} ``NaN'' indicates a numerical overflow has occurred.
}

\end{table}

\begin{table}[htbp]
\centering
\caption{The averaged test errors generated by GD under different learning rates and initialization strategies on the TIMIT dataset. The number of epochs and repetitions are 100 and 3, respectively.
Means and standard deviations are approximated as 3 decimal places and 2 significant figures, respectively.}
\label{tab:GD-tune-3}

\small
 
\begin{tabular}{l l l l}
\toprule
\textbf{lr\textsuperscript{*}} 
& $\mathcal{N}(0, 10^{-3})$ & $\mathcal{N}(0, 10^{-2})$ & $\mathcal{N}(0, 10^{-1})$ \\ 
\midrule
$10^{-4}$ & $5.927 \pm 3.2\text{e-}6$ & $5.927 \pm 3.2\text{e-}4$ & $6.723 \pm 0.42$ \\ 
$10^{-3}$ & $5.925 \pm 3.2\text{e-}6$ & $5.925 \pm 3.1\text{e-}4$ & $6.602 \pm 0.32$ \\ 
$10^{-2}$ & $5.906 \pm 2.8\text{e-}6$ & $5.907 \pm 2.6\text{e-}4$ & $6.142 \pm 0.061$ \\ 
$10^{-1}$ & $5.743 \pm 3.3\text{e-}6$ & $5.743 \pm 2.5\text{e-}4$ & $5.766 \pm 0.0076$ \\ 
$1$ & $5.146 \pm 9.4\text{e-}4$ & $\mathbf{5.131 \pm 0.021}$ & $5.147 \pm 0.0033$ \\ 
\bottomrule
\end{tabular}

\end{table}

\begin{table}[htbp]
\centering
\caption{The averaged test errors generated by GDC under different parameters and initialization strategies on the Synthetic dataset. The number of epochs and repetitions are 50 and 10, respectively.
Means and standard deviations are approximated as 3 decimal places and 2 significant figures, respectively.}
\label{tab:GDC-tune-1}
 
\resizebox{\textwidth}{!}{%
\begin{tabular}{l l l l l l l l}
\toprule
\textbf{lr\textsuperscript{*}} & \textbf{cln\textsuperscript{*}} & $\mathcal{N}(0, 10^{-3})$ & $\mathcal{N}(0, 10^{-2})$ & $\mathcal{N}(0, 10^{-1})$ & He & Glorot & LeCun \\ 
\midrule
$10^{-4}$ & 0.5 & $6.288 \pm 5.6\text{e-}6$ & $6.287 \pm 5.5\text{e-}4$ & $\mathbf{6.274 \pm 0.053}$ & $142.275 \pm 3.0\text{e}2$ & $7.188 \pm 1.7$ & $6.952 \pm 1.5$ \\
  & 1.0 & $6.288 \pm 5.6\text{e-}6$ & $6.287 \pm 5.5\text{e-}4$ & $\mathbf{6.274 \pm 0.053}$ & $136.023 \pm 2.9\text{e}2$ & $7.156 \pm 1.7$ & $6.927 \pm 1.4$ \\
  & 2.0 & $6.288 \pm 5.6\text{e-}6$ & $6.287 \pm 5.5\text{e-}4$ & $\mathbf{6.274 \pm 0.053}$ & $124.333 \pm 2.6\text{e}2$ & $7.106 \pm 1.6$ & $6.891 \pm 1.4$ \\
  & 4.0 & $6.288 \pm 5.6\text{e-}6$ & $6.287 \pm 5.5\text{e-}4$ & $\mathbf{6.274 \pm 0.053}$ & $104.463 \pm 2.1\text{e}2$ & $7.090 \pm 1.6$ & $6.890 \pm 1.4$ \\
\addlinespace
$10^{-3}$ & 0.5 & $6.327 \pm 5.5\text{e-}6$ & $6.327 \pm 5.5\text{e-}4$ & $6.314 \pm 0.052$ & $95.280 \pm 1.9\text{e}2$ & $6.927 \pm 1.5$ & $6.754 \pm 1.3$ \\
  & 1.0 & $6.327 \pm 5.5\text{e-}6$ & $6.327 \pm 5.5\text{e-}4$ & $6.314 \pm 0.052$ & $62.596 \pm 1.2\text{e}2$ & $6.764 \pm 1.4$ & $6.628 \pm 1.3$ \\
  & 2.0 & $6.327 \pm 5.5\text{e-}6$ & $6.327 \pm 5.5\text{e-}4$ & $6.314 \pm 0.052$ & $29.675 \pm 43$ & $6.675 \pm 1.4$ & $6.572 \pm 1.2$ \\
  & 4.0 & $6.327 \pm 5.5\text{e-}6$ & $6.327 \pm 5.5\text{e-}4$ & $6.314 \pm 0.052$ & $11.409 \pm 5.4$ & $6.673 \pm 1.4$ & $6.572 \pm 1.2$ \\
\addlinespace
$10^{-2}$ & 0.5 & $6.617 \pm 4.8\text{e-}6$ & $6.617 \pm 4.8\text{e-}4$ & $6.603 \pm 0.047$ & $7.767 \pm 2.2$ & $6.543 \pm 1.1$ & $6.517 \pm 1.0$ \\
  & 1.0 & $6.617 \pm 4.8\text{e-}6$ & $6.617 \pm 4.8\text{e-}4$ & $6.603 \pm 0.047$ & $6.979 \pm 1.8$ & $6.519 \pm 1.1$ & $6.501 \pm 0.99$ \\
  & 2.0 & $6.617 \pm 4.8\text{e-}6$ & $6.617 \pm 4.8\text{e-}4$ & $6.603 \pm 0.047$ & $6.913 \pm 1.8$ & $6.516 \pm 1.1$ & $6.499 \pm 0.99$ \\
  & 4.0 & $6.617 \pm 4.8\text{e-}6$ & $6.617 \pm 4.8\text{e-}4$ & $6.603 \pm 0.047$ & $6.898 \pm 1.8$ & $6.516 \pm 1.1$ & $6.499 \pm 0.99$ \\
\addlinespace
$10^{-1}$ & 0.5 & $6.939 \pm 4.8\text{e-}6$ & $6.939 \pm 4.8\text{e-}4$ & $6.909 \pm 0.046$ & $6.578 \pm 0.79$ & $6.564 \pm 0.53$ & $6.575 \pm 0.49$ \\
  & 1.0 & $6.939 \pm 4.8\text{e-}6$ & $6.939 \pm 4.8\text{e-}4$ & $6.909 \pm 0.046$ & $6.576 \pm 0.76$ & $6.560 \pm 0.52$ & $6.574 \pm 0.49$ \\
  & 2.0 & $6.939 \pm 4.8\text{e-}6$ & $6.939 \pm 4.8\text{e-}4$ & $6.909 \pm 0.046$ & $6.586 \pm 0.75$ & $6.559 \pm 0.52$ & $6.572 \pm 0.49$ \\
  & 4.0 & $6.939 \pm 4.8\text{e-}6$ & $6.939 \pm 4.8\text{e-}4$ & $6.909 \pm 0.046$ & $6.548 \pm 0.71$ & $6.558 \pm 0.52$ & $6.570 \pm 0.49$ \\
\addlinespace
$1$ & 0.5 & $6.706 \pm 0.18$ & $6.466 \pm 0.31$ & $6.317 \pm 0.35$ & $6.426 \pm 0.44$ & $6.457 \pm 0.46$ & $6.506 \pm 0.44$ \\
  & 1.0 & $6.706 \pm 0.18$ & $6.466 \pm 0.31$ & $6.317 \pm 0.35$ & $6.415 \pm 0.44$ & $6.426 \pm 0.49$ & $6.507 \pm 0.47$ \\
  & 2.0 & $6.706 \pm 0.18$ & $6.466 \pm 0.31$ & $6.317 \pm 0.35$ & $6.485 \pm 0.47$ & $6.436 \pm 0.47$ & $6.492 \pm 0.47$ \\
  & 4.0 & $6.706 \pm 0.18$ & $6.466 \pm 0.31$ & $6.317 \pm 0.35$ & $6.446 \pm 0.43$ & $6.440 \pm 0.47$ & $6.490 \pm 0.47$ \\
\bottomrule
\end{tabular}
}

\parbox{\textwidth}{
\textsuperscript{*} ``lr'' and ``cln'' refer to the learning rate and clipping norm.  
}

\end{table}

\begin{table}[htbp]
\centering
\caption{The averaged test errors generated by GDC under different parameters and initialization strategies on the Volatility of S\&P Index dataset. The number of epochs and repetitions are 100 and 3, respectively.
Means and standard deviations are approximated as 4 decimal places and 2 significant figures, respectively. Excessively large means are shown as 2 significant figures.}
\label{tab:GDC-tune-2}
 
\resizebox{\textwidth}{!}{%
\begin{tabular}{l l l l l l l l}
\toprule
\textbf{lr\textsuperscript{*}} & \textbf{cln\textsuperscript{*}} & $\mathcal{N}(0, 10^{-3})$ & $\mathcal{N}(0, 10^{-2})$ & $\mathcal{N}(0, 10^{-1})$ & He & Glorot & LeCun \\ 
\midrule
$10^{-4}$ & 0.5 & $1.5749 \pm 9.4\text{e-}6$ & $1.5748 \pm 9.5\text{e-}4$ & $1.5662 \pm 0.12$ & $36.4861 \pm 38$ & 	$3.5685 \pm 2.1$ & 	$2.9456 \pm 1.9$ \\
  & 1.0 & $1.5749 \pm 9.4\text{e-}6$ & $1.5748 \pm 9.5\text{e-}4$ & $1.5589 \pm 0.12$ & $32.7838 \pm 33$ & 	$3.4803 \pm 2.1$ & 	$2.8840 \pm 1.9$ \\
  & 2.0 & $1.5749 \pm 9.4\text{e-}6$ & $1.5748 \pm 9.5\text{e-}4$ & $1.5588 \pm 0.12$ & $26.5857 \pm 24$ & 	$3.3319 \pm 2.1$ & 	$2.7815 \pm 1.9$ \\
  & 4.0 & $1.5749 \pm 9.4\text{e-}6$ & $1.5748 \pm 9.5\text{e-}4$ & $1.5588 \pm 0.12$ & $19.5192 \pm 15$ & 	$3.2107 \pm 2.1$ & 	$2.7161 \pm 1.9$ \\
\addlinespace
$10^{-3}$ & 0.5 & $1.6252 \pm 1.1\text{e-}5$ & $1.6232 \pm 0.0011$ & $1.4963 \pm 0.13$ & $17.0220 \pm 12$ & 	$2.8973 \pm 2.1$ & 	$2.4769 \pm 1.9$ \\
  & 1.0 & $1.6252 \pm 1.1\text{e-}5$ & $1.6232 \pm 0.0011$ & $1.4213 \pm 0.15$ & $10.1935 \pm 4.3$ & 	$2.4518 \pm 2.0 $ & 	$2.1844 \pm 1.8$ \\
  & 2.0 & $1.6252 \pm 1.1\text{e-}5$ & $1.6232 \pm 0.0011$ & $1.4206 \pm 0.15$ & $4.8447 \pm 2.0 $ & 	$1.9756 \pm 1.6$ & 	$1.8390 \pm 1.5$ \\
  & 4.0 & $1.6252 \pm 1.1\text{e-}5$ & $1.6232 \pm 0.0011$ & $1.4206 \pm 0.15$ &$3.2491 \pm 2.1$ & 	$1.6752 \pm 1.5$ & 	$1.6926 \pm 1.5$ \\
\addlinespace
$10^{-2}$ & 0.5 & $1.7398 \pm 1.8\text{e-}4$ & $1.6193 \pm 0.018$ & $0.2436 \pm 0.18$ & $1.8189 \pm 1.0 $ & 	$0.5180 \pm 0.35$ & 	$0.4778 \pm 0.33$ \\
  & 1.0 & $1.7398 \pm 1.8\text{e-}4$ & $1.6192 \pm 0.018$ & $0.1723 \pm 0.13$ & $0.9377 \pm 0.39$ & 	$0.2613 \pm 0.14$ & 	$0.2566 \pm 0.15$ \\
  & 2.0 & $1.7398 \pm 1.8\text{e-}4$ & $1.6192 \pm 0.018$ & $0.1719 \pm 0.13$ & $0.8079 \pm 0.35$ & 	$0.2208 \pm 0.10$ & 	$0.2194 \pm 0.13$ \\
  & 4.0 & $1.7398 \pm 1.8\text{e-}4$ & $1.6192 \pm 0.018$ & $0.1719 \pm 0.13$ & $0.5461 \pm 0.13$ & 	$0.2181 \pm 0.10$ & 	$0.2177 \pm 0.13$  \\
\addlinespace
$10^{-1}$ & 0.5 & $0.0236 \pm 0.0012$ & $\mathbf{0.0224 \pm 9.5\text{e-}4}$ & $0.0398 \pm 0.0029$ & $0.3150 \pm 0.19$ & 	$0.0306 \pm 0.013$ & 	$0.0528 \pm 0.023$ \\
  & 1.0 & $0.0234 \pm 0.0010$ & $0.0228 \pm 0.0008$ & $0.0415 \pm 0.0024$ & $0.0928 \pm 0.023$ & 	$0.0537 \pm 0.021$ & 	$0.0538 \pm 0.029$ \\
  & 2.0 & $0.0234 \pm 0.0010$ & $0.0228 \pm 0.0008$ & $0.0415 \pm 0.0024$ & $0.1156 \pm 0.018$ & 	$0.0402 \pm 0.016$ & 	$0.0296 \pm 0.0014$ \\
  & 4.0 & $0.0234 \pm 0.0010$ & $0.0228 \pm 0.0008$ & $0.0415 \pm 0.0024$ & $0.1392 \pm 0.11$ & 	$0.0395 \pm 0.0096$ & 	$0.0329 \pm 0.0073$ \\
\addlinespace
$1$ & 0.5 & $0.0781 \pm 0.067$ & $0.1820 \pm 0.051$ & $0.1196 \pm 0.025$ & $0.3024 \pm 0.19$ & 	$0.3630 \pm 0.33$ & 	$0.1343 \pm 0.024$ \\
  & 1.0 & $0.1716 \pm 0.061$ & $0.3682 \pm 0.47$ & $0.6027 \pm 0.23$ & $1.3722 \pm 0.69$ & 	$0.8726 \pm 0.87$ & 	$1.1417 \pm 1.3$ \\
  & 2.0 & $25.0467 \pm 34$ & $0.4549 \pm 0.17$ & $4.9552 \pm 5.2$ & \text{NaN}\textsuperscript{**} & \text{NaN}\textsuperscript{**} & $6.0\text{e}31 \pm 8.5\text{e}31$ \\
  & 4.0 & \text{NaN}\textsuperscript{**} & \text{NaN}\textsuperscript{**} & \text{NaN}\textsuperscript{**} & \text{NaN}\textsuperscript{**} & \text{NaN}\textsuperscript{**} & \text{NaN}\textsuperscript{**} \\
\bottomrule
\end{tabular}
}

\parbox{\textwidth}{
\textsuperscript{*} ``lr'' and ``cln'' refer to the learning rate and clipping norm. \\
\textsuperscript{**} ``NaN'' indicates a numerical overflow. 
}

\end{table}

\begin{table}[htbp]
\centering
\caption{The averaged test errors generated by GDC under different parameters and initialization strategies on the TIMIT dataset. The number of epochs and repetitions are 100 and 3, respectively.
Means and standard deviations are approximated as 3 decimal places and 2 significant figures, respectively.}
\label{tab:GDC-tune-3}
 
\small 

\begin{tabular}{l l l l l}
\toprule
\textbf{lr\textsuperscript{*}} & \textbf{cln\textsuperscript{*}} & $\mathcal{N}(0, 10^{-3})$ & $\mathcal{N}(0, 10^{-2})$ & $\mathcal{N}(0, 10^{-1})$ \\ 
\midrule
$10^{-4}$ & Any\textsuperscript{**} & $5.927 \pm 3.2\text{e-}6$ & $5.927 \pm 3.2\text{e-}4$ & $6.723 \pm 0.42$ \\
$10^{-3}$ & Any\textsuperscript{**} & $5.925 \pm 3.2\text{e-}6$ & $5.925 \pm 3.1\text{e-}4$ & $6.602 \pm 0.32$ \\
$10^{-2}$ & Any\textsuperscript{**} & $5.906 \pm 2.8\text{e-}6$ & $5.907 \pm 2.6\text{e-}4$ & $6.142 \pm 0.061$ \\
$10^{-1}$ & Any\textsuperscript{**} & $5.743 \pm 3.3\text{e-}6$ & $5.743 \pm 2.5\text{e-}4$ & $5.766 \pm 0.0076$ \\
$1$ & Any\textsuperscript{**} & $5.146 \pm 9.4\text{e-}4$ & $\mathbf{5.131 \pm 0.021}$ & $5.147 \pm 0.0033$ \\
\bottomrule
\end{tabular}

\parbox{\textwidth}{\scriptsize
\textsuperscript{*} ``lr'' and ``cln'' refer to the learning rate and clipping norm. \\
\textsuperscript{**} ``Any'' refers to any clipping norm in $\{0.5, 1, 2, 4\}$.
}

\end{table}

\begin{table}[htbp]
\centering
\caption{The averaged test errors generated by GDNes under different learning rates and initialization strategies on the Synthetic dataset. The number of epochs and repetitions are 50 and 10, respectively.
Means and standard deviations are approximated as 3 decimal places and 2 significant figures, respectively.}
\label{tab:GDNes-tune-1}

\small 

\begin{tabular}{l l l l l l l}
\toprule
\textbf{lr\textsuperscript{*}} 
& $\mathcal{N}(0, 10^{-3})$ & $\mathcal{N}(0, 10^{-2})$ & $\mathcal{N}(0, 10^{-1})$ & He & Glorot & LeCun \\ 
\midrule
$10^{-4}$ & $6.321 \pm 5.5\text{e-}6$ & $6.320 \pm 5.5\text{e-}4$ & $6.307 \pm 0.052$ & $9.000 \pm 3.3$ & $6.684 \pm 1.4$ & $6.582 \pm 1.3$ \\ 
$10^{-3}$ & $6.598 \pm 4.9\text{e-}6$ & $6.597 \pm 4.9\text{e-}4$ & $6.584 \pm 0.047$ & $7.048 \pm 1.8$ & $6.523 \pm 1.1$ & $6.507 \pm 0.99$ \\ 
$10^{-2}$ & $6.925 \pm 4.2\text{e-}6$ & $6.925 \pm 4.2\text{e-}4$ & $6.900 \pm 0.041$ & $6.698 \pm 0.64$ & $6.588 \pm 0.53$ & $6.593 \pm 0.51$ \\ 
$10^{-1}$ & $6.936 \pm 0.0030$ & $6.744 \pm 0.16$ & $6.353 \pm 0.37$ & \text{NaN}\textsuperscript{**} & $\mathbf{6.277 \pm 0.37}$ & $6.299 \pm 0.35$ \\ 
$1$ & \text{NaN}\textsuperscript{**} & \text{NaN}\textsuperscript{**} & \text{NaN}\textsuperscript{**} & \text{NaN}\textsuperscript{**} & \text{NaN}\textsuperscript{**} & \text{NaN}\textsuperscript{**} \\ 
\bottomrule
\end{tabular}

\parbox{\textwidth}{
\textsuperscript{*} ``lr'' refers to the learning rate. \\
\textsuperscript{**} ``NaN'' indicates a numerical overflow has occurred.
}

\end{table}

\begin{table}[htbp]
\centering
\caption{The averaged test errors generated by GDNes under different learning rates and initialization strategies on the Volatility of S\&P Index dataset. The number of epochs and repetitions are 100 and 3, respectively.
Means and standard deviations are approximated as 4 decimal places and 2 significant figures, respectively.}
\label{tab:GDNes-tune-2}

\resizebox{\textwidth}{!}{%
\begin{tabular}{l l l l l l l}
\toprule
\textbf{lr\textsuperscript{*}} 
& $\mathcal{N}(0, 10^{-3})$ & $\mathcal{N}(0, 10^{-2})$ & $\mathcal{N}(0, 10^{-1})$ & He & Glorot & LeCun \\ 
\midrule
$10^{-4}$ & $1.6226 \pm 1.1\text{e-}5$ & $1.6208 \pm 0.0011$ & $1.4363 \pm 0.14$ & $2.7727 \pm 2.2$ & $1.6792 \pm 1.6$ & $1.7277 \pm 1.6$ \\ 
$10^{-3}$ & $1.7446 \pm 1.3\text{e-}4$ & $1.6802 \pm 0.013$ & $0.1790 \pm 0.16$ & \text{NaN}\textsuperscript{**} & $0.2274 \pm 0.15$ & $0.2352 \pm 0.15$ \\ 
$10^{-2}$ & $0.0324 \pm 0.0013$ & $0.0302 \pm 0.0011$ & $0.0425 \pm 0.0022$ & \text{NaN}\textsuperscript{**} & $0.0384 \pm 0.016$ & $0.0361 \pm 0.0082$ \\ 
$10^{-1}$ & $0.0487 \pm 0.0072$ & $0.0443 \pm 0.013$ & $\mathbf{0.0235 \pm 9.7\text{e-}4}$ & \text{NaN}\textsuperscript{**} & \text{NaN}\textsuperscript{**} & $0.3281 \pm 0.42$ \\ 
$1$ & \text{NaN}\textsuperscript{**} & \text{NaN}\textsuperscript{**} & \text{NaN}\textsuperscript{**} & \text{NaN}\textsuperscript{**} & \text{NaN}\textsuperscript{**} & \text{NaN}\textsuperscript{**} \\ 
\bottomrule
\end{tabular}
}

\parbox{\textwidth}{
\textsuperscript{*} ``lr'' refers to the learning rate. \\
\textsuperscript{**} ``NaN'' indicates a numerical overflow has occurred.
}

\end{table}

\begin{table}[htbp]
\centering
\caption{The averaged test errors generated by GDNes under different learning rates and initialization strategies on the TIMIT dataset. The number of epochs and repetitions are 100 and 3, respectively.
Means and standard deviations are approximated as 3 decimal places and 2 significant figures, respectively.}
\label{tab:GDNes-tune-3}

\small 

\begin{tabular}{l l l l}
\toprule
\textbf{lr\textsuperscript{*}} 
& $\mathcal{N}(0, 10^{-3})$ & $\mathcal{N}(0, 10^{-2})$ & $\mathcal{N}(0, 10^{-1})$ \\ 
\midrule
$10^{-4}$ & $5.925 \pm 3.2\text{e-}6$ & $5.926 \pm 3.1\text{e-}4$ & $6.611 \pm 0.33$ \\ 
$10^{-3}$ & $5.908 \pm 2.8\text{e-}6$ & $5.908 \pm 2.7\text{e-}4$ & $6.146 \pm 0.062$ \\ 
$10^{-2}$ & $5.754 \pm 2.9\text{e-}6$ & $5.754 \pm 2.2\text{e-}4$ & $5.774 \pm 0.0072$ \\ 
$10^{-1}$ & $5.137 \pm 2.6\text{e-}4$ & $5.129 \pm 0.0096$ & $5.136 \pm 0.0036$ \\ 
$1$ & $4.994 \pm 0.018$ & $\mathbf{4.716 \pm 0.043}$ & $5.007 \pm 0.036$ \\ 
\bottomrule
\end{tabular}

\parbox{\textwidth}{
\textsuperscript{*} ``lr'' refers to the learning rate.  
}

\end{table}

\begin{table}[htbp]
\centering
\caption{The averaged test errors generated by SGD under different parameters and initialization strategies on the Synthetic dataset. The number of epochs and repetitions are 50 and 10, respectively.
Means and standard deviations are approximated as 3 decimal places and 2 significant figures, respectively.}
\label{tab:SGD-tune-1}

\resizebox{\textwidth}{!}{%
\begin{tabular}{l l l l l l l l}
\toprule
\textbf{lr\textsuperscript{*}} & \textbf{bs\textsuperscript{*}} & $\mathcal{N}(0, 10^{-3})$ & $\mathcal{N}(0, 10^{-2})$ & $\mathcal{N}(0, 10^{-1})$ & He & Glorot & LeCun \\ 
\midrule
$10^{-4}$ & 1 & $6.303 \pm 8.9\text{e-}4$ & $6.303 \pm 0.0010$ & $6.290 \pm 0.052$ & $116.126 \pm 2.4\text{e}2$ & $7.076 \pm 1.6$ & $6.872 \pm 1.4$ \\
  & 2 & $6.298 \pm 5.3\text{e-}4$ & $6.298 \pm 5.9\text{e-}4$ & $6.285 \pm 0.052$ & $117.704 \pm 2.5\text{e}2$ & $7.124 \pm 1.6$ & $6.912 \pm 1.4$ \\
  & 4 & $6.292 \pm 2.7\text{e-}4$ & $6.292 \pm 6.7\text{e-}4$ & $\mathbf{6.278 \pm 0.053}$ & $96.677 \pm 2.0\text{e}2$ & $7.134 \pm 1.6$ & $6.919 \pm 1.4$ \\
\addlinespace
$10^{-3}$ & 1 & $6.444 \pm 0.0069$ & $6.444 \pm 0.0069$ & $6.430 \pm 0.047$ & $27.571 \pm 39$ & $6.498 \pm 1.2$ & $6.444 \pm 1.1$ \\
  & 2 & $6.418 \pm 0.0049$ & $6.418 \pm 0.0048$ & $6.405 \pm 0.048$ & $28.364 \pm 49$ & $6.688 \pm 1.4$ & $6.612 \pm 1.2$ \\
  & 4 & $6.366 \pm 0.0026$ & $6.366 \pm 0.0028$ & $6.353 \pm 0.052$ & $15.184 \pm 13$ & $6.726 \pm 1.4$ & $6.632 \pm 1.2$ \\
\addlinespace
$10^{-2}$ & 1 & $6.653 \pm 0.025$ & $6.652 \pm 0.025$ & $6.628 \pm 0.051$ & $6.350 \pm 0.82$ & $6.387 \pm 0.50$ & $6.398 \pm 0.46$ \\
  & 2 & $6.805 \pm 0.028$ & $6.805 \pm 0.028$ & $6.792 \pm 0.049$ & $6.622 \pm 1.2$ & $6.538 \pm 0.74$ & $6.549 \pm 0.67$ \\
  & 4 & $6.772 \pm 0.019$ & $6.772 \pm 0.019$ & $6.760 \pm 0.049$ & $6.806 \pm 1.6$ & $6.549 \pm 1.0$ & $6.553 \pm 0.91$ \\
\addlinespace
$10^{-1}$ & 1 & $6.615 \pm 0.12$ & $6.405 \pm 0.15$ & $6.293 \pm 0.20$ & $6.349 \pm 0.24$ & $6.310 \pm 0.21$ & $6.313 \pm 0.21$ \\
  & 2 & $6.869 \pm 0.079$ & $6.867 \pm 0.078$ & $6.814 \pm 0.082$ & $6.687 \pm 0.10$ & $6.708 \pm 0.084$ & $6.709 \pm 0.083$ \\
  & 4 & $6.931 \pm 0.069$ & $6.931 \pm 0.069$ & $6.892 \pm 0.085$ & $6.764 \pm 0.24$ & $6.756 \pm 0.17$ & $6.766 \pm 0.16$ \\
\addlinespace
$1$ & 1 & $6.916 \pm 0.93$ & $6.920 \pm 0.94$ & $6.922 \pm 0.95$ & $6.904 \pm 1.1$ & $6.896 \pm 1.1$ & $6.896 \pm 1.1$ \\
  & 2 & $6.824 \pm 0.41$ & $6.819 \pm 0.41$ & $6.825 \pm 0.41$ & $6.809 \pm 0.41$ & $6.770 \pm 0.38$ & $6.779 \pm 0.37$ \\
  & 4 & $6.861 \pm 0.35$ & $6.859 \pm 0.35$ & $6.845 \pm 0.36$ & $6.841 \pm 0.31$ & $6.852 \pm 0.36$ & $6.856 \pm 0.36$ \\
\bottomrule
\end{tabular}
}

\parbox{\textwidth}{
\textsuperscript{*} ``lr'' and ``bs'' refer to the learning rate and batch size. \\
\textsuperscript{**} ``NaN'' indicates a numerical overflow has occurred. 
}

\end{table}

\begin{table}[htbp]
\centering
\caption{The averaged test errors generated by SGD under different parameters and initialization strategies on the Volatility of S\&P Index dataset. The number of epochs and repetitions are 100 and 3, respectively.
Means and standard deviations are approximated as 4 decimal places and 2 significant figures, respectively. Excessively large means are shown as 2 significant figures.}
\label{tab:SGD-tune-2}

\resizebox{\textwidth}{!}{%
\begin{tabular}{l l l l l l l l}
\toprule
\textbf{lr\textsuperscript{*}} & \textbf{bs\textsuperscript{*}} & $\mathcal{N}(0, 10^{-3})$ & $\mathcal{N}(0, 10^{-2})$ & $\mathcal{N}(0, 10^{-1})$ & He & Glorot & LeCun \\ 
\midrule
$10^{-4}$ & 25 & $1.6564 \pm 0.0057$ & $1.6531 \pm 0.0059$ & $1.3446 \pm 0.16$ & $1.0\text{e}4 \pm 1.4\text{e}4$ & $1.3674 \pm 1.4$ & $1.4082 \pm 1.3$ \\
  & 50 & $1.6163 \pm 0.0026$ & $1.6149 \pm 0.0028$ & $1.4764 \pm 0.13$ & $3.5643 \pm 3.3$ & $2.0467 \pm 1.9$ & $1.9730 \pm 1.8$ \\
  & 100 & $1.5952 \pm 0.0029$ & $1.5947 \pm 0.0031$ & $1.5441 \pm 0.13$ & $5.3961 \pm 3.9$ & $2.6918 \pm 2.2$ & $2.4164 \pm 2.0$ \\
\addlinespace
$10^{-3}$ & 25 & $1.7480 \pm 0.022$ & $1.1011 \pm 0.086$ & $0.1062 \pm 0.055$ & $2.1\text{e}14 \pm 3.0\text{e}14$ & $0.2669 \pm 0.18$ & $0.2184 \pm 0.13$ \\
  & 50 & $1.7440 \pm 0.012$ & $1.6992 \pm 0.018$ & $0.3311 \pm 0.22$ & $0.7875 \pm 0.55$ & $0.4688 \pm 0.38$ & $0.4534 \pm 0.38$ \\
  & 100 & $1.7223 \pm 0.026$ & $1.7146 \pm 0.026$ & $1.0727 \pm 0.20$ & $1.6500 \pm 1.5$ & $0.9877 \pm 1.0$ & $1.0212 \pm 1.0$ \\
\addlinespace
$10^{-2}$ & 25 & $\mathbf{0.0181 \pm 0.0011}$ & $0.0198 \pm 9.9\text{e-}4$ & $0.0404 \pm 0.0039$ & $0.0674 \pm 0.019$ & $0.0428 \pm 0.015$ & $0.0311 \pm 0.024$ \\
  & 50 & $0.0341 \pm 9.0\text{e-}4$ & $0.0315 \pm 8.0\text{e-}4$ & $0.0505 \pm 0.0038$ & \text{NaN}\textsuperscript{**} & $0.0730 \pm 0.0096$ & $0.0479 \pm 0.015$ \\
  & 100 & $0.9463 \pm 0.13$ & $0.1335 \pm 0.053$ & $0.0654 \pm 0.018$ & \text{NaN}\textsuperscript{**} & $0.1653 \pm 0.091$ & $0.1201 \pm 0.047$ \\
\addlinespace
$10^{-1}$ & 25 & $0.0280 \pm 0.0013$ & $0.0286 \pm 9.8\text{e-}4$ & $0.0309 \pm 0.0044$ & \text{NaN}\textsuperscript{**} & $0.0284 \pm 0.0082$ & $0.025 \pm 0.0068$ \\
  & 50 & $0.0256 \pm 0.0018$ & $0.0268 \pm 0.0021$ & $0.0341 \pm 0.0022$ & \text{NaN}\textsuperscript{**} & $0.0486 \pm 0.018$ & $0.0373 \pm 0.025$ \\
  & 100 & $0.0190 \pm 0.0031$ & $0.0231 \pm 0.0029$ & $0.0406 \pm 0.0033$ & \text{NaN}\textsuperscript{**} & $0.0528 \pm 0.033$ & $0.0340 \pm 0.033$ \\
\addlinespace
$1$ & Any\textsuperscript{***} & \text{NaN}\textsuperscript{**} & \text{NaN}\textsuperscript{**} & \text{NaN}\textsuperscript{**} & \text{NaN}\textsuperscript{**} & \text{NaN}\textsuperscript{**} & \text{NaN}\textsuperscript{**} \\
\bottomrule
\end{tabular}
}

\parbox{\textwidth}{ 
\textsuperscript{*} ``lr'' and ``bs'' refer to the learning rate and batch size. \\
\textsuperscript{**} ``NaN'' indicates a numerical overflow has occurred. \\
\textsuperscript{***} ``Any'' refers to any batch size in $\{25,50,100\}$.
}

\end{table}

\begin{table}[htbp]
\centering
\caption{The averaged test errors generated by SGD under different parameters and initialization strategies on the TIMIT dataset. The number of epochs and repetitions are 100 and 3, respectively.
Means and standard deviations are approximated as 3 decimal places and 2 significant figures, respectively.}
\label{tab:SGD-tune-3}

\small 

\begin{tabular}{l l l l l}
\toprule
\textbf{lr\textsuperscript{*}} & \textbf{bs\textsuperscript{*}} & $\mathcal{N}(0, 10^{-3})$ & $\mathcal{N}(0, 10^{-2})$ & $\mathcal{N}(0, 10^{-1})$ \\ 
\midrule
$10^{-4}$ & 4 & $5.924 \pm 3.2\text{e-}6$ & $5.925 \pm 3.1\text{e-}4$ & $6.573 \pm 0.30$ \\
  & 8 & $5.925 \pm 3.2\text{e-}6$ & $5.926 \pm 3.1\text{e-}4$ & $6.647 \pm 0.36$ \\
  & 16 & $5.926 \pm 3.2\text{e-}6$ & $5.927 \pm 3.2\text{e-}4$ & $6.688 \pm 0.39$ \\
\addlinespace
$10^{-3}$ & 4 & $5.900 \pm 2.7\text{e-}6$ & $5.900 \pm 2.5\text{e-}4$ & $6.095 \pm 0.051$ \\
  & 8 & $5.912 \pm 2.9\text{e-}6$ & $5.912 \pm 2.8\text{e-}4$ & $6.237 \pm 0.093$ \\
  & 16 & $5.918 \pm 3.0\text{e-}6$ & $5.919 \pm 3.0\text{e-}4$ & $6.387 \pm 0.17$ \\
\addlinespace
$10^{-2}$ & 4 & $5.692 \pm 4.5\text{e-}6$ & $5.692 \pm 3.2\text{e-}4$ & $5.710 \pm 0.0072$ \\
  & 8 & $5.788 \pm 2.6\text{e-}6$ & $5.788 \pm 2.0\text{e-}4$ & $5.824 \pm 0.010$ \\
  & 16 & $5.842 \pm 2.3\text{e-}6$ & $5.842 \pm 2.0\text{e-}4$ & $5.912 \pm 0.019$ \\
\addlinespace
$10^{-1}$ & 4 & $5.087 \pm 0.0018$ & $5.069 \pm 0.025$ & $5.089 \pm 0.0031$ \\
  & 8 & $5.211 \pm 2.5\text{e-}4$ & $5.203 \pm 0.0097$ & $5.213 \pm 0.0044$ \\
  & 16 & $5.367 \pm 4.1\text{e-}5$ & $5.365 \pm 0.0024$ & $5.371 \pm 0.0061$ \\
\addlinespace
$1$ & 4 & $4.376 \pm 0.072$ & $\mathbf{4.304 \pm 0.0025}$ & $4.909 \pm 0.15$ \\
  & 8 & $4.892 \pm 0.10$ & $4.612 \pm 0.11$ & $4.994 \pm 0.015$ \\
  & 16 & $4.981 \pm 0.012$ & $4.900 \pm 0.066$ & $4.988 \pm 0.0027$ \\
\bottomrule
\end{tabular}

\parbox{\textwidth}{ 
\textsuperscript{*} ``lr'' and ``bs'' refer to the learning rate and batch size.  
}

\end{table}

\begin{table}[htbp]
\centering
\caption{The averaged test errors generated by Adam under different parameters and initialization strategies on the Synthetic dataset. The number of epochs and repetitions are 50 and 10, respectively.
Means and standard deviations are approximated as 3 decimal places and 2 significant figures, respectively.}
\label{tab:Adam-tune-1}

\resizebox{\textwidth}{!}{%
\begin{tabular}{l l l l l l l l}
\toprule
\textbf{lr\textsuperscript{*}} & \textbf{bs\textsuperscript{*}} & $\mathcal{N}(0, 10^{-3})$ & $\mathcal{N}(0, 10^{-2})$ & $\mathcal{N}(0, 10^{-1})$ & He & Glorot & LeCun \\ 
\midrule
$10^{-4}$ & 1 & $6.378 \pm 0.0031$ & $6.378 \pm 0.0032$ & $6.364 \pm 0.040$ & $58.117 \pm 1.1\text{e}2$ & $6.661 \pm 1.3$ & $6.547 \pm 1.1$ \\
  & 2 & $6.352 \pm 0.0020$ & $6.352 \pm 0.0021$ & $6.342 \pm 0.040$ & $88.731 \pm 1.8\text{e}2$ & $6.887 \pm 1.5$ & $6.729 \pm 1.3$ \\
  & 4 & $6.320 \pm 6.0\text{e-}4$ & $6.320 \pm 5.8\text{e-}4$ & $6.309 \pm 0.046$ & $113.413 \pm 2.3\text{e}2$ & $7.007 \pm 1.5$ & $6.815 \pm 1.4$ \\
\addlinespace
$10^{-3}$ & 1 & $6.529 \pm 0.038$ & $6.507 \pm 0.047$ & $6.544 \pm 0.055$ & $6.557 \pm 1.1$ & $6.454 \pm 0.44$ & $6.469 \pm 0.39$ \\
  & 2 & $6.701 \pm 0.017$ & $6.707 \pm 0.017$ & $6.719 \pm 0.025$ & $8.357 \pm 2.7$ & $6.574 \pm 0.80$ & $6.579 \pm 0.70$ \\
  & 4 & $6.617 \pm 0.0080$ & $6.616 \pm 0.0078$ & $6.607 \pm 0.015$ & $18.354 \pm 21$ & $6.591 \pm 1.1$ & $6.557 \pm 0.96$ \\
\addlinespace
$10^{-2}$ & 1 & $6.334 \pm 0.20$ & $6.267 \pm 0.20$ & $6.221 \pm 0.12$ & $6.302 \pm 0.19$ & $6.335 \pm 0.20$ & $6.316 \pm 0.23$ \\
  & 2 & $6.705 \pm 0.068$ & $6.677 \pm 0.065$ & $6.677 \pm 0.092$ & $6.638 \pm 0.17$ & $6.651 \pm 0.094$ & $6.648 \pm 0.095$ \\
  & 4 & $6.747 \pm 0.067$ & $6.745 \pm 0.062$ & $6.745 \pm 0.064$ & $6.818 \pm 0.67$ & $6.742 \pm 0.17$ & $6.746 \pm 0.14$ \\
\addlinespace
$10^{-1}$ & 1 & $6.120 \pm 0.24$ & $\mathbf{6.065 \pm 0.24}$ & $6.104 \pm 0.23$ & $6.083 \pm 0.27$ & $6.134 \pm 0.25$ & $6.126 \pm 0.29$ \\
  & 2 & $6.739 \pm 0.14$ & $6.740 \pm 0.15$ & $6.721 \pm 0.16$ & $6.714 \pm 0.12$ & $6.742 \pm 0.16$ & $6.737 \pm 0.22$ \\
  & 4 & $6.761 \pm 0.21$ & $6.767 \pm 0.22$ & $6.804 \pm 0.16$ & $6.855 \pm 0.13$ & $6.826 \pm 0.19$ & $6.837 \pm 0.17$ \\
\addlinespace
$1$ & 1 & $6.540 \pm 0.68$ & $6.740 \pm 0.36$ & $6.938 \pm 0.50$ & $6.789 \pm 0.74$ & $6.613 \pm 0.54$ & $6.610 \pm 0.62$ \\
  & 2 & $6.884 \pm 0.39$ & $6.737 \pm 0.40$ & $6.868 \pm 0.57$ & $6.767 \pm 0.37$ & $6.888 \pm 0.62$ & $6.928 \pm 0.51$ \\
  & 4 & $6.796 \pm 0.29$ & $7.000 \pm 0.46$ & $6.992 \pm 0.88$ & $6.787 \pm 0.63$ & $6.926 \pm 0.18$ & $6.826 \pm 0.20$ \\
\bottomrule
\end{tabular}
}

\parbox{\textwidth}{ 
\textsuperscript{*} ``lr'' and ``bs'' refer to the learning rate and batch size.  
}

\end{table}

\begin{table}[htbp]
\centering
\caption{The averaged test errors generated by Adam under different parameters and initialization strategies on the Volatility of S\&P Index dataset. The number of epochs and repetitions are 100 and 3, respectively.
Means and standard deviations are approximated as 4 decimal places and 2 significant figures, respectively. Excessively large means are shown as 2 significant figures.}
\label{tab:Adam-tune-2}

\resizebox{\textwidth}{!}{%
\begin{tabular}{l l l l l l l l}
\toprule
\textbf{lr\textsuperscript{*}} & \textbf{bs\textsuperscript{*}} & $\mathcal{N}(0, 10^{-3})$ & $\mathcal{N}(0, 10^{-2})$ & $\mathcal{N}(0, 10^{-1})$ & He & Glorot & LeCun \\ 
\midrule
$10^{-4}$ & 25 & $\mathbf{0.0092 \pm 0.0018}$ & $0.0124 \pm 0.0040$ & $0.0291 \pm 0.0093$ & $1.3958 \pm 0.25$ & $0.1290 \pm 0.10$ & $0.1148 \pm 0.093$ \\
  & 50 & $0.0520 \pm 0.050$ & $0.0579 \pm 0.058$ & $0.0857 \pm 0.033$ & $5.4409 \pm 3.2$ & $0.3801 \pm 0.22$ & $0.3909 \pm 0.25$ \\
  & 100 & $0.7278 \pm 0.51$ & $0.8798 \pm 0.51$ & $0.7811 \pm 0.40$ & $11.3302 \pm 8.3$ & $1.4518 \pm 1.2$ & $1.4038 \pm 1.2$ \\
\addlinespace
$10^{-3}$ & 25 & $0.0454 \pm 0.021$ & $0.0434 \pm 0.018$ & $0.0350 \pm 0.0084$ & $0.1106 \pm 0.074$ & $0.0270 \pm 0.018$ & $0.0206 \pm 0.014$ \\
  & 50 & $0.0590 \pm 0.033$ & $0.0597 \pm 0.031$ & $0.0381 \pm 0.012$ & $0.4460 \pm 0.25$ & $0.0511 \pm 0.050$ & $0.0397 \pm 0.036$ \\
  & 100 & $0.0172 \pm 0.010$ & $0.0213 \pm 0.015$ & $0.0353 \pm 0.0071$ & $1.2725 \pm 0.48$ & $0.0741 \pm 0.064$ & $0.0698 \pm 0.062$ \\
\addlinespace
$10^{-2}$ & 25 & $0.0351 \pm 0.0026$ & $0.0333 \pm 0.0061$ & $0.0427 \pm 0.0096$ & $0.0408 \pm 0.0056$ & $0.0481 \pm 0.0062$ & $0.0485 \pm 0.0081$ \\
  & 50 & $0.0552 \pm 0.0072$ & $0.0448 \pm 0.0063$ & $0.0515 \pm 0.0048$ & $0.0397 \pm 0.0060$ & $0.0496 \pm 0.0050$ & $0.0527 \pm 0.0071$ \\
  & 100 & $0.0195 \pm 0.011$ & $0.0294 \pm 0.0098$ & $0.0322 \pm 0.0084$ & $0.0759 \pm 0.050$ & $0.0290 \pm 0.019$ & $0.0490 \pm 0.023$ \\
\addlinespace
$10^{-1}$ & 25 & $0.0357 \pm 0.0082$ & $2.2\text{e}4 \pm 3.2\text{e}4$ & $0.0500 \pm 0.015$ & $0.0407 \pm 0.0073$ & $9.7\text{e}3 \pm 1.4\text{e}4$ & $0.0543 \pm 0.028$ \\
  & 50 & \text{NaN}\textsuperscript{**} & $0.0574 \pm 0.024$ & \text{NaN}\textsuperscript{**} & $0.0333 \pm 0.013$ & $0.0409 \pm 0.0081$ & $0.0371 \pm 0.010$ \\
  & 100 & $0.0321 \pm 0.0069$ & \text{NaN}\textsuperscript{**} & \text{NaN}\textsuperscript{**} & $0.0256 \pm 0.010$ & $0.0207 \pm 0.0072$ & $0.0582 \pm 0.0093$ \\
\addlinespace
$1$ & Any\textsuperscript{***} & \text{NaN}\textsuperscript{**} & \text{NaN}\textsuperscript{**} & \text{NaN}\textsuperscript{**} & \text{NaN}\textsuperscript{**} & \text{NaN}\textsuperscript{**} & \text{NaN}\textsuperscript{**} \\
\bottomrule
\end{tabular}
}

\parbox{\textwidth}{ 
\textsuperscript{*} ``lr'' and ``bs'' refer to the learning rate and batch size. \\
\textsuperscript{**} ``NaN'' indicates a numerical overflow has occurred. \\
\textsuperscript{***} ``Any'' refers to any batch size in $\{25,50,100\}$.
}

\end{table}

\begin{table}[htbp]
\centering
\caption{The averaged test errors generated by Adam under different parameters and initialization strategies on the TIMIT dataset. The number of epochs and repetitions are 100 and 3, respectively.
Means and standard deviations are approximated as 3 decimal places and 2 significant figures, respectively.}
\label{tab:Adam-tune-3}

\small 

\begin{tabular}{l l l l l}
\toprule
\textbf{lr\textsuperscript{*}} & \textbf{bs\textsuperscript{*}} & $\mathcal{N}(0, 10^{-3})$ & $\mathcal{N}(0, 10^{-2})$ & $\mathcal{N}(0, 10^{-1})$ \\ 
\midrule
$10^{-4}$ & 4 & $4.918 \pm 0.017$ & $4.887 \pm 0.0074$ & $5.005 \pm 0.19$ \\
  & 8 & $4.956 \pm 0.017$ & $4.927 \pm 0.0076$ & $5.279 \pm 0.15$ \\
  & 16 & $5.033 \pm 0.043$ & $4.995 \pm 0.0031$ & $5.570 \pm 0.053$ \\
\addlinespace
$10^{-3}$ & 4 & $4.232 \pm 0.059$ & $4.208 \pm 0.077$ & $4.317 \pm 0.030$ \\
  & 8 & $4.245 \pm 0.046$ & $\mathbf{4.201 \pm 0.054}$ & $4.277 \pm 0.044$ \\
  & 16 & $4.318 \pm 0.079$ & $4.237 \pm 0.084$ & $4.258 \pm 0.038$ \\
\addlinespace
$10^{-2}$ & 4 & $4.507 \pm 0.37$ & $4.296 \pm 0.091$ & $4.557 \pm 0.32$ \\
  & 8 & $4.450 \pm 0.42$ & $4.505 \pm 0.36$ & $4.371 \pm 0.043$ \\
  & 16 & $4.469 \pm 0.38$ & $4.472 \pm 0.38$ & $4.784 \pm 0.29$ \\
\addlinespace
$10^{-1}$ & 4 & $4.991 \pm 0.0016$ & $4.992 \pm 0.0024$ & $4.991 \pm 0.0029$ \\
  & 8 & $4.997 \pm 0.0013$ & $4.997 \pm 1.6\text{e-}4$ & $4.998 \pm 3.9\text{e-}4$ \\
  & 16 & $4.978 \pm 3.6\text{e-}4$ & $4.978 \pm 7.2\text{e-}4$ & $4.978 \pm 4.7\text{e-}4$ \\
\addlinespace
$1$ & 4 & $4.990 \pm 0.0088$ & $5.080 \pm 0.049$ & $5.051 \pm 0.025$ \\
  & 8 & $5.000 \pm 0.0027$ & $5.004 \pm 0.0064$ & $5.022 \pm 0.0064$ \\
  & 16 & $4.953 \pm 0.026$ & $4.971 \pm 0.024$ & $4.932 \pm 0.0049$ \\
\bottomrule
\end{tabular}

\parbox{\textwidth}{ 
\textsuperscript{*} ``lr'' and ``bs'' refer to the learning rate and batch size.  
}

\end{table}

\end{document}